\documentclass[11pt]{amsart}
\newcommand{\preprint}[1]{}

\newcommand{\hide}[1]{}

\usepackage{amssymb}
\usepackage{amsbsy}
\usepackage{amscd}
\usepackage{amsmath}
\usepackage{amsthm}
\usepackage{geometry}
\usepackage{MnSymbol}
\input xy
\xyoption{all}

\numberwithin{equation}{section}

\theoremstyle{plain}
\newtheorem{thm}{Theorem}[section]
\newtheorem{assumption}[thm]{Assumption}
\newtheorem{prop}[thm]{Proposition}
\newtheorem{claim}[thm]{Claim}

\newtheorem{cor}[thm]{Corollary}
\newtheorem{lem}[thm]{Lemma}
\newtheorem{observation}[thm]{Observation}

\theoremstyle{definition}
\newtheorem{defi}[thm]{Definition}

\theoremstyle{remark}
\newtheorem{example}[thm]{Example}
\newtheorem{question}[thm]{Question}
\newtheorem{rem}[thm]{Remark}

\newcommand{\A}{{\mathcal A}}
\newcommand{\B}{{\mathcal B}}

\newcommand{\C}{{\mathcal C}}

\newcommand{\F}{{\mathcal F}}
\newcommand{\FE}{{\mathcal F}{\mathcal E}}

\newcommand{\HH}{{\mathbb H}}

\newcommand{\K}{{\mathcal K}}

\newcommand{\LB}{{\mathcal L}}
\newcommand{\M}{{\mathcal M}}
\newcommand{\FM}{{\mathfrak M}}
\newcommand{\MV}{{\mathcal M}{\mathcal V}}

\newcommand{\PP}{{\mathbb P}}

\newcommand{\X}{{\mathcal X}}

\newcommand{\RealNumbers}{{\mathbb R}}
\newcommand{\RR}{{\mathbb R}}
\newcommand{\Integers}{{\mathbb Z}}
\newcommand{\ComplexNumbers}{{\mathbb C}}
\newcommand{\RationalNumbers}{{\mathbb Q}}

\newcommand{\linsys}[1]{{\mid}#1{\mid}}

\newcommand{\RightArrowOf}[1]{\stackrel{#1}{\rightarrow}}

\newcommand{\LongRightArrowOf}[1]{\stackrel{#1}{\longrightarrow}}

\newcommand{\StructureSheaf}[1]{{\mathcal O}_{#1}}

\newcommand{\rank}{{\rm rank}}

\renewcommand{\div}{{\rm div}}
\newcommand{\Pic}{{\rm Pic}}
\newcommand{\Pex}{{\mathcal P}ex}

\newcommand{\Hom}{{\rm Hom}}
\newcommand{\Aut}{{\rm Aut}}

\newcommand{\Abs}[1]{\left|\!#1\!\right|}

\newcommand{\Wedge}[1]{{\wedge}^{#1}}

\newcommand{\uell}{\underline{\ell}}
\newcommand{\tlambda}{\tilde{\lambda}}

\usepackage[OT2,T1]{fontenc}
\DeclareSymbolFont{cyrletters}{OT2}{wncyr}{m}{n}
\DeclareMathSymbol{\Sha}{\mathalpha}{cyrletters}{"58}

\begin{document}
\title[Lagrangian fibrations]
{Lagrangian fibrations 
%and the semigroup of effective divisor classes 
of holomorphic-symplectic varieties of $K3^{[n]}$-type}
\author{Eyal Markman}
\address{Department of mathematics and statistics, University of Massachusetts, Amherst MA 01003}
\email{markman@math.umass.edu}
\thanks{Partially supported by Simons Foundation Collaboration Grant 245840 and  by NSA grant H98230-13-1-0239.}
\date{\today}

\begin{abstract}
Let $X$ be a compact K\"{a}hler holomorphic-symplectic manifold, 
which is deformation equivalent to the Hilbert scheme of 
length $n$ subschemes of a $K3$ surface. 
Let $\LB$ be a nef line-bundle on $X$, such that 
the top power $c_1(\LB)^{2n}$ vanishes and $c_1(\LB)$  is primitive.
Assume that the two dimensional subspace $H^{2,0}(X)\oplus H^{0,2}(X)$ 
of $H^2(X,\ComplexNumbers)$ intersects $H^2(X,\Integers)$ trivially. 
We prove that the linear system of 
$\LB$ is base point free and it induces a Lagrangian fibration on $X$. 
In particular, the line-bundle $\LB$ is effective.
A determination of the 
semi-group of effective divisor classes on $X$ follows, when $X$ is projective. 
For a generic such pair $(X,\LB)$, not necessarily projective, 
we show that $X$ is bimeromorphic to a
Tate-Shafarevich twist of a moduli space of stable torsion sheaves, each with pure one dimensional support, 
on a {\em projective} $K3$ surface. 
\end{abstract}

\maketitle

\begin{center}
{\em Dedicated to Klaus Hulek on the occasion of his sixtieth birthday.}
\end{center}

\tableofcontents

%*************************************************************************************************
%
%*************************************************************************************************
\section{Introduction}
An {\em irreducible holomorphic symplectic manifold} is a simply connected compact K\"{a}hler manifold
such that $H^0(X,\Wedge{2}T^*X)$ is generated by an everywhere non-degenerate holomorphic $2$-form \cite{beauville}. 
A compact K\"{a}hler manifold $X$ is said to be of {\em $K3^{[n]}$-type}, if it is deformation equivalent to
the Hilbert scheme $S^{[n]}$ of length $n$ subschemes of a $K3$ surface $S$. 
Any manifold of $K3^{[n]}$-type is irreducible holomorphic symplectic \cite{beauville}.
%Example \ref{example-alpha-sum-of-stably-prime-exceptional} 
%exhibits an effective isotropic class, for each monodromy orbit. It is not clear at the moment if the 
%class remains effective under small deformations of Hodge-type $(1,1)$.
The second integral cohomology of an irreducible holomorphic symplectic manifold $X$
admits a natural symmetric non-degenerate integral bilinear pairing $(\bullet,\bullet)$
of signature $(3,b_2(X)-3)$, called the
{\em Beauville-Bogomolov-Fujiki pairing}. The Beauville-Bogomolov-Fujiki pairing is monodromy invariant, and is thus 
an invariant of the deformation class of $X$. 

%Let $X$ be an irreducible holomorphic symplectic manifold of $K3^{[n]}$-type, $n\geq 2$.
%Set $\Lambda:=H^2(X,\Integers)$. 

\begin{defi}
\label{def-special}
An irreducible holomorphic symplectic manifold $X$
is said to be {\em special},
if the intersection in $H^2(X,\ComplexNumbers)$ of $H^2(X,\Integers)$ and $H^{2,0}(X)\oplus H^{0,2}(X)$ 
is a non-zero subgroup. 
\end{defi}

The locus of special periods forms a countable union of real analytic subvarieties of half the
dimension in the corresponding moduli space.

\begin{defi}
\label{def-LB-induces-a-Lagrangian-fibration}
Let $X$ be a $2n$-dimensional irreducible holomorphic symplectic manifold and $\LB$ a line bundle on $X$.
We say that $\LB$ {\em induces a Lagrangian fibration,} if it satisfies the following two conditions.
\begin{enumerate}
\item
\label{cond-higher-cohomologies-of-LB-vanish}
$h^0\left(X,\LB\right)=n+1$.
%and $h^i\left(X,\LB\right)=0$, for
%$i>1$. 
\item
\label{cond-LB-induces-a-Lagrangian-fibration}
The linear system $\linsys{\LB}$ is base point free, and the generic fiber of the
morphism $\pi \colon  X\rightarrow \linsys{\LB}^*$
is a connected Lagrangian subvariety.
\end{enumerate}
\end{defi}

A line bundle $\LB$ on a holomorphic symplectic manifold $X$ is said to be {\em nef}, 
if $c_1(\LB)$ belongs to the closure in $H^{1,1}(X,\RealNumbers)$ of the K\"{a}hler cone of $X$.

\begin{thm}
\label{thm-main}
Let $X$ be an  irreducible holomorphic symplectic manifold 
of $K3^{[n]}$-type and $\LB$ a nef line-bundle, such that $c_1(\LB)$ is primitive
and isotropic with respect to the Beauville-Bogomolov-Fujiki pairing.
%and $(\kappa,c_1(\LB))>0$, for some K\"{a}hler class $\kappa$ on $X$.
Assume that $X$ is non-special.
Then  the line bundle $\LB$ 
induces a Lagrangian fibration $\pi:X\rightarrow \linsys{\LB}^*$. 
\end{thm}

See Theorem \ref{thm-case-X-projective-L-isotropic-but-not-nef} for a variant of Theorem \ref{thm-main}
dropping the assumption that $\LB$ is nef.
%We expect that Theorem \ref{thm-main} holds for non-algebraic $X$ as well.
%We can only prove the following statement for non-algebraic $X$.
%
%\begin{thm}
%\label{thm-main-non-projective}
%Let $X$ be an irreducible holomorphic symplectic manifold 
%of $K3^{[n]}$-type and $\LB$ a  line-bundle, such that $c_1(\LB)$ is primitive
%and isotropic with respect to the Beauville-Bogomolov-Fujiki pairing
%and $(\kappa,c_1(\LB))>0$, for some K\"{a}hler class $\kappa$ on $X$.
%Assume that the weight $2$ Hodge structure of $X$ is non-special.
%Then there exists an  irreducible holomorphic symplectic manifold $Y$ of $K3^{[n]}$-type and 
%an analytic correspondence $Z=\sum_{i=0}^kZ_i$ in $X\times Y$, of pure dimension $2n$, 
%with the following properties.
%\begin{enumerate}
%\item
%$Z_*(c_1(\LB))$
%is the class of a line-bundle on $Y$, which induces a Lagrangian fibration. 
%\item
%The irreducible
%component 
%$Z_0$ of the correspondence is the graph of 
%a bimeromorphic map $f:X\rightarrow Y$.
%\item
%The images in $X$ and $Y$ of all  other components  $Z_i$, $i>0$, 
%are of positive co-dimension. If the image is a divisor, then it is uniruled.
%%\item
%%The homomorphism $Z_*:H^*(X,\Integers)\rightarrow H^*(Y,\Integers)$ is a parallel-transport operator.
%\end{enumerate}
%\end{thm}
%
Theorem \ref{thm-main} is proven in section \ref{sec-proof-of-main-thm}.
The proof relies on Verbitsky's Global Torelli Theorem 
\cite{verbitsky,huybrechts-bourbaki}, on the determination of the 
monodromy group of $X$ \cite{markman-monodromy-I,markman-constraints}, and
on a result of Matsushita  that Lagrangian fibrations form an {\em open} 
subset in the moduli space of pairs $(X,\LB)$ \cite{matsushita}. 
Let us sketch the three main new ingredients in the proof of Theorem \nolinebreak\ref{thm-main}.
\begin{enumerate}
\item
We associate to the pair $(X,\LB)$ in Theorem \ref{thm-main} a projective $K3$ surface $S$
with a nef line bundle $\B$ of degree $\frac{2n-2}{d^2}$, 
where $d:=\gcd\{(c_1(\LB),\lambda) \ : \ \lambda\in H^2(X,\Integers)\}$. 
The sub-lattice $c_1(\B)^\perp$ orthogonal to $c_1(\B)$ in $H^2(S,\Integers)$
is Hodge-isometric to $c_1(\LB)^\perp/\Integers c_1(\LB)$.
The construction realizes the period domain $\Omega_{20}$
of the pairs $(X,\LB)$
as an affine line bundle over a period domain $\Omega_{19}$ of semi-polarized $K3$ surfaces
(Section \ref{sec-period-domains}).
\item
The bundle map $q:\Omega_{20}\rightarrow \Omega_{19}$ 
is invariant with respect to a subgroup $Q$ of the monodromy group  
(Lemma \ref{lemma-q-is-Q-alpha-invariant}). The group $Q$ is isomorphic to 
$c_1(\B)^\perp$. $Q$ acts on the fiber of $q$ over the period of a semi-polarized $K3$ surface $(S,\B)$.
Similarly, the lattice 
$c_1(\B)^\perp$ projects to a subgroup of $H^{0,2}(S)$, which acts on $H^{0,2}(S)$ by translations. 
There exists an isomorphism, of the fiber of $q$ with $H^{0,2}(S)$, which is equivariant with respect to the two actions
(Lemma \ref{lemma-density-in-a-generic-fiber}).
%The action of $Q$, on the fiber of $q$ over the period of a semi-polarized $K3$ surface $(S,\B)$, 
%factors through the following action of $c_1(\B)^\perp$. The lattice 
%$c_1(\B)^\perp$ projects to a subgroup of $H^{0,2}(S)$, which acts on $H^{0,2}(S)$ by translations. 
\item
The fiber of $q$ over the period of a semi-polarized $K3$ surface $(S,\B)$ contains the
period of a moduli space of sheaves on  $S$ with pure
one-dimensional support in the linear system $\linsys{\B^d}$ (Section \ref{sec-the-section-tau-of-q}). 
Each such moduli space of sheaves is known to be a Lagrangian fibration \cite{mukai-symplectic-structure}.
\end{enumerate}

%The condition that the Hodge structure is non-special is probably not necessary.
The assumption that $X$ is non-special in Theorem \ref{thm-main}  is probably not necessary.
Unfortunately, our proof will rely on it. When $X$ is non-special the
$Q$-orbit, of every point in the fiber of $q$ through the period of $X$, is a dense subset of the fiber
(Lemma \ref{lemma-density-in-a-generic-fiber}). This density will have a central role in this paper
due to the following elementary observation.
\begin{observation}
Let $T$ be a topological space and $Q$ a group acting on $T$.
Assume that the $Q$-orbit of every point of $T$ is dense in $T$. 
Then any nonempty $Q$-invariant open subset of $T$ must be the whole of $T$.
\end{observation}
The above observation
will be used in an essential way in three different proofs 
(Theorem \ref{thm-dimensions-of-h-i}, 
Proposition \ref{prop-family-over-Sha-is-Kahler},
and Theorem \ref{thm-Tate-Shafarevich-orbit}). 
%**************
% Hide
%**************
\hide{
The locus $V$ of periods of Lagrangian fibrations is thus a $Q$-invariant open subset 
in the period domain of pairs $(X,\LB)$, and $V$ intersects every fiber of $q$.
2) 
%The second new ingredient:
The $Q$-orbit of every non-special point of the period domain is {\em dense} in its $q$-fiber, and thus
intersects $V$ (Prop. \ref{prop-density-of-periods-of-Lagrangian-fibrations}). 
We conclude that the complement of $V$, which is $Q$-invariant, can not contain any 
non-special period.
%**************
% End Hide
%**************
}

The statement of the next result requires 
the notion of a Tate-Shafarevich twist, which we now recall.
Let $M$ be a complex manifold and $\pi:M\rightarrow B$ a proper map with connected 
fibers of pure dimension $n$.
Assume that the generic fiber of $\pi$ is a 
smooth abelian variety. 
Let $\{U_i\}$ be an open covering of $B$ in the analytic topology.
Set $U_{ij}:=U_i\cap U_j$ and $M_{ij}:=\pi^{-1}(U_{ij})$. Assume given a $1$-co-cycle $g_{ij}$ of
automorphisms  of $M_{ij}$, satisfying $\pi\circ g_{ij}=\pi$, and 
acting by translations on the smooth fibers of $\pi$. 
We can re-glue the open covering $\{M_i\}$ of $M$
using the co-cycle $\{g_{ij}\}$ to get a complex manifold $M'$ and a proper map $\pi':M'\rightarrow B$,
whose fibers are isomorphic to those of $\pi$. We refer to  $(M',\pi')$ as the {\em Tate-Shafarevich 
twist} of $(M,\pi)$ associated to the co-cycle $\{g_{ij}\}$. Tate-Shafarevich twists are standard in the study
of elliptic fibrations
\cite{kodaira,dolgachev-gross}.

%Let $(S,\LB)$ be the semi-polarized $K3$ surface 
%associated to $(X,\pi^*\StructureSheaf{\PP^n}(1))$ in section \ref{sec-S-alpha}. 
Let $\LB$ be a semi-ample line bundle on a $K3$ surface $S$ with an indivisible class $c_1(\LB)$.
Given an ample line bundle $H$ 
on $S$ and an integer $\chi$, 
denote by $M_H(0,\LB^d,\chi)$ the moduli space of $H$-stable coherent sheaves on $S$
of rank zero, determinant $\LB^d$, and Euler characteristic  $\chi$. 
Assume that $d$ and $\chi$ are relatively prime. For a generic polarization $H$,
the moduli space  $M_H(0,\LB^d,\chi)$ is smooth and projective and it
admits a Lagrangian fibration over the linear system $\linsys{\LB^d}$ \cite{mukai-symplectic-structure}. 
%Let $\Sigma\subset \linsys{\LB^d}$ be the locus of non-integral divisors.

Let $X$ be an irreducible holomorphic symplectic manifold of $K3^{[n]}$-type and
$\pi:X\rightarrow \PP^n$ a Lagrangian fibration. Set $\alpha:=\pi^*c_1(\StructureSheaf{\PP^n}(1))$.
The  {\em divisibility} of $(\alpha,\bullet)$ is the positive integer
$d:=\gcd\{(\alpha,\lambda) \ : \ \lambda\in H^2(X,\Integers)\}$. The integer $d^2$ divides $n-1$ 
(Lemma \ref{lemma-classification}). 

\begin{thm}
\label{thm-Tate-Shafarevich-orbit-in-introduction}
%(Theorem \ref{thm-Tate-Shafarevich-orbit})
Assume that $X$ is non-special and the intersection $H^{1,1}(X,\Integers)\cap\alpha^\perp$ is $\Integers\alpha$.
There exists a $K3$ surface $S$, a semi-ample line bundle $\LB$ on $S$ 
of degree $\frac{2n-2}{d^2}$ with an indivisible class $c_1(\LB)$,
an integer $\chi$ relatively prime to $d$, and a polarization $H$ on $S$, such that 
$X$ is bimeromorphic to a Tate-Shafarevich twist of the Lagrangian fibration
$M_H(0,\LB^d,\chi)\rightarrow \linsys{\LB^d}$.
\end{thm}

Theorem \ref{thm-Tate-Shafarevich-orbit-in-introduction} is proven in section \ref{sec-tate-shafarevich}.
%as a consequence of Theorem \ref{thm-Tate-Shafarevich-orbit}. 
The semi-polarized $K3$ surface $(S,\LB)$ in  Theorem  \ref{thm-Tate-Shafarevich-orbit-in-introduction} 
is the one mentioned already above, which is associated to $(X,\alpha)$ in section \ref{sec-S-alpha}. 
The equality $H^{1,1}(X,\Integers)\cap\alpha^\perp=\Integers\alpha$ is equivalent to the statement that $\Pic(S)$
is cyclic generated by $\LB$. This condition is relaxed in Theorem \ref{thm-X-is-birational-to-a-Tate-Shafarevich-twist},
which strengthens Theorem \ref{thm-Tate-Shafarevich-orbit-in-introduction}.

A reduced and irreducible divisor on $X$ is called {\em prime exceptional}, if it has negative Beauville-Bogomolov-Fujiki degree. A divisor $D$ on $X$ is called {\em movable}, if the base locus of the linear system $\linsys{D}$
has co-dimension $\geq 2$ in $X$. The {\em movable cone} $\MV_X$ of $X$ is the cone in
$N^1(X):=H^{1,1}(X,\Integers)\otimes_\Integers\RealNumbers$ generated by classes of movable divisors.
Assume that $X$ is a projective irreducible holomorphic symplectic manifold  of $K3^{[n]}$-type and
let $h\in N^1(X)$ be an ample class. Denote by $\Pex_X\subset H^{1,1}(X,\Integers)$ 
the set of classes of prime exceptional divisors. 
The set $\Pex_X$ is determined in \cite[Theorem 1.11 and Sec. 1.5]{markman-prime-exceptional}.
The closure of the movable cone in $N^1(X)$ is 
determined as follows:
\[
\overline{\MV}_X=
\{c\in N^1(X) \ : \ 
(c,c)\geq 0, \ (c,h)\geq 0, \ \mbox{and} \ (c,e)\geq 0, \ \mbox{for all} \ e\in \Pex_X
\},
\]
by a result of Boucksom 
\cite[Prop. 5.6 and Lemma 6.22]{boucksom,markman-torelli}\footnote{Prop. 5.6 and Lemma 6.22 in the last reference \cite{markman-torelli}. The same convention will be used throughout the paper for all citations with multiple references.}.

\begin{cor}
\label{cor-effective-semi-group}
Let $X$ be a projective irreducible holomorphic symplectic manifold of $K3^{[n]}$-type.
The semi-group  of effective divisor classes on $X$ is generated by the classes of
prime exceptional divisors and integral points in the closure of the movable cone 
in $N^1(X)$.
\end{cor}

Corollary \ref{cor-effective-semi-group} was shown to follow from Theorem 
\ref{thm-main} in \cite[Paragraph following Question 10.11]{markman-torelli}.
%The Theorem is proven in section \ref{sec-proof-of-main-thm}. 

We classify the deformation types of pairs 
$(X,\LB)$, consisting of 
an irreducible holomorphic symplectic manifold $X$ of $K3^{[n]}$-type, $n\geq 2$,
and a line bundle $\LB$ on $X$ with a primitive and isotropic first Chern class, such that
$(c_1(\LB),\kappa)>0$, for some K\"{a}hler class $\kappa$. 
The following proposition is proven in section \ref{sec-period-map}, using monodromy invariants 
introduced in Lemma \ref{lemma-classification}.
%Let $\alpha\in H^2(X,\Integers)$ be a primitive isotropic class.
%The {\em divisibility} of $(\alpha,\bullet)$ is defined as the positive integer
%$\gcd\{(\alpha,\lambda) \ : \ \lambda\in H^2(X,\Integers)\}$. 
\begin{prop}
\label{prop-counting-deformation-invariants}
Let $d$ be a positive integer, such that $d^2$ divides $n-1$. If 
$1\leq d\leq 4$, then there 
exists a unique deformation type of pairs $(X,\LB)$, with $c_1(\LB)$ primitive and isotropic, such that
$(c_1(\LB),\bullet)$ has divisibility $d$. 
For $d\geq5$, let $\nu(d)$ be half the number of multiplicative units in the ring $\Integers/d\Integers$. 
Then there are $\nu(d)$ deformation types of 
pairs $(X,\LB)$ as above, with  $(c_1(\LB),\bullet)$ of divisibility $d$.
\end{prop}

A {\em generalized Kummer  variety} of dimension $2n$ 
is the fiber of the Albanese map $S^{[n+1]}\rightarrow \nolinebreak S$
from the Hilbert scheme of length $n$ subschemes of an abelian surface $S$ to  $S$ itself \cite{beauville}. 
We expect all of the above results  to have analogues for $X$ an irreducible holomorphic-symplectic 
manifold deformation equivalent to a generalized Kummer variety. 
Yoshioka proved Theorem \ref{thm-main} for those $X$ associated to a moduli space
of sheaves on an abelian surface \cite{yoshioka-ample-cone}.
Let the pair $(X,\LB)$ consist of $X$, deformation equivalent to a generalized Kummer,  and a line bundle $\LB$ 
with a primitive and isotropic first Chern class.
The basic construction of section \ref{sec-S-alpha} associates to the pair $(X,\LB)$,
with $\dim(X)=2n$, $n\geq 2$, and with $(c_1(\LB),\bullet)$ of divisibility $d$, 
two dual pairs $(S_1,\alpha_1)$ and $(S_2,\alpha_2)$, each consisting of an abelian surface $S_i$ and
a class $\alpha_i$ in the Neron-Severi group of $S_i$ 
of self intersection $\frac{2n+2}{d^2}$, such that $S_2\cong S_1^*$ and
the natural isometry $H^2(S_1,\Integers)\cong H^2(S_2,\Integers)$ maps $\alpha_1$ to $\alpha_2$. 
A conjectural determination of the monodromy group of
generalized Kummer varieties was suggested in the comment after \cite[Prop. 4.8]{markman-mehrotra}. 
Assuming that the monodromy group is as conjectured, we expect that the proofs of all the results above can be adapted
to this deformation type. 

A version of Theorem \ref{thm-main} has been conjectured for irreducible holomorphic symplectic manifolds 
of all deformation types \cite[Conjecture 2]{markushevich,sawon,beauville-list}. 
Markushevich, Sawon, and Yoshioka proved a version of Theorem \ref{thm-main}, when $X$ is the Hilbert scheme 
of $n$ points on a $K3$ surface and  $(c_1(\LB),\bullet)$ has divisibility $1$
%$K3$-surfaces $S$ such that $S^{[n]}$ is a Lagrangian fibration, for every $n$
\cite[Cor. 4.4]{markushevich} and \cite{sawon} (the regularity of the fibration, in section 5 of \cite{sawon}, 
is due to Yoshioka). 
Bayer and Macri recently proved a strong version of Theorem \ref{thm-main} for
moduli spaces of sheaves on a projective $K3$ surface  \cite{bayer-macri}. 
%The case of divisibility $2$ is treated by Bayer and Macri 
%in \cite[Theorem 9.8]{bayer-macri}. 

\begin{rem} (Added in the final revision).
%Two related results were independently and simultaneously obtained by other authors.
Let $X_0$ be an irreducible holomorphic symplectic manifold and $\LB_0$ a nef line bundle on $X_0$, such that 
$c_1(\LB_0)$ is primitive and isotropic with respect to the Beauville-Bogomolov-Fujiki pairing.
Matsushita proved that if $\LB_0$ induces a Lagrangian fibration, then so does $\LB$ for every pair $(X,\LB)$
deformation equivalent to $(X_0,\LB_0)$, with $X$ irreducible holomorpic symplectic and $\LB$ nef
(preprint posted very recently \cite{matsushita-1310}, announced earlier in his talk \cite{matsushita-simons-talk}). 
It follows that Theorem \ref{thm-main} above holds also without the assumption that $X$ is non-special,
since a pair $(X,\LB)$ with $X$ special is a deformation of a pair $(X_0,\LB_0)$ with $X_0$ non-special.
In fact, this stronger version of Theorem \ref{thm-main}, dropping the non-speciality, follows already from the combination of
Matsushita's result and Example
\ref{ex-completely-integrable-moduli-spaces-for-each-value-of-the-monodromy-invariant} below, since Example
\ref{ex-completely-integrable-moduli-spaces-for-each-value-of-the-monodromy-invariant} exhibits a pair $(X_0,\LB_0)$,
with a line bundle $\LB_0$ inducing a Lagrangian fibration,
in each deformation class of pairs $(X,\LB)$ with $X$ of $K3^{[n]}$-type and $c_1(\LB)$ primitive, isotropic, and 
on the boundary of the positive cone. Matsushita's result does not seem to provide an alternative proof of
Theorem \ref{thm-Tate-Shafarevich-orbit-in-introduction}
and the only proof we know is presented in section \ref{sec-tate-shafarevich} and relies on the preceding sections.
%requires sections  \ref{sec-period-domains} and \ref{sec-density-of-periods-of-compactified-jacobians} below.
%a result, which when combined with the current statement of  Theorem \ref{thm-main}, 
%would enable one to drop the non-speciality assumption in Theorem \ref{thm-main}  \cite{matsushita-simons-talk}.  
\end{rem}

{\bf Acknowledgements:} I would like thank Yujiro Kawamata,
Christian Lehn,  Daisuke Matsushita,
Keiji Oguiso, Osamu Fujino,
Thomas Peternell, S\"{o}nke Rollenske, 
Justin Sawon, and Kota Yoshioka  for helpful communications. I would like to thank the two referees for their
careful reading and insightful comments and suggestions. 

%The results of the current paper were presented in the workshop 
%{\em Holomorphic symplectic manifolds and moduli spaces,}
%Institut Fourier, Grenoble, June 2012. I thank the organizers
%St\'{e}phane Druel and Laurant Manivel for the invitation.
%***************************************************************************
%
%***************************************************************************
\section{Classification of primitive-isotropic classes}
A lattice, in this note, is a finitely generated free abelian group with a symmetric
bilinear pairing $(\bullet,\bullet):L\otimes_\Integers L\rightarrow \Integers$.
The pairing may be degenerate. The isometry group $O(L)$ is the group of automorphisms of
$L$ preserving the bilinear pairing.

\begin{defi}
Two pairs $(L_i,v_i)$, $i=1,2$, 
each consisting of a lattice $L_i$ and an element $v_i\in L_i$, are said to be {\em isometric},
if there exists an isometry $g:L_1\rightarrow L_2$, such that $g(v_1)=v_2$.
\end{defi}

Let $X$ be an irreducible holomorphic symplectic manifold of $K3^{[n]}$-type, $n\geq 2$.
Set $\Lambda:=H^2(X,\Integers)$. We will refer to $\Lambda$ as the {\em $K3^{[n]}$-lattice.}
Let $\widetilde{\Lambda}$ be the Mukai lattice, i.e., the orthogonal direct sum of two copies 
of the negative definite $E_8(-1)$ lattice and four copies of the even unimodular rank two lattice
with signature $(1,-1)$.

\begin{thm}
\label{thm-orbit-iota-X} (\cite{markman-constraints}, Theorem 1.10)
$X$ comes with a natural $O(\widetilde{\Lambda})$-orbit
$\iota_X$ of primitive isometric embeddings 
$\iota:H^2(X,\Integers)\hookrightarrow \widetilde{\Lambda}$.
\end{thm}

Choose a primitive isometric embedding $\iota:\Lambda\hookrightarrow \widetilde{\Lambda}$
in the canonical $O(\widetilde{\Lambda})$-orbit 
$\iota_X$ provided by Theorem \ref{thm-orbit-iota-X}.
Choose a generator $v\in \widetilde{\Lambda}$ of the rank $1$ sub-lattice orthogonal
to $\iota(\Lambda)$. We say that
an isometry $g\in O(\Lambda)$  {\em stabilizes} the $O(\widetilde{\Lambda})$-orbit $\iota_X$, 
if given a representative isometric embedding $\iota$ in the orbit $\iota_X$, 
there exists an isometry $\tilde{g}\in O(\widetilde{\Lambda})$ satisfying
$\iota\circ g=\tilde{g}\circ \iota$. 
Note that $\tilde{g}$ necessarily maps $v$ to $\pm v$.

Set $\Lambda_\RealNumbers:=\Lambda\otimes_\Integers\RealNumbers$.
Let $\widetilde{\C}\subset \Lambda_\RealNumbers$ be the positive cone
$\{x\in \Lambda_\RealNumbers \ : \ (x,x)>0\}$.
Then $H^2(\widetilde{\C},\Integers)$ is isomorphic to $\Integers$ and is a natural character
of the isometry group $O(\Lambda)$ \cite[Lemma 4.1]{markman-torelli}.
Denote by $O^+(\Lambda)$ the kernel of this orientation character. Isometries in
$O^+(\Lambda)$ are said to be {\em orientation preserving}.

\begin{defi}
\label{def-monodromy}
Let $X$, $X_1$, and $X_2$ be irreducible holomorphic symplectic manifolds. 
An isometry $g:H^2(X_1,\Integers)\rightarrow H^2(X_2,\Integers)$ 
is a {\em parallel transport operator}, if there exists a family $\pi:\X \rightarrow B$
(which may depend on $g$) of irreducible holomorphic symplectic manifolds,
points $b_1$ and $b_2$ in $B$, isomorphisms $X_i\cong \X_{b_i}$, where $\X_{b_i}$
is the fiber over $b_i$, $i=1,2$, and a continuous path $\gamma$ from $b_1$ to $b_2$,
such that parallel transport along $\gamma$ in the local system 
$R^2\pi_*\Integers$ induces the isometry $g$. 
When $X=X_1=X_2$, we call $g$ a {\em monodromy operator}.
The {\em monodromy group} $Mon^2(X)$ of  $X$
is the subgroup, of the isometry group of $H^2(X,\Integers)$,
generated by monodromy operators. 
%We denote by $Mon^2(X)\subset O[H^2(X,\Integers)]$ the image of $Mon(X)$
%via the restriction homomorphism.
\end{defi}

\begin{thm} 
\label{thm-Mon-2-is-stabilizer}
(\cite{markman-constraints}, Theorem 1.2 and Lemma 4.2)
The subgroup $Mon^2(X)$ of $O(\Lambda)$ consists of 
orientation preserving isometries stabilizing the orbit $\iota_X$. 
\end{thm}

Given a lattice $L$, let $I_n(L)\subset L$ be the subset
of primitive classes $v$ with $(v,v)=2n-2$.
Notice that the orbit set $I_n(L)/O(L)$ parametrizes the set of isometry classes
of pairs $(L',v')$, such that $L'$ is isometric to $L$
and $v'$ is a primitive class in $L'$ with $(v',v')=2n-2$
\cite[Lemma 9.14]{markman-torelli}. 

Let $n$ be an integer $\geq 2$, let $\Lambda$ be the $K3^{[n]}$-lattice, and 
let $\alpha\in \Lambda$ be a primitive isotropic class. 
Let $\div(\alpha,\bullet)$ be the largest positive integer, such that 
$(\alpha,\bullet)/\div(\alpha,\bullet)$ is an integral class of $\Lambda^*$.
Set $d:=\div(\alpha,\bullet)$ and 
\[
\beta:=\iota(\alpha).
\]
Let $L\subset \widetilde{\Lambda}$ be the saturation\footnote{The {\em saturation} of a sublattice $L'$ of $\Lambda$ 
is the maximal sublattice $L$ of $\Lambda$, of the same rank as $L'$, which contains $L'$.} 
of ${\rm span}_\Integers\{\beta,v\}$.
Clearly, the isometry class of $(L,v)$ depends only on $\alpha$ and the 
$O(\widetilde{\Lambda})$-orbit of $\iota$. Consequently,
the isometry class of $(L,v)$ depends only on $\alpha$, as the $O(\widetilde{\Lambda})$-orbit $\iota_X$ of $\iota$
is natural, by Theorem \ref{thm-orbit-iota-X}. We denote by $[L,v](\alpha)$ the isometry class of the pair $(L,v)$ associated
to $\alpha$.

\begin{lem}
\label{lemma-classification}
\begin{enumerate}
\item
\label{lemma-item-d-square-divides-n-1} 
$d^2$ divides $n-1$.
\item
\label{lemma-item-L-is-L-n-d}
$L$ is isometric to the lattice $L_{n,d}$ with Gram matrix
$\frac{2n-2}{d^2}
\left(\begin{array}{cc}
1 & 0 \\
0 & 0
\end{array}\right).$
\item
\label{lemma-item-isometry-class-of-L-v-determines-Mon-2-orbit}
%The isometry class of the pair $(L,v)$ determines the $Mon^2(X)$-orbit of $\alpha$.
%\item
Let $d\geq 1$ be an integer, such that $d^2$ divides $n-1$.
The map $\alpha\mapsto [L,v](\alpha)$ induces a one-to-one correspondence 
between the set of $Mon^2(X)$-orbits, of primitive isotropic classes $\alpha$
with $\div(\alpha,\bullet)=d$, and the set of isometry classes
$I_n(L_{n,d})/O(L_{n,d})$.
\item
\label{lemma-item-integer-b-determines-isometry-class}
There exists an integer $b$, such that $(\beta-bv)/d$ is an integral class of $L$. 
The isometry class  
$[L,v](\alpha)$ is represented by $(L_{n,d},(d,b))$, for any such integer $b$.
\end{enumerate}
\end{lem}

\begin{proof}
Part (\ref{lemma-item-d-square-divides-n-1}):
There exists a class $\delta\in\Lambda$, such that $(\delta,\delta)=2-2n$ and 
the sub-lattice $\delta^\perp_\Lambda$ of $\Lambda$, orthogonal to $\delta$, is
a unimodular lattice isometric to the $K3$-lattice. 
The sub-lattice $[\iota(\delta^\perp_\Lambda)]^\perp_{\widetilde{\Lambda}}$ 
of $\widetilde{\Lambda}$, 
which  is the saturation of ${\rm span}\{\iota(\delta),v\}$, is
unimodular, hence isometric to the unimodular hyperbolic plane $U$ with
Gram matrix 
$\left(\begin{array}{cc}
0 & -1\\
-1&0
\end{array}\right).$
We may further assume that $v=(1,1-n)$ and $\iota(\delta)=(1,n-1)$, under this isomorphism.
If $X$ is the Hilbert scheme $S^{[n]}$ of a $K3$-surface and $\delta$ is half
the class of the big diagonal, then $\delta$ satisfies the above properties.
Write $\alpha=a\xi+b\delta$, where $\xi$ is a primitive class of the $K3$-lattice
$\delta^\perp_\Lambda$, $a>0$, and $\gcd(a,b)=1$. We get
\[
0=(\alpha,\alpha)=a^2(\xi,\xi)-(2n-2)b^2,
\]
and $(\xi,\xi)$ is even. Hence, $a^2$ divides $n-1$. Furthermore, 
$\div(\delta,\bullet)=2n-2$, $\div(\xi,\bullet)=1$, since $\delta^\perp_\Lambda$ is
unimodular, and $\div(\alpha,\bullet)=\gcd(\div(a\xi,\bullet),\div(b\delta,\bullet))=
\gcd(a,(2n-2)b)=a$.
Thus, $a=d:=\div(\alpha,\bullet)$.

Part (\ref{lemma-item-L-is-L-n-d}):
Note that $\iota(\delta)-v=(2n-2)e$, where $e$ is a primitive isotropic class of 
$\widetilde{\Lambda}$.
Set $\gamma:=\frac{1}{d}(\beta-bv)=\iota(\xi)+\frac{b(2n-2)}{d}e$.
We claim that the lattice $L:={\rm span}_\Integers\{v,\gamma\}$ 
is saturated in $\widetilde{\Lambda}$. 
Indeed, choose $\eta\in\delta^\perp_\Lambda$, such that $(\xi,\eta)=1$.
Then 
$\left(\begin{array}{cc}
(v,e) & (v,\eta)\\
(\gamma,e) & (\gamma,\eta)
\end{array}\right)=
\left(\begin{array}{cc}
-1 & 0 \\
0 & 1
\end{array}\right)$.

Let $G$ be the Gram matrix of $L$ in the basis $\{v,\gamma\}$. 
%Let $c_1$, $c_2$ be integers, such that $c_1b+c_2d=1$.
%Set $A:=\left(\begin{array}{cc}
%c_2&b\\
%-c_1 & d
%end{array}\right).$
Then
\[
G=\frac{2n-2}{d^2}\left(\begin{array}{cc}
d^2 & -bd\\
-bd & b^2
\end{array}\right)
=
\frac{2n-2}{d^2}
\left(\begin{array}{c}d\\-b\end{array}\right)
\left(\begin{array}{cc}d&-b\end{array}\right).
\]
Choose a $2\times 2$ invertible matrix $A$, with integer coefficients, such that 
$A\left(\begin{array}{c}d\\-b\end{array}\right)=\left(\begin{array}{c}1\\0\end{array}\right)$.
Then $AGA^t$ is the Gram matrix 
%of $L$ in the basis $\{c_2v-c_1\gamma,bv+d\gamma\}$ is that 
of $L_{n,d}$.

Part (\ref{lemma-item-isometry-class-of-L-v-determines-Mon-2-orbit}):
Assume given two primitive isotropic classes $\alpha_1$ and $\alpha_2$ in $\Lambda:=H^2(X,\Integers)$ and let
$(L_i,v_i)$ be the pair associated to $\alpha_i$ as above, for $i=1,2$.
In other words, $\iota_i:\Lambda\hookrightarrow \tilde{\Lambda}$ 
is a primitive embedding in the orbit $\iota_X$, $v_i$ generates the 
sub-lattice of $\widetilde{\Lambda}$ orthogonal to the image of $\iota_i$,
and $L_i$ is the saturation of
${\rm span}_\Integers\{\iota(\alpha_i),v_i\}$.

Let us check that the map $\alpha\mapsto [L,v](\alpha)$ is constant on
$Mon^2(X)$-orbits.
Assume that there exists an element $\mu\in Mon^2(X)$,
such that $\mu(\alpha_1)=\alpha_2$. 
%Choose a primitive isometric embedding 
%$\iota_1:\Lambda\hookrightarrow \widetilde{\Lambda}$,
%in the canonical orbit $\iota_X$ provided by Theorem \ref{thm-orbit-iota-X}.
%Set  $\iota_2:=\iota_1\circ \mu$. 
Then there exists an isometry $\tilde{\mu}\in O(\widetilde{\Lambda})$, 
satisfying $\tilde{\mu}\circ \iota_1=\iota_2\circ\mu$, 
by Theorem \ref{thm-Mon-2-is-stabilizer}.
We get that  $\tilde{\mu}(L_1)=L_2$ and $\tilde{\mu}(v_1)=v_2$, or 
$\tilde{\mu}(v_1)=-v_2$.
So, the isometry $\tilde{\mu}$ or $-\tilde{\mu}$ from $L_1$ onto $L_2$
provides an isometry of the pairs
$(L_i,v_i)$, $i=1,2$.

We show next that the map $\alpha\mapsto [L,v](\alpha)$ is injective, i.e., that 
the isometry class of the pair $(L,v)$ determines the $Mon^2(X)$-orbit of $\alpha$.
Assume  that there exists as isometry $f:L_1\rightarrow L_2$,
such that $f(v_1)=v_2$. Then there exists an isometry 
$\tilde{f}\in O(\widetilde{\Lambda})$,
such that $\tilde{f}(L_1)=L_2$ and the restriction of $\tilde{f}$
to $L_1$ is $f$, by (\cite{nikulin}, Proposition 1.17.1 and Theorem 1.14.4, see also 
\cite{markman-monodromy-I}, Lemma 8.1 for more details). 
In particular, $\tilde{f}(v_1)=v_2$.
There exists a unique isometry $h\in O(\Lambda)$
satisfying $\iota_2\circ h=\tilde{f}\circ\iota_1$.
There exists an isometry $\phi\in O(\widetilde{\Lambda})$, such that
$\phi\circ\iota_2=\iota_1$, since both $\iota_i$ belong to the same 
$O(\widetilde{\Lambda})$-orbit $\iota_X$. 
We get the equality $\iota_1\circ h=\phi\circ\iota_2\circ h=(\phi\circ\tilde{f})\circ\iota_1$.
If $h$ is orientation preserving, then $h$ belongs to $Mon^2(X)$,
otherwise, $-h$ does, by Theorem \ref{thm-Mon-2-is-stabilizer}. Let $\mu=h$,
if it is orientation preserving. Otherwise, set $\mu:=-h$. 
Then $\mu$ is a monodromy operator and
$\iota_2(\mu(\alpha_1))=\pm\iota_2(h(\alpha_1))=\pm \tilde{f}(\iota_1(\alpha_1))$.
The class $\iota_1(\alpha_1)$ spans the null space of $L_1$, and $\tilde{f}$
restricts to an isometry from $L_1$ to $L_2$. 
Hence, $\iota_2(\mu(\alpha_1))$ spans the null space of $L_2$.
Hence, $\mu(\alpha_1)=\pm\alpha_2$. 

Finally we show  that $\alpha_2$ and $-\alpha_2$ belong to the same $Mon^2(X)$-orbit.
%Note  that $f(\iota(\alpha_1))=\pm\iota(\alpha_2)$, since the isometry
%$f$ must take the null space of $L_1$ to that of $L_2$.
%Hence, $\tilde{f}(\iota_1(\alpha_1))=\pm \iota_2(\alpha_2)$. 
%Thus, 
%$\iota_1(h(\alpha_1))=
%\phi(\tilde{f}(\iota_1(\alpha_1)))=\pm \phi\circ \iota_2(\alpha_2)=\pm \iota_1(\alpha_2)$.
%We conclude that $h(\alpha_1)=\pm\alpha_2$. 
%Consequently, the monodromy operator $\mu$
%takes $\alpha_1$ to $\alpha_2$ or $-\alpha_2$. 
%In the former case, we are done. Otherwise, 
There exists an element $\tau\in \Lambda$ satisfying
$(\tau,\tau)=2$, and $(\tau,\alpha_2)=0$. The isometry
$\rho_\tau\in O(\Lambda)$, given by
$\rho_\tau(\lambda)=-\lambda+(\lambda,\tau)\tau$,
belongs to $Mon^2(X)$, by (\cite{markman-monodromy-I}, Corollary 1.8), 
and it sends $\alpha_2$ to $-\alpha_2$.
%Hence, $\rho_\tau\circ\mu$ belongs to $Mon^2(X)$ and
%$\rho_\tau(\mu(\alpha_1))=\alpha_2$.

It remains to prove that the map $\alpha\mapsto [L,v](\alpha)$
is surjective. Assume given a primitive class $v\in L_{n,d}$
with $(v,v)=2n-2$. There exists a primitive isometric embedding 
$f:L_{n,d}\hookrightarrow \widetilde{\Lambda}$, by
(\cite{nikulin}, Proposition 1.17.1). The lattice 
$f(v)^\perp_{\widetilde{\Lambda}}$, orthogonal to $f(v)$ in $\widetilde{\Lambda}$, is isometric to
the $K3^{[n]}$-lattice $\Lambda$. 
Choose such an isometry $h:f(v)^\perp_{\widetilde{\Lambda}}\rightarrow\Lambda$, 
with the property that $h^{-1}:\Lambda\hookrightarrow\widetilde{\Lambda}$
belongs to the $O(\widetilde{\Lambda})$-orbit $\iota_X$. 
Such a choice exists, since $O(\Lambda)$ acts transitively 
on the orbit space $O(\Lambda,\widetilde{\Lambda})/O(\widetilde{\Lambda})$,
by (\cite{markman-constraints}, Lemma 4.3). Above,
$O(\Lambda,\widetilde{\Lambda})$ denotes the set of primitive isometric 
embeddings of $\Lambda$ in $\widetilde{\Lambda}$.
Denote by $\beta\in L_{n,d}$ a generator of the null space
of $L_{n,d}$.  Set $\alpha:=h(f(\beta))$.
Then $\alpha$ is a class in $\Lambda$, such that
$[L,v](\alpha)$ is represented by $(L_{n,d},v)$. 
%Choose a connected component
%$\FM_\Lambda^0$, of the moduli space of marked pairs $(X,\eta)$ of 
%$K3^{[n]}$-type, such that the $O(\widetilde{\Lambda})$-orbit of isometric embeddings 
%$\iota_X\circ \eta^{-1}$, of $\Lambda$ in $\widetilde{\Lambda}$, 
%is represented by $h^{-1}$. 
%Such a component exists, since $O(\Lambda)$ acts transitively 
%on the orbit space $O(\Lambda,\widetilde{\Lambda})/O(\widetilde{\Lambda})$,
%by (\cite{markman-constraints}, Lemma 4.3).
%The period map from $\FM_\Lambda^0$ to the period domain $\Omega_\Lambda$
%is surjective, by (\cite{huybrechts-basic-results}, Theorem 8.1). 
%In particular, the exists a pair $(X,\eta)$ with period in the
%non-empty  hyperplane section $\Omega_\Lambda\cap \alpha^\perp$. 
%Then both $\tilde{\alpha}:=\eta^{-1}(\alpha)$ 
%and $-\tilde{\alpha}$ are classes in $H^{1,1}(X,\Integers)$ with isometry class 
%$[L,v](\pm\tilde{\alpha})$ is represented by $(L_{n,d},v)$. 
%Choose the sign of 
%$\pm\tilde{\alpha}$, so that it belongs to the closure of the positive cone $\C_X$.

%Set $\varphi:=\phi\circ \tilde{f}$.
%Now $\phi$ maps $v_1$ to $v_2$ or $-v_2$. Hence, $\varphi(v_1)$
%is $v_1$ or $-v_1$. 
Part (\ref{lemma-item-integer-b-determines-isometry-class}):
The existence of such an integer $b$ was established in the course of proving part
(\ref{lemma-item-d-square-divides-n-1}).
The rest of the statement follows from Lemma \ref{lemma-isometry-orbits-in-L-n-d}.
\end{proof}

If $d=2$, set $\nu(d):=1$. If $d>2$, let $\nu(d)$ be half the number of multiplicative units 
in the ring $\Integers/d\Integers$.

\begin{lem}
\label{lemma-isometry-orbits-in-L-n-d}
A vector 
$(x,y)\in L_{n,d}$ is primitive of degree $2n-2$, if and only if $\Abs{x}=d$ and
$\gcd(d,y)=1$. 
Two primitive vectors
$(d,y)$, $(d,z)$ belong to the same $O(L_{n,d})$-orbit, if and only if 
$y\equiv z$ modulo $d$, or  $y\equiv -z$ modulo $d$. 
Consequently, $\nu(d)$ 
is equal to the number of  $O(L_{n,d})$-orbits of primitive vectors in $L_{n,d}$ of degree $2n-2$.
\end{lem}

\begin{proof}
The isometry group of $L_{n,d}$ consists of matrices of the form
$\left(\begin{array}{cc}
\pm 1 & 0\\
c & \pm 1
\end{array}\right)$.
The orbit $O(L_{n,d})(d,y)$ consists of vectors of the form
$(\pm d, cd\pm y)$. 
Consequently, the number of  $O(L_{n,d})$-orbits
of primitive vectors in $L_{n,d}$ of degree $2n-2$ is equal to 
the number of orbits in $\{y : \ 0<y<d \ \mbox{and} \ \gcd(y,d)=1\}$ under the action $y\mapsto d-y$. 
The latter number is $\nu(d)$.
\end{proof}

%***************************************************************************
%
%***************************************************************************
\section{An example of a Lagrangian fibration for each value of the monodromy invariants}

Let $S$ be a projective $K3$ surface, $K(S)$ its topological $K$-group, generated by classes
of complex vector bundles, and $H^*(S,\Integers)$ its integral cohomology ring. 
Let $td_S:=1+\frac{c_2(S)}{12}$ be the Todd class of $S$ and $\sqrt{td_S}:=1+\frac{c_2(S)}{24}$
its square root.
The homomorphism $v:K(S)\rightarrow H^*(S,\Integers)$, given by 
$
v(x)=ch(x)\sqrt{td_S}
$
is an isomorphism of free abelian groups. Given a coherent
sheaf $E$ on $S$, the class $v(E)$ is called the {\em Mukai vector} of $E$.
Given integers $r$ and $s$ and a class $c\in H^2(S,\Integers)$, we will denote by 
$(r,c,s)$ the class of $H^*(S,\Integers)$, whose graded summand in $H^0(S,\Integers)$
is $r$ times the class Poincare dual to $S$, 
its graded summand in $H^2(S,\Integers)$ is $c$, 
and its graded summand in $H^4(S,\Integers)$ is
$s$ times the class Poincare dual to a point.
We endow $H^*(S,\Integers)$ with the {\em Mukai pairing} 
\[
((r,c,s),(r',c',s')) \ \ := \ \ (c,c')-rs'-r's,
\]
where $(c,c'):=\int_Sc\cup c'.$ 
Then $(v(x),v(y))=-\chi(x\otimes y)$, where $\chi:K(S)\rightarrow \Integers$ is the Euler
characteristic. $H^*(S,\Integers)$, 
endowed with the Mukai pairing, is called the {\em Mukai lattice}. 
The Mukai lattice is an even unimodular lattice of rank $24$, which is  
isometric to the orthogonal direct sum
of two copies of the negative definite $E_8(-1)$ lattice and 
four copies of the even unimodular rank $2$ hyperbolic lattice $U$.

Let $v\in K(S)$ be the class with Mukai vector 
$(0,d\xi,s)$ in $H^*(S,\Integers)$, such that  
$\xi$ a primitive effective class in $H^{1,1}(S,\Integers)$, $(\xi,\xi)>0$, 
$d$ is a positive integer, and $\gcd(d,s)=1$. 
 There is a system of hyperplanes in the ample cone of $S$, called $v$-walls,
that is countable but locally finite \cite[, Ch. 4C]{huybrechts-lehn-book}.
An ample class is called {\em $v$-generic}, if it does not
belong to any $v$-wall. Choose a $v$-generic ample class $H$. 
Let $M_H(v)$ be the moduli space of $H$-stable  
sheaves on the $K3$ surface $S$ with class $v$.
%When non-empty, the moduli space 
$M_H(v)$ is a smooth projective irreducible holomorphic symplectic variety
of $K3^{[n]}$-type, with $n=\frac{(v,v)+2}{2}=\frac{d^2(\xi,\xi)+2}{2}$.
This is a special case of a result, which is due to several people, including 
Huybrechts, Mukai, O'Grady, and Yoshioka. It can be found in its final form in
\cite{yoshioka-abelian-surface}.

Over $S\times M_H(v)$ there exists a universal sheaf
$\F$, possibly twisted with respect to a non-trivial 
Brauer class pulled-back from $M_H(v)$.
Associated to $\F$ is a class $[\F]$ in $K(S\times M_H(v))$
(\cite{markman-integral-generators}, Definition 26).
Let $\pi_i$ be the projection from $S\times M_H(v)$ onto the $i$-th factor. 
Denote by $v^\perp$ the sub-lattice  in $H^*(S,\Integers)$ orthogonal to $v$.
The second integral cohomology $H^2(M_H(v),\Integers)$, its Hodge structure, 
and its Beauville-Bogomolov-Fujiki pairing, are all described by 
Mukai's Hodge-isometry
\begin{equation}
\label{eq-Mukai-isomorphism}
\theta \ : \ v^\perp \ \ \  \longrightarrow \ \ \  H^2(M_H(v),\Integers),
\end{equation}
given by $\theta(x):=c_1\left(\pi_{2_!}\{\pi_1^!(x^\vee)\otimes [\F]\}\right)$
(see \cite{yoshioka-abelian-surface}).

We provide next an example of a moduli space $M_H(v)$  
and a primitive isotropic class $\alpha\in H^{1,1}(M_H(v),\Integers)$, such that $[L,v](\alpha)$
is represented by $(L_{n,d},(d,b))$, for every integer $n\geq 2$, for every positive integer $d$, such that $d^2$
divides $n-1$, and for every integer $b$ satisfying $\gcd(b,d)=1$.

\begin{example}
\label{ex-completely-integrable-moduli-spaces-for-each-value-of-the-monodromy-invariant}
Let $d$ be a positive integer, such that $d^2$ divides $n-1$.
Let $S$ be a $K3$ surface with a nef line bundle $\LB$
of degree $\frac{2n-2}{d^2}$. Let $\lambda$ be the class $c_1(\LB)$ in $H^2(S,\Integers)$.
Fix an integer $b$ satisfying $\gcd(b,d)=1$.
Set $v:=(0,d\lambda,s)$, where $s$ is an integer satisfying $sb=1$ (modulo $d$).
Then $v$ is a primitive Mukai vector and $(v,v)=2n-2$. 
Choose a $v$-generic ample line bundle $H$.
A sheaf $F$ of class $v$ is $H$-stable, if and only if it is
$H$-semi-stable. The moduli space $M_H(v)$,  of $H$-stable sheaves of class $v$, is smooth,
projective, holomorphic symplectic, and of $K3^{[n]}$-type.
% (see \cite{yoshioka-abelian-surface}).
%Let $H^*(S,\Integers)$ be the Mukai lattice of $S$. 
%\[
%\theta \ : \ v^\perp \ \ \ \rightarrow \ \ \ H^2(M_H(v),\Integers)
%\]
%be Mukai's isometry. 
Set $\alpha:=\theta((0,0,1))$.
Let $\iota:H^2(M_H(v),\Integers)\rightarrow H^*(S,\Integers)$ be the composition of $\theta^{-1}$ with the
inclusion of $v^\perp$ into $H^*(S,\Integers)$. 
A Mukai vector $(r,c,t)$ belongs to $v^\perp$, if and only if $rs=d(c,\lambda)$. It follows that $d$ divides $r$,
since $\gcd(d,s)=1$. Thus, $\div(\alpha,\bullet)=d$. 
Now
\[
\iota(\alpha)-bv=(0,-bd\lambda,1-bs)
\]
is divisible by $d$, by our assumption on $s$. 
Hence, the monodromy invariant $[L,v](\alpha)$ is equal to the isometry class of
$(L_{n,d},(d,b))$, by Lemma \ref{lemma-classification}.
%The linear system $\linsys{\LB}$ is non-empty and all its divisors are connected, since 
%$\LB$ is a nef divisor of positive degree
%\footnote{Choose an effective divisor $D$ in $\linsys{\LB}$.
%Each connected component of $D$ must have non-negative degree, since $\LB$ is nef, 
%and at least one, say $D_1$, must have positive degree, since $D$ has positive degree. The class of
%an effective divisor of positive degree belongs to the interior of the positive cone, and that of an effective
%divisor of non-negative degree belongs to the closure of the positive cone.
%Hence, $D_1$ intersects each of the other connected components.
%}. 
%Hence, $h^1(S,\LB)=0,$ by 
%\cite[Theorem 4]{kodaira}.  
The cohomology $H^1(S,\LB^d)$ vanishes, since
$\LB$ is a nef divisor of positive degree \cite[Prop. 1]{mayer}.
Thus, the vector space $H^0(S,\LB^d)$ has dimension $\chi(\LB^d)=n+1$.
The support morphism 
$\pi:M_H(v)\rightarrow \linsys{\LB^d}$ realizes $M_H(v)$ as a completely integrable system. 
The equality $\pi^*c_1(\StructureSheaf{\linsys{\LB^d}}(1))=\alpha$ is easily verified. 
\end{example}

%\begin{example}
%Let $p:S\rightarrow \PP^1$ be an elliptic $K3$ surface with a section $\sigma:\PP^1\rightarrow S$.
%Assume that $\Pic(S)$ is generated by the line-bumdles $\StructureSheaf{S}(\sigma(\PP^1))$ and
%$p^*\StructureSheaf{\PP^1}(1)$, and set $\LB:=\StructureSheaf{S}(\sigma(\PP^1))\otimes p^*\StructureSheaf{\PP^1}(n)$, 
%$n\geq 2$.
%Then $\LB$ is nef, $\lambda:=c_1(\LB)$ is primitive, and so $v:=(0,\lambda,s)$ is primitive for every integer $s$.
%When $s=1$, $M_H(v)$ is birational to $M_H(1,\lambda,0)$. The latter is isomorphic to
%the Hilbert scheme $S^{[n]}$. Now $S^{[n]}$ admits the Lagrangian fibration
%$\pi:S^{[n]}\rightarrow (\PP^1)^{[n]}\cong \PP^n$ and
%$\pi^*c_1(\StructureSheaf{\PP^n}(1))$ corresponds to $\theta(0,0,1)$ in $H^2(M_H(v),\Integers)$. 
%\end{example}

%***************************************************************************
%
%***************************************************************************
\section{Period domains and period maps}
\label{sec-period-domains}
%\begin{question}
%Let $X$ be of $K3^{[n]}$-type, $n\geq 2$, and $\alpha\in H^{1,1}(X,\Integers)$
%a primitive, isotropic, and nef class. Is $\alpha$ Poincare-dual to the class of an 
%effective divisor?
%\end{question}
%***************************************************************************
%
%***************************************************************************
\subsection{A projective $K3$ surface associated to an isotropic class}
\label{sec-S-alpha}
Let $X$ be an irreducible holomorphic symplectic manifold of $K3^{[n]}$-type, $n\geq 2$.
Assume that there exists a non-zero primitive isotropic class $\alpha\in H^{1,1}(X,\Integers)$. 
%We may assume that $\alpha$ belongs to the boundary of the positive cone,
%possibly after replacing $\alpha$ by $-\alpha$. 
%We may further assume that 
%$\alpha$ belongs to the boundary of the birational K\"{a}hler cone, possibly after 
%replacing $\alpha$ by an element in the $W_{Exc}$-orbit of $\alpha$ \cite[Theorem 6.18]{markman-torelli}.
Let $\widetilde{\Lambda}$ be the Mukai lattice. Choose a primitive isometric embedding 
$\iota:H^2(X,\Integers) \rightarrow \widetilde{\Lambda}$ 
in the canonical $O(\widetilde{\Lambda})$-orbit $\iota_X$ of Theorem \ref{thm-orbit-iota-X}. 
Set $\widetilde{\Lambda}_\ComplexNumbers:=\widetilde{\Lambda}\otimes_\Integers\ComplexNumbers$.
Endow 
$\widetilde{\Lambda}_\ComplexNumbers$ with 
the weight $2$ Hodge structure, so that $\widetilde{\Lambda}_\ComplexNumbers^{2,0}=\iota(H^{2,0}(X))$.
Set $\beta:=\iota(\alpha)$.
Then $\beta$ belongs to $\widetilde{\Lambda}_\ComplexNumbers^{1,1}$.
Set
\[
\Lambda_{k3} \ \ := \ \ \beta^\perp_{\widetilde{\Lambda}}/\Integers\beta
\]
and endow 
$\Lambda_{k3}$ with the 
induced Hodge structure. 
Let $U$ be the even unimodular 
rank $2$ lattice of signature $(1,1)$, and $E_8(-1)$  the negative definite $E_8$ lattice.
Then $\Lambda_{k3}$  is isometric to the $K3$ lattice,
which is the orthogonal direct sum of two copies of $E_8(-1)$ and three copies of $U$.
Indeed, this is clear if $\beta$ is a class in a direct summand of $\widetilde{\Lambda}$ isometric to $U$. 
It follows in general, since the isometry group of $\widetilde{\Lambda}$ acts transitively on the 
set of primitive isotropic classes in $\widetilde{\Lambda}$.
%since $\beta$ is contained in some orthogonal direct summand 
%of $\widetilde{\Lambda}$ isometric to $U$. 
The induced Hodge structure on $\Lambda_{k3}$ 
is the weight $2$ Hodge structure of some $K3$ surface $S(\alpha)$, by the surjectivity of the period map.

Let $v$ be a generator of the rank $1$ sub-lattice of $\widetilde{\Lambda}$ orthogonal to the image
of $\iota$. Then $v$ is of Hodge-type $(1,1)$. Set $\Lambda:= H^2(X,\Integers)$. Then $v^\perp$ is isometric to $\Lambda$. 
We claim that $(v,v)=2n-2$. 
Indeed, the pairing induces an isomorphism of the two discriminant groups
$(\Integers v)^*/\Integers v$ and $\Lambda^*/\Lambda$, since  $\Integers v$ and $\Lambda$
are a pair of primitive sublattices, which are orthogonal complements in the unimodular lattice $\widetilde{\Lambda}$.
We conclude that the order $|(v,v)|$ of $(\Integers v)^*/\Integers v$ is equal to the order $2n-2$ of
$\Lambda^*/\Lambda$. Finally, $(v,v)>0$, by comparing the signatures of $\Lambda$ and $\widetilde{\Lambda}$.

Let $\bar{v}$ be the coset $v+\Integers\beta$ in $\Lambda_{k3}$.
Then $\bar{v}$ is of Hodge-type $(1,1)$ and 
$(\bar{v},\bar{v})=2n-2$. Hence $S(\alpha)$ is a projective $K3$ surface
(even if $X$ is not projective). We may further choose the Hodge isometry 
$\eta:H^2(S(\alpha),\Integers)\rightarrow \Lambda_{k3}$, so that 
that $\bar{v}$ corresponds  to a class in the positive cone of $S(\alpha)$,
possibly after replacing $v$ by $-v$. We may further assume that $\bar{v}$ corresponds to a nef class 
of $S(\alpha)$, possibly after replacing $\eta$ with $\eta\circ w$, where 
$w$ is an element of the subgroup 
$W\subset O^+(H^2(S(\alpha),\Integers))$, generated by reflections by classes of smooth rational curves on $S(\alpha)$
\cite[Prop. 1.9]{looijenga-peters}.

%*************************************************************************************************
%
%*************************************************************************************************
\subsection{A period domain as an affine line bundle over another}
Keep the notation of section \ref{sec-S-alpha}.
Set $\Lambda:=H^2(X,\Integers)$.
Set $d:=\div(\alpha,\bullet)$. 
Let $\alpha^\perp_\Lambda$
be the (degenerate) lattice orthogonal to $\alpha$ in $\Lambda$.
Set $Q_\alpha:=\alpha^\perp_\Lambda/\Integers\alpha.$

\begin{lem}
\label{lemma-Lambda-alpha-is-isometric-to-v-bar-perp}
$Q_\alpha$ is isometric to the sub-lattice $\bar{v}^\perp$ of $\Lambda_{k3}$ and both are 
isometric to the orthogonal direct sum
\[
E_8(-1)\oplus E_8(-1)\oplus U\oplus U \oplus \Integers \lambda,
\]
where $(\lambda,\lambda)=\frac{2-2n}{d^2}$.
\end{lem}

\begin{proof}
The $K3$ lattice $\Lambda_{k3}:=[\beta^\perp_{\widetilde{\Lambda}}]/\Integers\beta$
is isometric to $E_8(-1)\oplus E_8(-1)\oplus U\oplus U \oplus U$.
%since $\beta$ is contained in some orthogonal direct summand 
%of $\widetilde{\Lambda}$ isometric to $U$. 
Let $L$ be the saturation 
of ${\rm span}_\Integers\{v,\beta\}$ in $\widetilde{\Lambda}$.
Then $L$ is contained in $\beta^\perp_{\widetilde{\Lambda}}$ and the image
of $L$ in $\Lambda_{k3}$ is spanned by a class 
$\xi$ of self-intersection $\frac{2n-2}{d^2}$, such that $\bar{v}=d\xi$, by Lemma \ref{lemma-classification}.

It remains to prove that $Q_\alpha$ is isometric to $\xi^\perp_{\Lambda_{k3}}$.
Consider the following commutative diagram.
\[
\begin{array}{ccccccc}
0\rightarrow & \Integers\beta& \rightarrow & \beta^\perp_{\widetilde{\Lambda}}&
\rightarrow &
%[\beta^\perp_{\widetilde{\Lambda}}]/L\Integers\beta
\Lambda_{k3} & \rightarrow 0
\\
&= \ \uparrow \ \hspace{1Ex} & & \uparrow & & \hspace{1ex} \ \uparrow \ j
\\
0\rightarrow & \Integers\beta & \rightarrow & L^\perp_{\widetilde{\Lambda}} & 
\rightarrow & L^\perp_{\widetilde{\Lambda}} /\Integers\beta&\rightarrow 0
\\
& \cong \ \uparrow \ \hspace{1Ex} & & \cong \ \uparrow \ \iota & & 
\hspace{1Ex} \ \uparrow \ \bar{\iota}
\\
0\rightarrow & \Integers\alpha & \rightarrow & \alpha^\perp_\Lambda & \rightarrow &
Q_\alpha & \rightarrow 0.
\end{array} 
\]
The lower vertical arrow $\bar{\iota}$ in the rightmost column is evidently an isomorphism.
The image of the upper one $j$ is precisely $\xi^\perp_{\Lambda_{k3}}$.
%The upper vertical arrows are all injective. Hence, it suffices to prove that 
%$j$ is surjective. It also suffices to prove that the induced homomorphism
%\[
%\psi:L/\Integers\beta\rightarrow 
%[\beta^\perp_{\widetilde{\Lambda}}]/L^\perp_{\widetilde{\Lambda}}
%\]
%is surjective, by the snake lemma applied 
%to the upper two rows of the above diagram.
%The homomorphism $\psi$ is injective, by the snake lemma.
%We have the natural isomorphism 
%$[\beta^\perp_{\widetilde{\Lambda}}]/L^\perp_{\widetilde{\Lambda}}\cong [L/\Integers\beta]^*$,
%since the lattice $\widetilde{\Lambda}$ is unimodular.
\end{proof}

Let $\Omega_\Lambda$ be the period domain
\begin{equation}
\label{eq-period-domain}
\Omega_\Lambda:= \{\ell\in \PP[H^2(X,\ComplexNumbers)] \ : \
(\ell,\ell)=0 \ \mbox{and} \ (\ell,\bar{\ell})>0\}.
\end{equation}
%Let $\Omega_\Lambda$ be the period domain, given in Equation (\ref{eq-period-domain}), 
Set
\begin{equation}
\label{hyperplane-section-of-period-domain-orthogonal-to-alpha}
\Omega_{\alpha^\perp}:=\{\ell\in \Omega_\Lambda \ : \ (\ell,\alpha)=0\}.
\end{equation}
Then $\Omega_{\alpha^\perp}$ is an affine line-bundle over the period domain
\[
\Omega_{Q_\alpha} \ \ := \ \ 
\{\ell\in \PP[Q_\alpha\otimes_\Integers\ComplexNumbers] \ : \
(\ell,\ell)=0 \ \mbox{and} \  (\ell,\bar{\ell})>0 \}.
\]
Given a point of $\Omega_{Q_\alpha}$, corresponding to 
a one-dimensional subspace 
$\ell$ of $Q_\alpha\otimes_\Integers\ComplexNumbers$, 
we get a two dimensional subspace $V_\ell$ of $H^2(X,\ComplexNumbers)$ orthogonal to 
$\alpha$ and containing $\alpha$. 
The line in $\Omega_{\alpha^\perp}$, over the point
$\ell$ of $\Omega_{Q_\alpha}$, is 
$\PP[V_\ell]\setminus \{\PP[\ComplexNumbers\alpha]\}$.
Denote by 
\begin{equation}
\label{eq-q}
q \ : \ \Omega_{\alpha^\perp} \ \ \rightarrow \ \ \Omega_{Q_\alpha}
\end{equation}
the bundle map.
A {\em semi-polarized} $K3$ surface of degree $k$ is a pair consisting of a $K3$ surface together with a nef line bundle 
of degree $k$ (also known as weak algebraic polarization of degree $k$ in \cite[Section 5]{morrison-some-remarks}).
Note that each component of $\Omega_{Q_\alpha}$ is isomorphic to the period domain
of the moduli space of semi-polarized $K3$ surfaces of degree
$\frac{2n-2}{d^2}$. 

\begin{defi}
\label{def-tate-shafarevich-lines}
Fibers of $q$ will be called
{\em Tate-Shafarevich lines} for reasons that will become apparent in section 
\ref{sec-tate-shafarevich}.
\end{defi}

Tate-Shafarevich lines are limits of twistor lines, as will be explained in Remark
\ref{rem-Tate-Shafarevich-lines-are-limits-of-twistor-lines}.

%*************************************************************************************************
%
%*************************************************************************************************
\subsection{The period map}
\label{sec-period-map}
Given a period $\ell\in \Omega_\Lambda$, set
$\Lambda^{1,1}(\ell,\Integers):=\{\lambda\in \Lambda \ : \ (\lambda,\ell)=0\}$.
Define $Q_\alpha^{1,1}(q(\ell),\Integers)$ similarly.
%for $q(\ell)\in\Omega_{Q_\alpha}$.
%If $\bar{\ell}\in\Omega_{Q_\alpha}$ is the image of $\ell\in \Omega_{\alpha^\perp}$, 
We get the short exact sequence
\[
0\rightarrow \Integers\alpha\rightarrow 
\left[\alpha^\perp\cap \Lambda^{1,1}(\ell,\Integers)\right]\rightarrow 
Q_\alpha^{1,1}(q(\ell),\Integers)\rightarrow 0.
\]

$\Omega_{\alpha^\perp}$ 
has two connected components, since $\Omega_{Q_\alpha}$ has
two connected components. 
Indeed, $Q_\alpha$ has signature $(2,b_2(X)-4)$,
and a period $\ell$ comes with an oriented positive definite plane
$[\ell\oplus \bar{\ell}]\cap [\Lambda_\RealNumbers]$, 
which, in turn, determines the orientation of the positive cone
in $Q_\alpha\otimes_\Integers{\RealNumbers}$.

The positive cone $\widetilde{\C}_\Lambda$ in $\Lambda_\RealNumbers$ is the cone
\begin{equation}
\label{eq-positive-cone-in-Lambda}
\widetilde{\C}_\Lambda := \{x\in \Lambda_\RealNumbers \ : \ (x,x)>0\}.
\end{equation}
The cohomology group $H^2(\widetilde{\C}_\Lambda,\Integers)$ is isomorphic to $\Integers$
and an {\em orientation} of $\widetilde{\C}_\Lambda$ is the choice of one of the two generator 
of $H^2(\widetilde{\C}_\Lambda,\Integers)$. An orientation of $\widetilde{\C}_\Lambda$
determines an orientation of every positive definite three dimensional subspace of $\Lambda_\RealNumbers$
\cite[Lemma 4.1]{markman-torelli}.
A choice of an orientation of $\widetilde{\C}_\Lambda$ determines a choice of a
component of $\Omega_{\alpha^\perp}$ as follows. 
A period
$\ell\in \Omega_\Lambda$ determines the subspace $\Lambda^{1,1}(\ell,\RealNumbers)$
and the cone $\C'_\ell:=\{x\in \Lambda^{1,1}(\ell,\RealNumbers) \ : \ (x,x)>0\}$
in $\Lambda^{1,1}(\ell,\RealNumbers)$
has two connected components. 
A choice of a 
connected component of $\C'_\ell$ is equivalent to a choice of an orientation of the
positive cone of $\widetilde{\C}_\Lambda$. 
Indeed, a non-zero element $\sigma\in \ell$ and an element $\omega\in \C'_\ell$
determine a basis $\{{\mbox Re}(\sigma), {\mbox Im}(\sigma), \omega\}$,
hence an orientation, 
of a positive definite three dimensional subspace of $\Lambda_\RealNumbers$,
and the corresponding orientation of $\widetilde{\C}_\Lambda$ is independent of the choice
of $\sigma$ and $\omega$. 
Thus, the choice of the orientation of 
the positive cone $\widetilde{\C}_\Lambda$ determines a connected component 
$\C_\ell$ of $\C'_\ell$, called the {\em positive cone} (for the orientation). 
If $\ell$ belongs to $\Omega_{\alpha^\perp}$, then the class $\alpha$ belongs to 
$\Lambda^{1,1}(\ell,\RealNumbers)$ and $\alpha$ is in the closure of precisely
one of the two connected components of $\C'_\ell$. 
The connected component of $\Omega_{\alpha^\perp}$, compatible with the chosen orientation of 
$\widetilde{\C}_\Lambda$, is the one 
for which $\alpha$ belongs to the boundary of the positive cone $\C_\ell$ for the chosen orientation.

A {\em marked pair} $(Y,\psi)$ consists of an irreducible holomorphic symplectic manifold $Y$ 
and an isometry $\psi$ from $H^2(Y,\Integers)$ onto a fixed lattice. 
The moduli space of isomorphism classes of marked pairs is a 
non-Hausdorff complex manifold \cite{huybrechts-basic-results}.
Let $\FM^0_\Lambda$ be a connected component 
of the moduli space of marked pairs of $K3^{[n]}$-type, where the fixed lattice is $\Lambda$. 
The {\em period map} 
\[
P_0:\FM^0_\Lambda\rightarrow \Omega_\Lambda
\]
sends a marked pair $(Y,\psi)$ to the point $\psi(H^{2,0}(Y))$ of $\Omega_\Lambda$. 
$P_0$ is a holomorphic map and a local homeomorphism \cite{beauville}.
The positive cone $\C_Y$ is the connected component of
the cone $\{x\in H^{1,1}(Y,\RealNumbers) \ : \ (x,x)>0\}$ containing the K\"{a}hler cone. Hence, 
the positive cone in $H^2(Y,\RealNumbers)$ comes with a canonical orientation and the marking $\psi$ 
determines an orientation of the positive cone in $\widetilde{\C}_\Lambda$. We conclude that 
$\FM^0_\Lambda$ determines an orientation of the positive cone $\widetilde{\C}_\Lambda$ \cite[Sec. 4]{markman-torelli}.
Let 
\begin{equation}
\label{eq-connected-component-of-Omaga-alpha-perp}
\Omega_{\alpha^\perp}^+
\end{equation}
be the connected component
of $\Omega_{\alpha^\perp}$, inducing the same orientation of
$\widetilde{\C}_\Lambda$ as $\FM^0_\Lambda$.
Let
\begin{equation}
\label{eq-FM-alpha-perp}
\FM^0_{\alpha^\perp}
\end{equation}
be the inverse image $P_0^{-1}(\Omega_{\alpha^\perp}^+)$.

\begin{thm}
\label{thm-Torelli}
(The Global Torelli Theorem \cite{verbitsky,huybrechts-bourbaki})
The period map $P_0:\FM^0_\Lambda\rightarrow \Omega_\Lambda$ is surjective. 
Any two points in the same fiber of $P_0$ are inseparable. If
$(X_1,\eta_1)$ and $(X_2,\eta_2)$ correspond to two inseparable points in $\FM^0_\Lambda$,
then $X_1$ and $X_2$ are bimeromorphic. If the K\"{a}hler cone of $X$ is equal to its positive cone and 
$(X,\eta)$ corresponds to a point of $\FM^0_\Lambda$, then this point is separated.
\end{thm}

\begin{lem}
\label{lemma-connected}
$\FM^0_{\alpha^\perp}$ is path-connected.
\end{lem}

\begin{proof}
The statement follows from the Global Torelli Theorem \ref{thm-Torelli} and the fact that $\Omega_{\alpha^\perp}^+$
is connected. The proof is 
similar to that of \cite[Proposition 5.11]{markman-prime-exceptional}.
\end{proof}

\begin{prop}
\label{prop-same-connected-component}
Let $X_1$ and $X_2$ be two irreducible holomorphic symplectic manifolds of $K3^{[n]}$-type
and $\eta_j:H^2(X_j,\Integers)\rightarrow \Lambda$, $j=1,2$, isometries.
The marked pairs 
$(X_1,\eta_1)$ and $(X_2,\eta_2)$ belong to the same 
connected moduli space $\FM^0_{\alpha^\perp}$, provided the following conditions hold.
\begin{enumerate}
\item
\label{cor-item-orbits-are-equal}
The $O(\widetilde{\Lambda})$ orbits $\iota_{X_j}\circ\eta_j^{-1}$, $j=1,2$, are equal. Above
$\iota_{X_j}$ is the canonical $O(\widetilde{\Lambda})$-orbit of primitive isometric embeddings
of $H^2(X_j,\Integers)$ into $\widetilde{\Lambda}$ mentioned in Theorem \ref{thm-orbit-iota-X}.
\item
\label{cor-item-orientation-preserving}
$\eta_2^{-1}\circ\eta_1:H^2(X_1,\Integers)\rightarrow H^2(X_2,\Integers)$ is orientation preserving.
\item
\label{cor-item-alpha-pulls-back-to-a-1-1-class}
$\eta_j^{-1}(\alpha)$ is of Hodge type $(1,1)$ and it belongs to the boundary of the positive cone
$\C_{X_j}$ in $H^{1,1}(X_j,\RealNumbers)$, for $j=1,2$.
\end{enumerate}
\end{prop}

\begin{proof}
Conditions \ref{cor-item-orbits-are-equal} and \ref{cor-item-orientation-preserving} imply that 
$\eta_2^{-1}\circ\eta_1$ is a parallel-transport operator, by Theorem
\ref{thm-Mon-2-is-stabilizer}. Hence, the two marked pairs belong to the same 
connected component $\FM_\Lambda^0$ of $\FM_\Lambda$.
Condition \ref{cor-item-alpha-pulls-back-to-a-1-1-class} implies that both belong 
to $\FM_{\alpha^\perp}^0$, and the latter is connected, by Lemma 
\ref{lemma-connected}.
\end{proof}

\begin{proof} (of Proposition \ref{prop-counting-deformation-invariants})
Lemma \ref{lemma-classification} introduced the monodromy invariant $[L,v](c_1(\LB))$ of the pair
$(X,\LB)$. The claimed number of deformation types in the statement of the proposition is equal to
the number of values of the monodromy invariant $[L,v](\bullet)$ for fixed $n$ and $d$, by
Lemma \ref{lemma-isometry-orbits-in-L-n-d}. 
Assume given another pair $(X',\LB')$ as above, such that the  monodromy invariants
$[L,v](c_1(\LB'))$ and $[L,v](c_1(\LB))$ are equal. 
Choose  a parallel transport operator 
$g:H^2(X',\Integers)\rightarrow H^2(X,\Integers)$. 
We do not assume that $g(c_1(\LB'))$ is of Hodge type $(1,1)$.
Set $\alpha:=c_1(\LB)$ and $\alpha':=c_1(\LB')$. 
The monodromy invariant $[L,v](g(\alpha'))$ is equal to $[L,v](\alpha')$
and hence also to $[L,v](\alpha)$. Hence, there exists a monodromy operator
$f\in Mon^2(X)$, such that $fg(\alpha')=\alpha$, by Lemma \ref{lemma-classification}. 
Choose a marking $\eta:H^2(X,\Integers)\rightarrow \Lambda$.
Then $\eta':=\eta\circ f\circ g$ is a marking of $X'$ satisfying $\eta(\alpha)=\eta'(\alpha')$. 
Hence, the triples $(X,\alpha,\eta)$ and $(X',\alpha',\eta')$ both belong to the 
moduli space $\FM^0_{\eta(\alpha)^\perp}$, by Proposition
\ref{prop-same-connected-component}. 
$\FM^0_{\eta(\alpha)^\perp}$ is connected, by Lemma 
\ref{lemma-connected}.
Hence, $(X,\LB)$ and $(X',\LB')$ are deformation equivalent.
\end{proof}

\begin{rem}
\label{rem-Tate-Shafarevich-lines-are-limits-of-twistor-lines}
Tate-Shafarevich lines (Definition \ref{def-tate-shafarevich-lines})
are limits of twistor lines in the following sense. 
Let $\ell$ be a point of $\Omega_\Lambda$ and  
$\omega$ a class in the positive cone $\C_\ell$ in
$\Lambda^{1,1}(\ell,\RealNumbers)$. Assume that $\omega$ 
is not orthogonal to any class in $\Lambda^{1,1}(\ell,\Integers)$.
Then there exists a marked pair $(X,\eta)$ in each connected component
$\FM^0_\Lambda$ of the moduli space of marked pairs, such that $P(X,\eta)=\ell$ and 
$\eta^{-1}(\omega)$ is a K\"{a}hler class of $X$ \cite[Cor. 5.7]{huybrechts-basic-results}. 
Set
$W':=\ell\oplus\bar{\ell}\oplus \ComplexNumbers\omega$.
$\PP(W')\cap \Omega_\Lambda$ is a {\em twistor line} for $(X,\eta)$; it admits 
a canonical lift to a smooth rational curve in $\FM^0_\Lambda$ containing the point $(X,\eta)$
\cite[Cor. 5.8]{huybrechts-basic-results}. This lift corresponds to an action of the quaternions $\HH$ on 
the real tangent bundle of the differentiable manifold $X$, such that the unit quaternions act as 
integrable complex structures, one of which is the complex structure of $X$.
Let $\alpha\in \Lambda$ be the primitive isotropic class
as above. Assume that $\ell$ belongs to $\Omega_{\alpha^\perp}^+$.
Consider the three dimensional subspace $W:=\ell\oplus\bar{\ell}\oplus \ComplexNumbers\alpha$ of 
$H^2(X,\ComplexNumbers)$. 
%and set $W_\RealNumbers:=W\cap H^2(X,\RealNumbers)$.
%Then $W_\RealNumbers$ is a limit of positive definite three-dimensional subspaces of 
%$H^2(X,\RealNumbers)$, 
Then $W$ is a limit of a sequence of three dimensional subspaces $W'_i$, associated to some sequence of 
classes $\omega_i$ as above, 
since $\alpha$ belongs to the boundary of the positive cone $\C_\ell$. 
Now $W$ is contained in $\alpha^\perp$, and so 
$
\PP(W)\cap \Omega_{\alpha^\perp} = \PP(W)\cap \Omega_\Lambda.
$
In this degenerate case, the conic $\PP(W)\cap \Omega_\Lambda$ consists of two irreducible
components, the Tate-Shafarevich line 
$\PP[\ell\oplus \ComplexNumbers\alpha]\setminus \{\PP[\ComplexNumbers\alpha]\}$ in
$\Omega_{\alpha^\perp}^+$
and the line $\PP[\bar{\ell}\oplus \ComplexNumbers\alpha]\setminus \{\PP[\ComplexNumbers\alpha]\}$
in the other connected component $\Omega_{\alpha^\perp}^-$ of $\Omega_{\alpha^\perp}$. 
Theorem \ref{thm-Tate-Shafarevich-orbit} will provide a lift of a generic Tate-Shafarevich line 
in the period domain to a line in the moduli space of marked pairs.
\end{rem}

\begin{center}{\bf A summary of notation related to lattices and period domains}\end{center}
\begin{tabular}{|c|l|}  
\hline 
$U$ & The rank 2 even unimodular lattice of signature $(1,1)$.
\\
\hline
$E_8(-1)$ & The root lattice of type $E_8$ with a negative definite pairing.
\\
\hline 
$\widetilde{\Lambda}$ & The Mukai lattice; the orthogonal direct sum $U^{\oplus 4}\oplus E_8(-1)^{\oplus 2}$.
\\
\hline
$\Lambda$ & The $K3^{[n]}$-lattice; the orthogonal direct sum $U^{\oplus 3}\oplus E_8(-1)^{\oplus 2}\oplus \langle 2-2n\rangle$, 
\\
& 
where $\langle 2-2n\rangle$ is the rank $1$ lattice generated by a class of self-intersection $2-2n$.
\\
\hline
$\alpha$ & A primitive isotropic class in $\Lambda$.
\\
\hline
$Q_\alpha$ & The subquotient $\alpha^\perp/\Integers\alpha$.
\\
\hline
$\iota$ & A primitive embedding of $\Lambda$ in $\widetilde{\Lambda}$.
\\
\hline
$\beta$ & The primitive isotropic class $\iota(\alpha)$ in $\widetilde{\Lambda}$.
\\
\hline
$\Lambda_{k3}$ & The subquotient $\beta^\perp/\Integers\beta$, which is isomorphic to the 
$K3$ lattice $U^{\oplus 3}\oplus E_8(-1)^{\oplus 2}$.
\\
\hline
$v$ & A generator of the rank $1$ sublattice of $\widetilde{\Lambda}$ orthogonal to $\iota(\Lambda)$.
\\
\hline
$\bar{v}$ & The coset $v+\Integers \beta$ in  $\Lambda_{k3}$.
\\
\hline
$d$ & The divisibility of $(\alpha,\bullet)$ in $\Lambda^*$; $d:=\gcd\{(\alpha,\lambda) \ : \ \lambda\in \Lambda\}.$
\\
\hline
$\xi$ & The integral element $(1/d)\bar{v}$ of $\Lambda_{k3}$. We have $(\xi,\xi)=\frac{2n-2}{d^2}$.
\\
\hline
$\Omega_\Lambda$ & The period domain given in (\ref{eq-period-domain}).
\\
\hline
$\widetilde{\C}_\Lambda$ & The positive cone given in (\ref{eq-positive-cone-in-Lambda}).
\\
\hline
$\Omega_\Lambda^+$ & The connected component of $\Omega_\Lambda$ determined by the orientation 
of $\widetilde{\C}_\Lambda$.
\\
\hline
$\Omega_{\alpha^\perp}$ & The hyperplane section of $\Omega_\Lambda$ given in 
(\ref{hyperplane-section-of-period-domain-orthogonal-to-alpha}).
\\
\hline
$\Omega_{\alpha^\perp}^+$ & The connected component of $\Omega_{\alpha^\perp}$ 
given in (\ref{eq-connected-component-of-Omaga-alpha-perp}).
\\
\hline
$\Omega_{Q_\alpha}$ & The period domain of the lattice $Q_\alpha$.
\\
\hline
$q$ & The fibration $q:\Omega_{\alpha^\perp}\rightarrow \Omega_{Q_\alpha}$ by Tate-Shafarevich lines
given in (\ref{eq-q}).
\\
\hline
$\FM^0_\Lambda$ & A connected component of the moduli space of marked pairs.
\\
\hline
$P_0$ & The period map $P_0:\FM^0_\Lambda\rightarrow \Omega_\Lambda^+$.
\\
\hline
$\FM^0_{\alpha^\perp}$ & The inverse image of $\Omega_{\alpha^\perp}^+$ in $\FM^0_\Lambda$ via $P_0$.
\\
\hline
$[L,v](\alpha)$ & The monodromy invariant associated to the class $\alpha$ in 
Lemma \ref{lemma-classification} (\ref{lemma-item-integer-b-determines-isometry-class}).
\\
\hline
\end{tabular}

%***************************************************************************
%
%***************************************************************************
\section{Density of periods of relative compactified Jacobians}
\label{sec-density-of-periods-of-compactified-jacobians}
We keep the notation of section \ref{sec-period-domains}.
%Following is an outline of this section.
In subsection \ref{sec-the-section-tau-of-q} we construct a section 
$\tau:\Omega_{Q_\alpha}^+\rightarrow \Omega_{\alpha^\perp}^+$, given in 
(\ref{eq-section-of-q}),
of the fibration $q: \Omega_{\alpha^\perp}^+\rightarrow \Omega_{Q_\alpha}^+$
by Tate-Shafarevich lines. We then show that $\tau$ maps a period $\underline{\ell}$, of
a semi-polarized $K3$ surface $(S,\B)$ 
in the period domain $\Omega_{Q_\alpha}^+$, 
to the period $\tau(\underline{\ell})$ of a moduli space $M$ of sheaves on $S$ with pure one-dimensional support in
the linear system $\linsys{\B^d}$. 
The moduli space $M$ admits a Lagrangian fibration over $\linsys{\B^d}$.
In subsection \ref{sec-monodromy-invariance-of-fibration-q} we construct an injective homomorphism 
$g:Q_\alpha\rightarrow O(\Lambda)$, whose image is contained in the subgroup of the monodromy group
which stabilizes $\alpha$. We get an action of $Q_\alpha$ on the period domain $\Omega^+_{\alpha^\perp}$,
which lifts to an action on  connected components $\FM^0_{\alpha^\perp}$ of the moduli space of marked pairs given in Equation
(\ref{eq-FM-alpha-perp}). We then show that the fibration $q$ by Tate-Shafarevich lines is $g(Q_\alpha)$-invariant.
In subsection \ref{sec-density} we prove that the $g(Q_\alpha)$-orbit of every point in a
non-special Tate-Shafarevich line is dense in that line. Consequently, the non-special  Tate-Shafarevich line 
$q^{-1}(\underline{\ell})$ contains the dense orbit $g(Q_\alpha)\tau(\underline{\ell})$ of periods of 
marked pairs in $\FM^0_{\alpha^\perp}$ admitting a Lagrangian fibration.

{\bf Conventions:}
%We have the lattices $\Lambda$ and $\widetilde{\Lambda}$,  a primitive
%isometric embedding\footnote{
The discussion in the current section 
\ref{sec-density-of-periods-of-compactified-jacobians} concerns only 
period domains, so we are free to choose the embedding $\iota$.
When we consider in subsequent sections a component $\FM_{\Lambda}^0$ of the moduli space of marked pairs 
$(X,\eta)$ of $K3^{[n]}$-type, together
with such an embedding $\iota:\Lambda\rightarrow \widetilde{\Lambda}$,
we will always assume that $\iota$ is chosen so that 
$\iota\circ\eta$ belongs to the canonical $O(\widetilde{\Lambda})$-orbit $\iota_X$ of Theorem
\ref{thm-orbit-iota-X}, for all
$(X,\eta)$ in $\FM_{\Lambda}^0$.
%} 
%$\iota:\Lambda\rightarrow \widetilde{\Lambda}$, 
%and primitive isotropic classes $\alpha\in \Lambda$ and $\beta:=\iota(\alpha)$. We set
%$\Lambda_{k3}:=\beta^\perp_{\widetilde{\Lambda}}/\Integers\beta$.
%The class $v$ in $\widetilde{\Lambda}$ is primitive and orthogonal to the image of
%$\iota$. We set $\bar{v}:=v+\Integers\beta \in \Lambda_{k3}$.
We choose the orientation of the positive cone $\widetilde{\C}_\Lambda$ of $\Lambda$, so that $\alpha$
belongs to the boundary of the positive cone in
$\Lambda^{1,1}(\ell,\RealNumbers)$, for every $\ell\in \Omega_{\alpha^\perp}^+$.
We choose the orientation of the positive cone $\widetilde{\C}_{\Lambda_{k3}}$,
so that $\bar{v}$ belongs to the positive cone in
$\Lambda^{1,1}_{k3}(\uell,\RealNumbers)$, for every $\uell\in\Omega_{\bar{v}^\perp}^+$.
%We set $d:=\div(\alpha,\bullet)$ and choose an integer $b$, such that $\beta-bv$ is divisible by $d$.
%Such an integer $b$ exists, by Lemma \ref{lemma-classification}.
%We set $Q_\alpha:=\alpha^\perp_\Lambda/\Integers\alpha$. 
Note that the composition
$\alpha^\perp_\Lambda\RightArrowOf{\iota}\beta^\perp_{\widetilde{\Lambda}}\rightarrow \Lambda_{k3}$ 
induces an isometry from 
$Q_\alpha:=\alpha^\perp_\Lambda/\Integers\alpha$ onto $\bar{v}^\perp_{\Lambda_{k3}}$,
by Lemma \ref{lemma-Lambda-alpha-is-isometric-to-v-bar-perp}. 
The choice of orientation of the positive cone of $\Lambda_{k3}$ determines an orientation of the
positive cone of $Q_\alpha$.
%The class $\xi:=\frac{1}{d}\bar{v}$ is integral and satisfies $(\xi,\xi)=\frac{2n-2}{d^2}.$

%***************************************************************************
%
%***************************************************************************
\subsection{A period of a Lagrangian fibration in each Tate-Shafarevich line}
\label{sec-the-section-tau-of-q}
Choose a class $\gamma$ in $\widetilde{\Lambda}$ satisfying $(\gamma,\beta)=-1$ and $(\gamma,\gamma)=0$.
Note that $\beta$ and $\gamma$ span a unimodular sub-lattice of $\widetilde{\Lambda}$ of signature $(1,1)$.
%We do not impose any condition on the Hodge-type of $\gamma$. 
We construct next a section of the affine bundle
$q:\Omega_{\alpha^\perp}\rightarrow \Omega_{Q_\alpha}$, given in Equation (\ref{eq-q}), in terms of 
$\gamma$. We have the following split short exact sequence.
\begin{equation}
\label{eq-split-exact-sequence-of-sigma-gamma}
\xymatrix{
0 \ar[r] &
\Integers\beta \ar[r] &
\beta^\perp_{\widetilde{\Lambda}} \ar[r]^{j}
\ar@/^2pc/[l] _{\sigma_\gamma} &
\Lambda_{k3} \ar[r]
\ar@/^2pc/[l] _{\tilde{\tau}_\gamma} &
0.
}
\end{equation}
Above, $\sigma_\gamma(x)=-(x,\gamma)\beta$, and
$\tilde{\tau}_\gamma(y)=\tilde{y}+(\tilde{y},\gamma)\beta$, where $\tilde{y}$ is any element of
$\beta^\perp_{\widetilde{\Lambda}}$ satisfying $j(\tilde{y})=y$.
One sees that $\tilde{\tau}_\gamma$ is well defined as follows.
If $\tilde{y}_1$ and $\tilde{y}_2$ satisfy $j(\tilde{y}_k)=y$, then the difference
$[\tilde{y}_1+(\tilde{y_1},\gamma)\beta]-[\tilde{y}_2+(\tilde{y_2},\gamma)\beta]$
belongs to the kernel of $j$ and is sent to $0$ via $\sigma_\gamma$, so the difference is equal to $0$.
Note that  $\tilde{\tau}_\gamma$ is an isometric embedding and its image 
is precisely
$\{\beta,\gamma\}^\perp_{\widetilde{\Lambda}}$.

We regard $\Omega_{Q_\alpha}^+$ as the period domain for semi-polarized $K3$ surfaces,
with a nef line bundle of degree $\frac{2n-2}{d^2}$, via the isomorphism $\bar{v}^\perp_{\Lambda_{k3}}\cong Q_\alpha$
of Lemma \ref{lemma-Lambda-alpha-is-isometric-to-v-bar-perp}.
The homomorphism $\iota^{-1}\circ\tilde{\tau}_\gamma$ induces an isometric 
embedding of $Q_\alpha$ in $\alpha^\perp_\Lambda$.
We get a section
\begin{equation}
\label{eq-section-of-q}
\tau_\gamma \ : \ \Omega_{Q_\alpha}^+\rightarrow \Omega_{\alpha^\perp}^+
\end{equation}
of $q:\Omega_{\alpha^\perp}^+\rightarrow \Omega_{Q_\alpha}^+.$
Following is an explicit description of $\tau_\gamma$.
Let $\underline{\ell}$ be a period in $\Omega_{Q_\alpha}^+$. 
Choose a period $\ell$ in $\Omega_{\alpha^\perp}^+$ satisfying $q(\ell)=\underline{\ell}$.
Let $x$ be a non-zero element of the line $\ell$ in 
$\alpha^\perp_\Lambda\otimes_\Integers\ComplexNumbers$.
Then 
\begin{equation}
\label{eq-bar-tau-gamma}
\tau_\gamma(\underline{\ell}) = {\rm span}_\ComplexNumbers\{x+(\iota(x),\gamma)\alpha\}.
\end{equation}
We see that $\gamma$ belongs to $\widetilde{\Lambda}^{1,1}(\tau_\gamma(\underline{\ell}))$,
for every $\underline{\ell}$ in $\Omega_{Q_\alpha}^+$.

Fix a period $\uell$ in $\Omega_{Q_\alpha}^+$.
We construct next a marked pair $(M_H(u),\eta_1)$ with period $\tau_\gamma(\underline{\ell})$,
such that $\eta_1^{-1}(\alpha)$ induces a Lagrangian fibration.
Let $S$ be a $K3$ surface and $\eta:H^2(S,\Integers)\rightarrow \nolinebreak\Lambda_{k3}$ a marking, 
such that the period $\eta(H^{2,0}(S))$ 
is $\underline{\ell}$. Such a marked pair $(S,\eta)$ exists, by the surjectivity of the period map.
Extend $\eta$ to the Hodge isometry
\[
\tilde{\eta} \ : \ H^*(S,\Integers) \rightarrow \widetilde{\Lambda},
\]
given by $\tilde{\eta}((0,0,1))=\beta$, $\tilde{\eta}((1,0,0))=\gamma$, and
$\tilde{\eta}$ restricts to $H^2(S,\Integers)$ as $\tilde{\tau}_\gamma\circ \eta$.
We have the equality
$v=\sigma_\gamma(v)+\tilde{\tau}_\gamma(\bar{v})=-(\gamma,v)\beta+\tilde{\tau}_\gamma(\bar{v})$.
Set $a:=-(\gamma,v)$ and $u:=(0,\eta^{-1}(\bar{v}),a)$. Then
$\tilde{\eta}(u)=v$.
We may choose the marking $\eta$ so that the class $\eta^{-1}(\bar{v})$
is nef, possibly after replacing $\eta$ by  $\pm\eta\circ w$,
where $w$ is an element of the group of isometries of $H^2(S,\Integers)$,
generated by reflections by $-2$ curves \cite[Ch. VIII Prop. 3.9]{BHPV}.
Choose a $u$-generic polarization $H$ of $S$.
Then $M_H(u)$ is a projective irreducible holomorphic symplectic manifold. 
Let
\[
\theta:u^\perp\rightarrow H^2(M_H(u),\Integers)
\]
be Mukai's isometry, given in Equation (\ref{eq-Mukai-isomorphism}). We get the commutative diagram:
\begin{equation}
\label{diagram-marked-M-H-u}
\xymatrix{
\Lambda \ar[r]^\iota_\cong &
v^\perp \ar[r]^{\subset} &
\widetilde{\Lambda}
\\
H^2(M_H(u),\Integers) \ar[r]^-{\theta^{-1}} \ar[u]^{\eta_1} &
u^\perp  \ar[r]^-{\subset} \ar[u]^{\eta_2}&
H^*(S,\Integers), \ar[u]^{\tilde{\eta}}
}
\end{equation}
where $\eta_2$ is the restriction of $\tilde{\eta}$ and
$\eta_1=\iota^{-1}\circ \eta_2\circ \theta^{-1}$.
Note that $\eta_1(\theta(0,0,1))=\alpha$. Let $L$ be the saturation in $H^*(S,\Integers)$
of the sub-lattice spanned by $(0,0,1)$ and $u$. 
Let $b$ be an integer satisfying $ab\equiv 1$ (modulo $d$).
The monodromy invariant $[L,u](\theta(0,0,1))$ of Lemma \ref{lemma-classification}
is the isometry class of the pair 
$(L_{n,d},(d,b))$, by the commutativity of the above diagram.
Furthermore, $\eta_1$ is a Hodge isometry with respect to the 
Hodge structure on $\Lambda$ induced by $\tau_\gamma(\underline{\ell})$.
In particular, $(M_H(u),\eta_1)$ is a marked pair with period $\tau_\gamma(\underline{\ell})$.
Example 
\ref{ex-completely-integrable-moduli-spaces-for-each-value-of-the-monodromy-invariant} exhibits 
$\theta(0,0,1)$ as the class $\pi^*c_1(\StructureSheaf{\linsys{\LB^d}}(1))$,
for a Lagrangian fibration $\pi:M_H(u)\rightarrow \linsys{\LB^d}$, where $\LB$ is the line bundle over $S$
with class $\eta^{-1}(\xi)$. 

\begin{rem}
\label{rem-eta-1-is-compatible-with-the-orientations}
The isometry $\eta_1$ is compatible with the orientations of the positive cones, the canonical one of 
$H^2(M_H(u),\Integers)$ and the chosen one of $\Lambda$. Indeed, it maps the class $\theta(0,0,1)$,
on the boundary of the positive cone of $H^{1,1}(M_H(u),\RealNumbers)$, 
to the class $\alpha$ on the boundary of the positive cone of 
$\Lambda^{1,1}(\tau_\gamma(\underline{\ell}),\RealNumbers)$. 
The composition $\tilde{\eta}\circ \theta^{-1}$ in Diagram
(\ref{diagram-marked-M-H-u})
belongs to the canonical orbit $\iota_{M_H(u)}$ of Theorem \ref{thm-orbit-iota-X},
by \cite[Theorem 1.14]{markman-constraints}.
The commutativity of the Diagram implies that the isometric embedding $\iota\circ \eta_1$
also belongs to the orbit $\iota_{M_H(u)}$.
\end{rem}

%***************************************************************************
%
%***************************************************************************
\subsection{Monodromy equivariance of the fibration by Tate-Shafarevich lines}
\label{sec-monodromy-invariance-of-fibration-q}
Denote by $O(\widetilde{\Lambda})_{\beta,v}^+$ the subgroup of $O(\widetilde{\Lambda})^+$
stabilizing both $\beta$ and $v$. 
Following is a natural homomorphism 
\begin{equation}
\label{eq-h}
h:O(\widetilde{\Lambda})_{\beta,v}^+\rightarrow O(\Lambda_{k3})_{\bar{v}}.
\end{equation} 
If $\psi$ belongs to $O(\widetilde{\Lambda})_{\beta,v}^+$, then 
$\psi(\beta)=\beta$ and $\beta^\perp_{\widetilde{\Lambda}}$ is $\psi$-invariant. 
Thus $\psi$ induces an isometry $h(\psi)$
of $\Lambda_{k3}:=\beta^\perp_{\widetilde{\Lambda}}/\Integers\beta$.
We construct next a large subgroup in the kernel of $h$.

Given an element $z$ of $\widetilde{\Lambda}$, orthogonal to $\beta$ and $v$, define the map
$\tilde{g}_z:\widetilde{\Lambda}\rightarrow \widetilde{\Lambda}$ by
\[
\tilde{g}_z(x) \ \ := \ \ x-(x,\beta)z+
\left[(x,z)-\frac{1}{2}(x,\beta)(z,z)
\right]\beta.
\]

\begin{lem}
\label{lemma-characterization-of-g-z}
The map $\tilde{g}_z$ is the unique isometry in $O(\widetilde{\Lambda})_{\beta,v}$, 
which sends $\gamma$ to an element of $\widetilde{\Lambda}$ congruent to $\gamma+z$ modulo
$\Integers\beta$ and belongs to the kernel of $h$.
The isometry $\tilde{g}_z$ is orientation preserving.
\end{lem}

\begin{proof}
We first define an isometry $f$ with the above property, then prove its uniqueness, and finally  
prove that it is equal to $\tilde{g}_z$. 
Set $\gamma_1:=\gamma+z+\left[(\gamma,z)+\frac{1}{2}(z,z)\right]\beta$.
Then $(\gamma_1,\gamma_1)=0$, $(\gamma_1,\beta)=-1$, and $\gamma_1$ is the unique element
of $\widetilde{\Lambda}$ satisfying the above equalities and congruent to $\gamma+z$ modulo $\Integers\beta$.
Define $\tilde{\sigma}_\gamma:\widetilde{\Lambda}\rightarrow \Integers\beta +\Integers\gamma$ by 
$\tilde{\sigma}_\gamma(x):=-(x,\beta)\gamma-(x,\gamma)\beta$. We get the commutative diagram with 
split short exact rows:
\[
\xymatrix{
0 \ar[r] &
\Integers\beta +\Integers\gamma \ar[r] \ar[d] &
\widetilde{\Lambda} \ar[r]^{\tilde{j}} \ar[d]^f
\ar@/_2pc/[l] _{\tilde{\sigma}_\gamma} &
\Lambda_{k3} \ar[r] \ar[d]^{id}
\ar@/_2pc/[l] _{\tilde{\tau}_\gamma} &
0
\\
0 \ar[r] &
\Integers\beta +\Integers\gamma_1\ar[r] &
\widetilde{\Lambda} \ar[r]^{\tilde{j}_1}
\ar@/^2pc/[l] _{\tilde{\sigma}_{\gamma_1}} &
\Lambda_{k3} \ar[r]
\ar@/^2pc/[l] _{\tilde{\tau}_{\gamma_1}} &
0
}
\]
Above $\tilde{\tau}_\gamma$ and $j$ 
are the homomorphisms given in equation (\ref{eq-split-exact-sequence-of-sigma-gamma}), 
$\tilde{j}(x)=j(x+(x,\beta)\gamma)$, 
and $\tilde{\sigma}_{\gamma_1}$, $\tilde{\tau}_{\gamma_1}$, and $\tilde{j}_1$ are defined similarly,
replacing $\gamma$ by $\gamma_1$. The map $f$ is defined by $f(\beta)=\beta$,
$f(\gamma)=\gamma_1$, and $f(\tilde{\tau}_\gamma(y))=\tilde{\tau}_{\gamma_1}(y)$.
Then $f$ is clearly an isometry. 

The isometry $f$ can be extended to an isometry of $\widetilde{\Lambda}_\RealNumbers$ 
and we can continuously deform $z$  to $0$ in $\{\beta,v\}^\perp\otimes_\Integers\RealNumbers$, 
resulting in a continuous deformation of $f$ to the identity. Hence, $f$ is orientation preserving.

Note the equalities $\tilde{\sigma}_\gamma(v)=-(v,\gamma)\beta=-(v,\gamma_1)\beta=\tilde{\sigma}_{\gamma_1}(v)$,
where the middle one follows from that fact that both $z$ and $\beta$ are orthogonal to $v$. We get
the equality
\[
\tilde{\tau}_\gamma(\bar{v}) = v-\tilde{\sigma}_\gamma(v)=v-\tilde{\sigma}_{\gamma_1}(v)=\tilde{\tau}_{\gamma_1}(\bar{v}).
\]
Thus $f(v)=v$ and $f$ belongs to $O(\widetilde{\Lambda})^+_{\beta,v}$.
Let $x$ be an element of $\beta^\perp$. Then $\tilde{j}(x)=j(x)=\tilde{j}_1(x)$. Set $y:=j(x)$.
Now $\tilde{\tau}_\gamma(y)\equiv \tilde{\tau}_{\gamma_1}(y)$ modulo $\Integers\beta$, by definition
of both. Hence, $h(f)$ is the identity isometry of $\Lambda_{k3}$. 

Let $f'$ be another isometry of $\widetilde{\Lambda}$ satisfying the assumptions of the Lemma.
Then $f'(\gamma)=\gamma_1$, by the characterization of $\gamma_1$ mentioned above.
Set $e:=f^{-1}\circ f'$. Then $e(\beta)=\beta$, $e(\gamma)=\gamma$, $e(v)=v$, 
and $h(e)=id$. Given $x\in \beta^\perp$, we get that
$e(x)\equiv x$ modulo $\Integers\beta$. 
Now $(e(x),\gamma)=(e(x),e(\gamma))=(x,\gamma)$. Thus, $e$ restricts to the identity on $\beta^\perp$.
We conclude that $e$ is the identity of $\widetilde{\Lambda}$, as the latter is spanned by $\gamma$
and $\beta^\perp$. Thus $f'=f$.

It remains to prove the equality $f=\tilde{g}_z$. We already know that
$f(\gamma)=\gamma_1=\tilde{g}_z(\gamma)$ and
$f(\beta)=\beta=\tilde{g}_z(\beta)$. Given $y\in \Lambda_{k3}$, we have
\[
\tilde{g}_z(\tilde{\tau}_\gamma(y))=
\tilde{\tau}_\gamma(y)+(\tilde{\tau}_\gamma(y),z)\beta=
\tilde{\tau}_{\gamma_1}(y)=f(\tilde{\tau}_\gamma(y)).
\]
Hence, $\tilde{g}_z=f$.
\end{proof}

%One easily checks that $\tilde{g}_z$ 
%is an orientation preserving 
%isometry satisfying $\tilde{g}_z(\beta)=\beta$ and $\tilde{g}_z(v)=v$. 
%Furthermore, $\tilde{g}_z$ belongs to the kernel of $h$.
Let 
\[
\tilde{g}:\alpha^\perp_{\Lambda}
%\RightArrowOf{\iota}L^\perp_{\widetilde{\Lambda}}
\rightarrow O(\widetilde{\Lambda})_{\beta,v}^+
\]
be the map sending $z$ to $\tilde{g}_{\iota(z)}$.
Denote by $Mon^2(\Lambda,\iota)$
the subgroup of $O^+(\Lambda)$ of isometries stabilizing the orbit 
$O(\widetilde{\Lambda})\iota$. Note that $O(\widetilde{\Lambda})^+_v$ is conjugated via $\iota$
onto $Mon^2(\Lambda,\iota)$, if $n=2$, and to an index $2$ subgroup of $Mon^2(\Lambda,\iota)$,
if $n\geq 2$ \cite[Lemma 4.10]{markman-monodromy-I}. 
Let $Mon^2(\Lambda,\iota)_\alpha$ be the subgroup of $Mon^2(\Lambda,\iota)$
stabilizing $\alpha$.

\begin{lem}
\label{lemma-q-is-Q-alpha-invariant}
\begin{enumerate}
\item
\label{lemma-item-g-is-a-homomorphism}
The map $\tilde{g}$ is a group homomorphism with kernel $\Integers\alpha$. It thus factors
through an injective homomorphism
\[
g:Q_\alpha
%:=L^\perp_{\widetilde{\Lambda}}/\Integers\beta 
\rightarrow  Mon^2(\Lambda,\iota)_\alpha.
\]
\item
\label{lemma-item-lattice-action-on-alpha-perp}
Let $z$ be an element of $\alpha^\perp_\Lambda$ and $[z]$ its coset in $Q_\alpha$.
Then $g_{[z]}:\alpha^\perp\rightarrow\alpha^\perp$ sends 
$x\in \alpha^\perp$ to $x+(x,z)\alpha$.
\item
\label{lemma-item-q-is-equivariant}
The map $q:\Omega_{\alpha^\perp}^+\rightarrow \Omega_{Q_\alpha}^+$ is
$Mon^2(\Lambda,\iota)_\alpha$-equivariant and it is invariant with respect to
the image $g(Q_\alpha)\subset Mon^2(\Lambda,\iota)_\alpha$ of $g$. 
\item
\label{lemma-item-ker-h-equals-im-g}
The image of $\tilde{g}$ is equal to the kernel of the homomorphism $h$,
given in Equation (\ref{eq-h}).
\end{enumerate}
\end{lem}

\begin{proof}
Part (\ref{lemma-item-g-is-a-homomorphism}) follows from the characterization of $\tilde{g}_z$ in 
Lemma \ref{lemma-characterization-of-g-z}.
Part (\ref{lemma-item-lattice-action-on-alpha-perp}) is straightforward 
as is the $Mon^2(\Lambda,\iota)_\alpha$-equivariance of $q$. 
The $g(Q_\alpha)$-invariance of $q$ follows from
part (\ref{lemma-item-lattice-action-on-alpha-perp}). Part (\ref{lemma-item-q-is-equivariant}) is thus proven.

Part (\ref{lemma-item-ker-h-equals-im-g}): 
The image of $\tilde{g}$ is contained in the kernel of $h$, by Lemma \ref{lemma-characterization-of-g-z}.
Let $f\in O(\widetilde{\Lambda})_{\beta,v}$ belong to the kernel of $h$.
Set $\gamma_1:=f(\gamma)$ and $z:=\gamma_1-\gamma$. 
Then $(\gamma_1,\beta)=(f(\gamma),\beta)=(f(\gamma),f(\beta))=(\gamma,\beta)$ and similarly
$(\gamma_1,v)=(\gamma,v)$. Hence, $(z,\beta)=0$ and $(z,v)=0$.
The isometry $\tilde{g}_z$ is thus well defined and it is equal to $f$, by Lemma
\ref{lemma-characterization-of-g-z}.
%*************
% Hide
%*************
\hide{
Let $\mu\in O^+(\widetilde{\Lambda})_{\beta,v}$ be an element in the kernel of $h$. 
The image of $\tilde{\tau}_\gamma$ is contained in $\beta^\perp_{\widetilde{\Lambda}}$,
and the latter is $\mu$-invariant.
Hence, there exists some function $f:\Lambda_{k3}\rightarrow\Integers$,
such that 
$\mu(\tilde{\tau}_\gamma(\lambda))=\tilde{\tau}_\gamma(\lambda)+f(\lambda)\beta$, 
for all $\lambda\in\Lambda_{k3}$. The function $f$ must be linear.
There exists an element $z\in \Lambda_{k3}$, such that 
$f(\lambda)=(z,\lambda)$, since $\Lambda_{k3}$ is unimodular.  

Let $j:\tilde{\Lambda}\rightarrow \Lambda_{k3}$ be the homomorphism
given by $j(\tlambda)=\tlambda+(\tlambda,\beta)\gamma+\Integers\beta$.
Then $j$ is the adjoint of $\tilde{\tau}_\gamma$.
Given $\tlambda\in \tilde{\Lambda}$, we have 
$\tilde{\tau}_\gamma(j(\tlambda))=\tlambda+(\tlambda,\beta)\gamma+(\tlambda,\gamma)\beta$.
Hence, 
\begin{eqnarray*}
\mu(\tlambda)&=&\mu\left(\tilde{\tau}_\gamma(j(\tlambda))-(\tlambda,\beta)\gamma-(\tlambda,\gamma)\beta\right)
\\
&=& \tilde{\tau}_\gamma(j(\tlambda))+(z,j(\tlambda))\beta-(\tlambda,\beta)\mu(\gamma)-(\tlambda,\gamma)\beta
\\
&=&\tlambda+(\tlambda,\beta)(\gamma-\mu(\gamma))+
(z,j(\tlambda))\beta
\\
&=& \tlambda+(\tlambda,\beta)(\gamma-\mu(\gamma))+
(\tilde{z},\tlambda)\beta,
\end{eqnarray*}
where $\tilde{z}:=\tilde{\tau}_\gamma(z)$.
The equality $(\tlambda,\gamma)=(\mu(\tlambda),\mu(\gamma))$ yields
\begin{equation}
\label{eq-tlambda-paired-with-gamma}
(\tlambda,\gamma)=
(\tlambda,\mu(\gamma))+
(\tlambda,\beta)(\gamma-\mu(\gamma),\mu(\gamma))+
(\tilde{z},\tlambda)(\beta,\mu(\gamma)).
\end{equation}
We conclude that $(\tlambda,\mu(\gamma))=0$, for all
$\tlambda$ in $\{\beta,\gamma,\tilde{z}\}^\perp$. 
Hence, $\mu(\gamma)$ belongs to the saturation of 
${\rm span}_\Integers\{\beta,\gamma,\tilde{z}\}$.

Write $\mu(\gamma)=a\tilde{z}+b\beta+c\gamma$. 
Note that $\tilde{z}$ is in the image $\{\beta,\gamma\}^\perp$ of $\tilde{\tau}_\gamma$.
Substitute $\beta$ for $\tlambda$ in (\ref{eq-tlambda-paired-with-gamma})
to get $c=1$. 
We also have
\[
0=(\gamma,\gamma)=(\mu(\gamma),\mu(\gamma))=a^2(z,z)-2b.
\]
So $b=\frac{a^2}{2}(z,z)$.
Substitute $\tilde{z}$ for $\tlambda$ in (\ref{eq-tlambda-paired-with-gamma})
to get the equality $(a-c)(z,z)=0$.
If $(z,z)\neq 0$, then $a=c=1$.
If $(z,z)=0$, then $b=0$ and Equation (\ref{eq-tlambda-paired-with-gamma}) simplifies to
\[
(\tlambda,\gamma)=(\tlambda,\tilde{z})(a-1)+(\tlambda,\gamma).
\]
We get the equality $a=1$ in this case as well.
We see that $\mu(\gamma)=\tilde{z}+\gamma+\frac{1}{2}(z,z)\beta$,
and so $\mu(\tlambda)=\tilde{g}_{\tilde{z}}(\tlambda)$, for all $\tlambda\in\widetilde{\Lambda}$.
Consequently, $v=\mu(v)=v+(v,\tilde{z})\beta$. We conclude that $(v,\tilde{z})=0$ and so
$\mu$ belongs to the image of $\tilde{g}$.
%*************
% End Hide
%*************
}
\end{proof}

%***************************************************************************
%
%***************************************************************************
\subsection{Density}
\label{sec-density}
A period $\underline{\ell}$ in $\Omega_{\Lambda_{k3}}$  is said to be {\em special},
if it satisfies the  condition analogous to the one in Definition
\ref{def-special}. We identify $\Omega_{Q_\alpha}$ as a submanifold of $\Omega_{\Lambda_{k3}}$,
via Lemma \ref{lemma-Lambda-alpha-is-isometric-to-v-bar-perp}.
Note that a period $\ell\in \Omega_{\alpha^\perp}$ is special, if and only if the period $q(\ell)$ is.

\begin{lem}
\label{lemma-density-in-a-generic-fiber}
\begin{enumerate}
\item
\label{lemma-item-non-special-iff-dense}
$g(Q_\alpha)$ has a dense orbit in $q^{-1}(\underline{\ell})$, 
if and only if $\underline{\ell}$ is non-special.
\item
\label{lemma-item-if-one-orbit-is-dense-then-every}
If $g(Q_\alpha)$ has a dense orbit in $q^{-1}(\underline{\ell})$, 
then every $g(Q_\alpha)$-orbit in $q^{-1}(\underline{\ell})$ is dense.
\end{enumerate}
\end{lem}

\begin{proof}
Part \ref{lemma-item-if-one-orbit-is-dense-then-every} follows from the description of the action in 
Lemma \ref{lemma-q-is-Q-alpha-invariant} part \ref{lemma-item-lattice-action-on-alpha-perp}.
We prove part \ref{lemma-item-non-special-iff-dense}.
Fix a period $\ell$ such that $q(\ell)=\underline{\ell}$ and choose
a non-zero element $t$ of the line $\ell$ in
$\alpha^\perp_\Lambda\otimes_\Integers\ComplexNumbers$. Then 
$q^{-1}(\underline{\ell})=\PP[\ComplexNumbers\alpha+\ComplexNumbers t]\setminus\{\PP[\ComplexNumbers\alpha]\}$
and $g_{[z]}(a\alpha+t)=(a+(t,z))\alpha+t$, by
Lemma \ref{lemma-q-is-Q-alpha-invariant} part \ref{lemma-item-lattice-action-on-alpha-perp}. 
The fiber $q^{-1}(\uell)$ has a dense  $g(Q_\alpha)$-orbit, if and only if 
the image of 
\begin{equation}
\label{eq-homomorphism-from-Q-alpha-to-C}
(t,\bullet):Q_\alpha\rightarrow \ComplexNumbers
\end{equation} 
is dense in $\ComplexNumbers$.

Suppose first that $\uell$ is special. 
Set $V:=[\uell\oplus \bar{\uell}]\cap [Q_\alpha\otimes_\Integers\RealNumbers]$.
Let  $\lambda$ be a non-zero element in $V\cap Q_\alpha$.
There exists an element $t\in \uell$, such that $\lambda=t+\bar{t}$.
Given an element $z\in Q_\alpha$, then 
$2Re(z,t)=(z,t)+(z,\bar{t})=(z,\lambda)$ is an integer. Thus, $Re(z,t)$ 
belongs to the discrete subgroup $\frac{1}{2}\Integers$ of $\RealNumbers$. 
Hence, the image of the homomorphism (\ref{eq-homomorphism-from-Q-alpha-to-C}) is not dense
in $\ComplexNumbers$.

Assume next that $\uell$ is non-special. 
Denote by $\Theta(\uell)\subset Q_\alpha$ the lattice orthogonal to 
the kernel of the homomorphism (\ref{eq-homomorphism-from-Q-alpha-to-C}).
$\Theta(\uell)$ is the transcendental lattice of the $K3$-surface with period $\uell$.
We know that $\Theta(\uell)$  has rank at least two, and if the rank of 
$\Theta(\uell)$ is $2$, then the Hodge decomposition is defined over $\RationalNumbers$ and so $\uell$
is special. Thus, the rank of $\Theta(\uell)$ is at least three. 
Let $G\subset \Theta(\uell)$ be a co-rank $1$ subgroup. 
We claim that the image $(t,G)$, 
of $G$ via the homomorphism (\ref{eq-homomorphism-from-Q-alpha-to-C}),
spans $\ComplexNumbers$ as a $2$-dimensional real vector space. 
The latter statement is equivalent to the statement that the image of $G$ in $V^*$, 
under the map $z\mapsto (z,\bullet)$ which has real values on $V$, 
spans $V^*$. The equivalence is clear considering the following isomorphisms of two dimensional real
vector spaces:
\[
\xymatrix{
%(z,t) & (z,\bullet) & Re(z,\bullet) & Re(z,(\bullet)^{2,0})
%\\
\ComplexNumbers &
\Hom_\ComplexNumbers(\ell,\ComplexNumbers) \ar[l]_-{ev_t} \ar[r]^{Re}&
\Hom_\RealNumbers(\ell,\RealNumbers) \ar[r]^-{p^*} &
\Hom_\RealNumbers(V,\RealNumbers)=V^*,
}
\]
where $ev_t$ is evaluation at $t$, $Re$ takes $(z,\bullet)$ to its real part $Re(z,\bullet)$, and $p^*$
is pullback via the projection $p:V\rightarrow \ell$
on the $(2,0)$ part.
Assume that the image of $G$ in $V^*$
spans a one-dimensional subspace $W$.
Let $U$ be the subspace of $V$ annihilated by $W$, and hence also by $(z,\bullet)$, $z\in G$.
Then the kernel of the homomorphism
$\Lambda_{k3}\rightarrow U^*$, given by $z\mapsto (z,\bullet)$, 
has co-rank $1$ in $\Lambda_{k3}$.
It follows that the decomposition $\Lambda_{k3}\otimes_\Integers\RealNumbers=U\oplus U^\perp$
is defined over $\RationalNumbers$.
Thus, $U\cap \Lambda_{k3}$ is non-trivial and $\uell$ is special. A contradiction. 
Thus, indeed, the image
$(t,G)$ of $G$ spans $\ComplexNumbers$. Let $Z\subset \ComplexNumbers$ be the image $(t,\Theta(\uell))$
of $\Theta(\uell)$ via the homomorphism (\ref{eq-homomorphism-from-Q-alpha-to-C}).
We have established that $Z$ satisfies the hypothesis of
Lemma \ref{lemma-density-of-Z-in-R-2} below, which implies that the image of 
the homomorphism (\ref{eq-homomorphism-from-Q-alpha-to-C})
is dense in $\ComplexNumbers$.
\end{proof}

\begin{lem}
\label{lemma-density-of-Z-in-R-2}
Let $Z\subset \RealNumbers^2$ be a free additive subgroup of rank $\geq 3$. 
Assume that any co-rank $1$ subgroup of $Z$ spans $\RealNumbers^2$ as a real vector space.
Then $Z$ is dense in $\RealNumbers^2$.
\end{lem}

\begin{proof}
Let $\Sigma$ be the set of all bases of $\RR^2$, consisting of elements of $Z$.
Given a basis $\beta\in\Sigma$, $\beta=\{z_1,z_2\}$, set $\Abs{\beta}=\Abs{z_1}+\Abs{z_2}$.
Set
$
I := \inf\{\Abs{\beta} \ : \ \beta\in\Sigma\}.
$
Note that the closed parallelogram $P_\beta$ with vertices $\{0,z_1,z_2,z_1+z_2\}$
has diameter $< \Abs{\beta}$. Furthermore, every point of the plane belongs to a translate of
$P_\beta$ by an element of the subset ${\rm span}_\Integers\{z_1,z_2\}$ of $Z$.
Hence, it suffices to prove that $I=0$.

The proof is by contradiction. Assume that $I>0$. Let $\beta=\{z_1,z_2\}$ be a basis satisfying
$I\leq \Abs{\beta}<\frac{12}{11}I$.
% for some real number $\epsilon$ in the interval $(0,\frac{1}{11})$. 
We may assume, without loss of generality, that $\Abs{z_1}\geq \Abs{z_2}$. 

We prove next that there exists an element $w\in Z$, such that $w=c_1z_1+c_2z_2$, where the coefficients $c_i$
are irrational. Set $r:=\rank(Z)$.
Let $z_3, \dots, z_r$ be elements of $Z$ completing 
$\{z_1,z_2\}$ to a subset, which is linearly independent over $\RationalNumbers$. 
Write $z_j=c_{j,1}z_1+c_{j,2}z_2$, for $3\leq j\leq r$. 
Assume that $c_{j,1}$ are rational, for $3\leq j\leq r$. 
Then there exists a positive integer $N$, such that $Nc_{j,1}$ are integers, for all $3\leq j\leq r$. 
Then
\[
\{z_2, Nz_3-Nc_{3,1}z_1, \dots, Nz_r-Nc_{r,1}z_1
\}
\]
spans a co-rank $1$ subgroup of $Z$, which lies on $\RR z_2$. This contradicts the assumption 
on $Z$. Hence, there exists an element $w\in Z$, such that $w=c_1z_1+c_2z_2$, where the coefficient $c_1$
is irrational. Repeating the above argument for $c_2$, we get the desired conclusion.

Choose an element $w$ as above. By adding vectors in ${\rm span}_\Integers\{z_1,z_2\}$, 
and possibly after changing the signs of $z_1$ or $z_2$, we may assume that 
$w=c_1z_1+c_2z_2$, with $0<c_1<\frac{1}{2}$ and
$0< c_2< \frac{1}{2}$.
Then $w$ belongs to the parallelogram
$\frac{1}{2}P_\beta$ with vertices
$\{0, \frac{z_1}{2},\frac{z_2}{2},\frac{z_1+z_2}{2}\}$. 
If $c_1$ and $c_2$ are both larger than $\frac{1}{3}$ replace $w$ by
$z_1+z_2-2w$. We may thus assume further, that at least one $c_i$ is $\le\frac{1}{3}$.
In particular, 
$
\Abs{w}\leq c_1\Abs{z_1}+c_2\Abs{z_2}<\frac{5}{6}\Abs{z_1}.
%< \frac{1}{2}\Abs{z_1}+\frac{1}{2}\Abs{z_2}\leq \Abs{z_1}. 
$
Consider the new basis $\tilde{\beta}:=\{w,z_2\}$ of $\RR^2$. 
Then
$
\Abs{\tilde{\beta}}=\Abs{w}+\Abs{z_2}<\frac{5}{6}\Abs{z_1}+\Abs{z_2}=\Abs{\beta}-\frac{1}{6}\Abs{z_1}
\leq \frac{11}{12}\Abs{\beta}<I.
$
We obtain the desired contradiction.
%****************
% Hide
%****************
\hide{
Set $\delta:=\frac{2\epsilon}{1+\epsilon}$.
If $\frac{\Abs{w}}{\Abs{z_1}}<1-\delta$, consider the new basis
$\tilde{\beta}:=\{w,z_2\}$ of $\RR^2$. 
Then 
\[
\Abs{\tilde{\beta}}=\Abs{w}+\Abs{z_2}<\Abs{\beta}-\delta\Abs{z_1}\leq
\Abs{\beta}(1-\delta/2)<(1+\epsilon)(1-\delta/2)I=I.
\]
We obtain the desired contradiction.

The above construction can be carried out for a sequence of positive values of $\epsilon$ approaching zero.
If $1-\delta\leq \frac{\Abs{w}}{\Abs{z_1}}< 1$, then as $\epsilon$ approaches zero, 
$\frac{\Abs{z_2}}{\Abs{z_1}}$ approaches $1$, the angle $\theta$ at the vertex
$0$ of $\frac{1}{2}P_\beta$ approaches zero, and $w$ approaches the vertex
$\frac{z_1+z_2}{2}$ of $\frac{1}{2}P_\beta$ (draw a picture).
But then $w-z_2$ approaches
$\frac{z_1-z_2}{2}$, and 
\[
\Abs{\frac{z_1-z_2}{2}}
=
\frac{1}{2}\sqrt{
\Abs{z_1}^2+
\Abs{z_2}^2-
2\Abs{z_1}\Abs{z_2}\cos(\theta),
}
\]
which approaches $\frac{\Abs{z_1}-\Abs{z_2}}{2}$, which approaches $0$.
In that case consider the basis 
\[\tilde{\beta}:=\{z_2-w,z_2\}\] of $\RR^2$
to get the desired contradiction.
%****************
% End Hide
%****************
}
\end{proof}

%Fix a period $\underline{\ell}\in \Omega_{Q_\alpha}^+$.
%Let $q^{-1}(\underline{\ell})$ be the line in $\Omega_{\alpha^\perp}$ over $\underline{\ell}$. 

%\begin{lem}
%\label{lemma-density-in-a-generic-fiber}
%For a generic $\underline{\ell}$ in $\Omega_{Q_\alpha}^+$, 
%every $g(Q_\alpha)$-orbit in $q^{-1}(\underline{\ell})$ is 
%dense in $q^{-1}(\underline{\ell})$.
%\end{lem}

%\begin{proof}
%Fix a period $\ell$ such that $q(\ell)=\underline{\ell}$ and choose
%a non-zero element $x$ of the line $\ell$ in
%$\alpha^\perp_\Lambda\otimes_\Integers\ComplexNumbers$. Then 
%$q^{-1}(\underline{\ell})=\PP[\ComplexNumbers x+\ComplexNumbers\alpha]\setminus\{\PP[\ComplexNumbers\beta]\}$
%and $g_{[z]}(ax+b\beta)=ax+(b+(x,z))\beta$. 
%Consider  the homomorphism
%\[
%(x,\bullet) \ : \ Q_\alpha \rightarrow \ComplexNumbers,
%\]
%sending $[z]$ to $(x,z)$.
%It suffices to prove that $(x,Q_\alpha)$ is dense in $\ComplexNumbers$, for a generic
%$\underline{\ell}$.
%Indeed, for a generic $\underline{\ell}$, 
%$(x,Q_\alpha)$ is a free additive subgroup of $\ComplexNumbers$ of rank $21$, which is
%dense in $\ComplexNumbers$.
%\end{proof}

Denote by $J_\alpha\subset \Omega_{\alpha^\perp}$ the union of all the $g(Q_\alpha)$ translates of the
section $\tau_\gamma$ constructed in Equation (\ref{eq-section-of-q}) above. 
\[
J_\alpha:= \bigcup_{y\in Q_\alpha}
g_y
\left[\tau_\gamma\left(\Omega^+_{Q_\alpha}\right)\right].
\]
One easily checks that $g_{[z]}\circ\tau_\gamma=\tau_\delta$, where
$\delta:=\gamma+\iota(z)+(\gamma,\iota(z))\beta+\frac{(z,z)}{2}\beta$,
for all $z\in \alpha^\perp_{\Lambda}$, and so $J_\alpha$ is independent of the choice of $\gamma$.

\begin{prop}
\label{prop-density-of-periods-of-Lagrangian-fibrations}
\begin{enumerate}
\item
\label{cor-item-density-of-J-alpha}
$J_\alpha$ is a dense subset of $\Omega_{\alpha^\perp}^+$. 
\item
\label{cor-item-open-subset-V-must-contain-all-non-special-periods}
If $V$ is a $g(Q_\alpha)$-invariant open subset of $\Omega_{\alpha^\perp}^+$, which contains $J_\alpha$, 
then $V$ contains every non-special period in $\Omega_{\alpha^\perp}^+$.
\item
\label{cor-item-periods-in-J-alpha-are-lagrangian-fibrations}
For every $\ell\in J_\alpha$, there exists a marked pair $(M,\eta)$,
consisting of a smooth projective irreducible holomorphic symplectic manifold $M$ of $K3^{[n]}$-type
and a marking $\eta:H^2(M,\Integers)\rightarrow \Lambda$ with period $\ell$
satisfying the following properties.
\begin{enumerate}
\item
The composition $\iota\circ \eta:H^2(M,\Integers)\rightarrow \widetilde{\Lambda}$
belongs to the canonical $O(\widetilde{\Lambda})$-orbit $\iota_M$
of Theorem \ref{thm-orbit-iota-X}. 
\item
There exists a Lagrangian fibration 
$\pi:M\rightarrow \PP^n$, such that 
the class $\eta^{-1}(\alpha)$ is equal to $\pi^*c_1(\StructureSheaf{\PP^n}(1))$.
\end{enumerate}
\end{enumerate}
\end{prop}

\begin{proof}
(\ref{cor-item-density-of-J-alpha})
The density of $J_\alpha$ follows from Lemma
\ref{lemma-density-in-a-generic-fiber}.

(\ref{cor-item-open-subset-V-must-contain-all-non-special-periods})
$V$ intersects every non-special fiber $q^{-1}(\uell)$ in a non-empty open $g(Q_\alpha)$-equivariant
subset of the latter. The complement $q^{-1}(\uell)\setminus V$ is thus a closed 
$g(Q_\alpha)$-equivariant proper subset of the fiber. But any $g(Q_\alpha)$-orbit in
the non-special fiber $q^{-1}(\uell)$ is dense in $q^{-1}(\uell)$, by Lemma
\ref{lemma-density-in-a-generic-fiber}. 
Hence, the complement $q^{-1}(\uell)\setminus V$ must be empty.

(\ref{cor-item-periods-in-J-alpha-are-lagrangian-fibrations})
If $\ell_0$ belongs to the section $\tau_\gamma\left(\Omega^+_{Q_\alpha}\right)$,
then such a pair $(M,\eta):=(M_H(u),\eta_1)$ was constructed in Diagram
(\ref{diagram-marked-M-H-u}), as mentioned in  Remark \ref{rem-eta-1-is-compatible-with-the-orientations}. 
If $\ell=g_z(\ell_0)$, 
%belongs to $g_z\left[\tau_\gamma\left(\Omega^+_{Q_\alpha}\right)\right]$, 
$z\in\alpha^\perp_\Lambda$,
set $(M,\eta)=(M_H(u),g_z\circ\eta_1)$.
\end{proof}
%***************************************************************************
%
%***************************************************************************
\section{Primitive isotropic classes and Lagrangian fibrations}
\label{sec-proof-of-main-thm}
We prove Theorem \ref{thm-main}
in this section using the geometry of the moduli 
space $\FM^0_{\alpha^\perp}$ given in Equation
(\ref{eq-FM-alpha-perp}).
Recall that $\FM^0_{\alpha^\perp}$ is a connected component of the moduli space 
of marked pairs $(X,\eta)$ with $X$ of $K3^{[n]}$-type and such that 
$\eta^{-1}(\alpha)$ is a primitive isotropic class of Hodge type $(1,1)$ 
in the boundary of the positive cone in $H^{1,1}(X,\RealNumbers)$.
%Let $\mu$ be an element of $Mon^2(X)$, set $\alpha_1:=\alpha$, 
%and $\alpha_2:=\mu(\alpha)$. Then $\mu$ induces an isomorphism
%\begin{eqnarray*}
%\FM^0_{\alpha_1^\perp} & \longrightarrow & \FM^0_{\alpha_2^\perp}
%\\
%(Y,\eta) & \mapsto & (Y,\eta\circ \mu^{-1}).
%\end{eqnarray*}
%Hence, we need only consider one representative $\alpha$ from each
%$Mon^2(X)$-orbit of primitive isotropic classes. There are finitely many such orbits,
%by Lemma \ref{lemma-classification}. These $Mon^2(X)$-orbits are enumerated by
%the invariants $d:=\div(\alpha,\bullet)$ and the isometry class $[L,v](\alpha)$.

Fix a connected moduli space $\FM^0_{\alpha^\perp}$ 
as in Equation (\ref{eq-FM-alpha-perp}).
%Proposition \ref{prop-same-connected-component}. 
Denote by $\LB_{\eta^{-1}(\alpha)}$ the line bundle on $X$ with
$c_1(\LB)=\eta^{-1}(\alpha)$.
Let $V$ be the subset of $\FM^0_{\alpha^\perp}$ consisting of all pairs $(X,\eta)$, 
such that $\LB_{\eta^{-1}(\alpha)}$ induces a Lagrangian fibration.

\begin{thm}
\label{thm-dimensions-of-h-i}
%$V$ is  a dense open subset of $\FM^0_{\alpha^\perp}$. 
The image of $V$ via the period map contains every non-special period in
$\Omega_{\alpha^\perp}^+$. 
\end{thm}

%The Theorem is proven in section \ref{sec-proof-of-main-thm}.

\begin{proof} %(of Theorem \ref{thm-dimensions-of-h-i}).
Let $(X,\eta)$ be a marked pair in $\FM_{\alpha^\perp}^0$.
The property that $\eta^{-1}(\alpha)$ is the first Chern class of 
a line-bundle $\LB$ on $X$, which 
induces a Lagrangian fibration $X\rightarrow \linsys{\LB}^*$, 
is an open property in the
moduli space of marked pairs, by a result of Matsushita  \cite{matsushita}. 
$V$ is thus an open subset.

Choose a primitive embedding
$\iota:\Lambda\rightarrow \widetilde{\Lambda}$  with the property that 
$\iota\circ\eta$ belongs to the canonical $O(\widetilde{\Lambda})$-orbit $\iota_X$ of Theorem
\ref{thm-orbit-iota-X}, for all $(X,\eta)$ in $\FM_{\Lambda}^0$.
Let $Mon^2(\Lambda,\iota)$ and its subgroup $Mon^2(\Lambda,\iota)_\alpha$ be the subgroups of $O^+(\Lambda)$ introduced in Lemma \ref{lemma-q-is-Q-alpha-invariant}.
The component $\FM^0_\Lambda$ of the moduli space of marked pairs is invariant 
under $Mon^2(\Lambda,\iota)$, by Theorem \ref{thm-Mon-2-is-stabilizer}. The subset 
$\FM_{\alpha^\perp}^0$ of  $\FM^0_\Lambda$ is invariant under the subgroup
$Mon^2(\Lambda,\iota)_\alpha$. 
Hence, the subset $V$ is $Mon^2(\Lambda,\iota)_\alpha$ invariant.
The construction in section \ref{sec-the-section-tau-of-q} yields a marked pair 
$(M_H(u),\eta_1)$ with period in the image of  the section
$\tau_\gamma:\Omega_{Q_\alpha}^+\rightarrow \Omega_{\alpha^\perp}^+$,
given in Equation (\ref{eq-section-of-q}). Furthermore, the class $\eta_1^{-1}(\alpha)$ induces 
a Lagrangian fibration of $M_H(u)$. 
The marked pair $(M_H(u),\eta_1)$ belongs to $\FM_{\alpha^\perp}^0$, by Proposition
\ref{prop-same-connected-component} 
%(the commutative diagram (\ref{diagram-marked-M-H-u}) and
(Remark \ref{rem-eta-1-is-compatible-with-the-orientations} verifies the conditions of  Proposition
\ref{prop-same-connected-component}).
Hence, $(M_H(u),\eta_1)$  belongs to $V$ and the image of the section
$\tau_\gamma:\Omega_{Q_\alpha}^+\rightarrow \Omega_{\alpha^\perp}^+$
%given in Equation (\ref{eq-section-of-q}), 
is thus contained in the image of $V$ via the period map.
The period map $P_0$ is $Mon^2(\Lambda,\iota)_\alpha$ equivariant and a local homeomorphism,
by the Local Torelli Theorem \cite{beauville}. 
Hence, the image $P_0(V)$ is an open and $Mon^2(\Lambda,\iota)_\alpha$ invariant subset of
$\Omega_{\alpha^\perp}^+$. Any $Mon^2(\Lambda,\iota)_\alpha$ invariant subset,
which contains the section $\tau_\gamma(\Omega_{Q_\alpha}^+)$, contains also the 
dense subset $J_\alpha$ of Proposition
\ref{prop-density-of-periods-of-Lagrangian-fibrations}. $P_0(V)$ thus contains
every non-special period in $\Omega_{\alpha^\perp}^+$, by Proposition
\ref{prop-density-of-periods-of-Lagrangian-fibrations} (\ref{cor-item-open-subset-V-must-contain-all-non-special-periods}). 
%*********
% Hide
%*********
\hide{
We prove the density of $V$ next. Let $(X,\eta)$ be a point of $\FM_{\alpha^\perp}^0$ and $U$ an open subset
of $\FM_{\alpha^\perp}^0$ containing $(X,\eta)$.
A generic point $(Y,\psi)$ of $U$ has the following two properties.
(i) $H^{1,1}(Y,\Integers)$ is cyclic generated by $\psi^{-1}(\alpha)$, and
(ii) The period $P_0(Y,\psi)$ is non-special. $P_0(Y,\psi)$ belongs to $P_0(V)$, by property (ii).
Property (i) implies that the K\"{a}hler cone of $Y$ is equal to the positive cone
\cite[Corollaries 5.7 and 7.2]{huybrechts-basic-results}.
Hence, the isomorphism class of $(Y,\psi)$ is a separated point of $\FM_{\alpha^\perp}^0$ and 
the unique point in the fiber of $P_0$ in $\FM_{\alpha^\perp}^0$,
by Verbitsky's Global Torelli Theorem \ref{thm-Torelli}.
It follows that $(Y,\psi)$ belongs to the intersection $V\cap U$.
We conclude that $V$ is a dense subset of $\FM_{\alpha^\perp}^0$.
%*********
% End Hide
%*********
}
\end{proof}

We will need the following criterion of Kawamata for a line bundle to be semi-ample.
Let $X$ be a smooth projective variety and $D$ a divisor class on $X$.
Set $\nu(X,D):=\max\{e \ : \ D^e\not\equiv 0\}$, where 
$\equiv$ denotes numerical equivalence.
If $D\equiv 0$, we set $\nu(X,D)=0$.
Denote by $\Phi_{kD}:X\dashedrightarrow \linsys{kD}^*$ the rational map, defined whenever the 
linear system is non-empty.
Set $\kappa(X,D):=\max\{\dim \Phi_{kD}(X) \ : \ k> 0\}$, if $\linsys{kD}$ is non-empty for some positive 
integer $k$,
and $\kappa(X,D):=-\infty$, otherwise.
\begin{thm} 
\label{thm-kawamata-semi-ampleness-criterion}
(A special case of \cite[Theorem 6.1]{kawamata}).
Let $X$ be a smooth projective variety with a trivial canonical bundle and $D$ 
a nef divisor. Assume that $\nu(X,D)=\kappa(X,D)$ and $\kappa(X,D)\geq 0$. 
Then $D$ is semi-ample, i.e., there exists a positive integer $k$ such that the linear system 
$\linsys{kD}$ is base point free.
\end{thm}

An alternate proof of Kawamata's Theorem is provided in \cite{fujino}.
A reduced and irreducible divisor $E$ on $X$ is called {\em prime-exceptional},
if the class $e\in H^2(X,\Integers)$ of $E$ satisfies $(e,e)<0$. Consider 
the reflection $R_E:H^2(X,\Integers)\rightarrow H^2(X,\Integers)$, given by
\[
R_E(x) = x-\frac{2(x,e)}{(e,e)}e.
\]
It is known that the reflection $R_E$ by the class of a prime exceptional divisor $E$ 
is a monodromy operator, 
and in particular an integral isometry 
\cite[Cor. 3.6]{markman-prime-exceptional}. 
Let $W(X)\subset O(H^2(X,\Integers))$ be the subgroup generated by reflections
$R_E$ by classes of prime exceptional divisors $E\subset X$. Elements of 
$W(X)$ preserve the Hodge structure, hence $W(X)$ acts on $H^{1,1}(X,\Integers)$.

Let $\Pex_X\subset H^{1,1}(X,\Integers)$ be the set of classes of prime exceptional divisors. 
The {\em fundamental exceptional chamber} of the positive cone $\C_X$ is
the set
\[
\FE_X \ \ := \ \ \{a\in\C_X \ : \ (a,e)>0, \ \mbox{for all} \ e\in\Pex_X\}.
\]
The closure of $\FE_X$ in $\C_X$  is a fundamental domain for the action of $W(X)$
\cite[Theorem 6.18]{markman-torelli}. Let $f:X\dashedrightarrow Y$ be a bimeromorphic map 
to an irreducible holomorphic symplectic manifold $Y$ and $\K_Y$
the K\"{a}hler cone of $Y$. Then $f^*\K_Y$ is an open subset of $\FE_X$. Furthermore, 
the union of $f^*\K_Y$, as $f$ and $Y$ vary over all such pairs, is a dense open subset of $\FE_X$, by a result of Boucksom
\cite{boucksom} (see also 
\cite[Theorem 1.5]{markman-torelli}).

\begin{proof} (of Theorem \ref{thm-main}).
\underline{Step 1:}
Keep the notation in the opening paragraph of section \ref{sec-density-of-periods-of-compactified-jacobians}.
Choose a marking $\eta:H^2(X,\Integers)\rightarrow\Lambda$, such that $\iota\circ\eta$ belongs to the 
canonical $O(\widetilde{\Lambda})$-orbit $\iota_X$. 
Set $\alpha:=\eta(c_1(\LB))$. 
Then $(X,\eta)$ belongs to a component $\FM_{\alpha^\perp}^0$ of the moduli space of marked pairs
of $K3^{[n]}$-type considered in Theorem \ref{thm-dimensions-of-h-i}. 
We use here the assumption that $\LB$ is nef
%(or the inequality $(\kappa,c_1(\LB))>0$ in the non-algebraic case), 
in order to deduce that 
$\eta^{-1}(\alpha)$ belongs to the boundary of  the positive cone of $X$, used in Theorem \ref{thm-dimensions-of-h-i}.

The period $P_0(X,\eta)$ is non-special, by assumption.
There exists a marked pair $(Y,\psi)$ in $\FM_{\alpha^\perp}^0$
satisfying $P_0(Y,\psi)=P_0(X,\eta)$, such that the
class $\psi^{-1}(\alpha)$ induces a Lagrangian fibration,
by Theorem \ref{thm-dimensions-of-h-i}.
%The composition $\psi^{-1}\circ \eta:H^2(X,\Integers)\rightarrow H^2(Y,\Integers)$ 
%is a parallel transport operator and a Hodge isometry.
The marked pairs $(X,\eta)$ and $(Y,\psi)$ correspond to inseparable points in the moduli space $\FM^0_{\alpha^\perp}$,
by the Global Torelli Theorem \ref{thm-Torelli}.
Hence, there exists an analytic correspondence 
$Z\subset X\times Y$, 
$Z=\sum_{i=0}^kZ_i$ in $X\times Y$, of pure dimension $2n$, 
with the following properties, by results of Huybrechts
\cite[Theorem 4.3]{huybrechts-basic-results} (see also \cite[Sec. 3.2]{markman-torelli}).
\begin{enumerate}
\item
The homomorphism $Z_*:H^*(X,\Integers)\rightarrow H^*(Y,\Integers)$ is a Hodge isometry, which is
equal to $\psi^{-1}\circ\eta$.
%\item
%$Z_*(c_1(\LB))$
%is the class of a line-bundle on $Y$, which induces a Lagrangian fibration. 
%\item
The irreducible
component 
$Z_0$ of the correspondence is the graph of 
a bimeromorphic map $f:X\dashedrightarrow Y$.
\item
The images in $X$ and $Y$ of all  other components  $Z_i$, $i>0$, 
are of positive co-dimension. 
%If the image is a divisor, then it is uniruled.
\end{enumerate}

\underline{Step 2:}
We prove next that the line bundle $\LB$ over $X$ is semi-ample. 
We consider separately the projective and non-algebraic cases.

\underline{Step 2.1:}
Assume that $X$ is not projective.\footnote{I thank K. Oguiso and S. Rollenske for pointing out to me that 
in the non-algebraic case the result should follow from the above via the results of 
reference \cite{COP}.} 
We claim that 
$f_*(c_1(\LB))=
%Z_*(c_1(\LB))=
\psi^{-1}(\alpha)$. 
%The second equality follows from the equality $Z_*=\psi^{-1}\circ\eta$ mentioned above. 
The Neron-Severi group $NS(X)$ does not contain any positive class, by Huybrechts
projectivity criterion \cite{huybrechts-basic-results}. Hence, the Beauville-Bogomolov-Fujiki pairing restricts to 
$NS(X)$ as a non-positive pairing with a rank one null sub-lattice spanned by the class $c_1(\LB)$.
Similarly, the Beauville-Bogomolov-Fujiki pairing restricts to 
$NS(Y)$ with a rank one null space spanned by $\psi^{-1}(\alpha)$. 
Hence, $f_*(c_1(\LB))=\pm\psi^{-1}(\alpha)$. Now $\psi^{-1}(\alpha)$  is semi-ample 
and hence belongs to the closure of $\FE_Y$. 
The class $c_1(\LB)$ is assumed nef, and hence belongs to the closure of $\FE_X$.
The bimeromorphic map $f$ 
induces a Hodge-isometry $f_*:H^2(X,\Integers)\rightarrow H^2(Y,\Integers)$, which
maps $\FE_X$ onto $\FE_Y$ \cite{boucksom}.
Hence, $f_*(c_1(\LB))$ belongs to $\overline{\FE}_Y$ as well.
We conclude the equality $f_*(c_1(\LB))=\psi^{-1}(\alpha)$.

Let $\LB_2$ be the line bundle with $c_1(\LB_2)=\psi^{-1}(\alpha)$.
The bimeromorphic map $f:X\dashedrightarrow Y$ is holomorphic in co-dimension one, and
so induces an isomorphism $f_1:\linsys{\LB}\rightarrow \linsys{\LB_2}$ of the two linear systems.
Denote by $\Phi_{\LB_2}:Y\rightarrow \linsys{\LB_2}^*$ the Lagrangian fibration induced by $\LB_2$.
We conclude that $\linsys{\LB}$ is $n$ dimensional and 
the meromorphic map $\Phi_{\LB}:X\dashedrightarrow \linsys{\LB}^*$ is an algebraic reduction of $X$
(see \cite{COP}). 
By definition, an {\em algebraic reduction} of $X$ is a dominant meromorphic map 
$\pi:X\dashedrightarrow B$ to a normal projective variety $B$, such that $\pi^*$ induces an isomorphism of the function 
fields of meromorphic functions \cite{COP}. Only the birational class of $B$ is determined by $X$.
Fibers of the algebraic reduction $\pi$ are defined via a resolution of indeterminacy, 
and are closed connected analytic subsets of $X$.
In our case, the generic fiber of $\Phi_\LB$ 
 is bimeromorphic to the generic fiber of $\Phi_{\LB_2}$. 
The generic fiber of $\Phi_{\LB_2}$ is a complex torus, and hence algebraic, by 
\cite[Prop. 2.1]{campana}. Hence, the generic fiber of $\Phi_\LB$ has algebraic dimension $n$. 
It follows that the line bundle $\LB$ is semi-ample, it is the pullback of an ample line-bundle over $B$,
via a holomorphic reduction map
$\pi:X\rightarrow B$ which is a regular morphism, by 
\cite[Theorems 1.5 and 3.1]{COP}.

\underline{Step 2.2:}
When $X$ is projective 
there exists an element $w\in W(X)$, such that Huybrecht's birational map $f:X\dashedrightarrow Y$
satisfies $f^*\circ \psi^{-1}\circ \eta=w$, by \cite[Theorem 1.6]{markman-torelli}.
Set $\alpha_X:=\eta^{-1}(\alpha)$  and $\alpha_Y:=\psi^{-1}(\alpha)$. 
We get the equality $w(\alpha_X)=f^*(\alpha_Y)$. 

Let $\overline{\FE}_X$ be the closure of the fundamental exceptional chamber $\FE_X$ in $H^{1,1}(X,\RealNumbers)$.
The class $\alpha_X$ is nef, by assumption, and it thus belongs to $\overline{\FE}_X$.
We already know that $\alpha_Y$ is the class of a line bundle, which induces a Lagrangian fibration.
Hence, $f^*(\alpha_Y)$ belongs to $\overline{\FE}_X$. The class $w(\alpha_X)$ thus belongs to the 
intersection $w\left(\overline{\FE}_X\right)\cap \overline{\FE}_X$. 

Let $J$ be the subset of $\Pex_X$ given by 
$J=\{e\in\Pex_X \ : \ (e,\alpha_X)=0\}$. Denote by $W_J$ the subgroup of $W(X)$
generated by reflections $R_e$, for all $e\in J$. Then $W_J$ is equal to 
\[
\{w \in W(X) \ : \ w(\alpha_X) \in \overline{\FE}_X\},
\]
by a general property of crystalographic hyperbolic reflection groups
\cite[Lecture 3, Proposition on page 15]{howlett}. We conclude that $w(\alpha_X)=\alpha_X$
and 
\begin{equation}
\label{eq-alpha-X-is-pullback-of-semiample-class}
\alpha_X=f^*(\alpha_Y). 
\end{equation}

We are ready to prove\footnote{I thank C. Lehn for reference \cite[Prop. 2.4]{lai}, used in an earlier argument, 
and T. Peternell and Y. Kawamata for suggesting  the current more direct argument.} that $\LB$ is semi-ample.
The rational map $f$ is regular in co-dimension one. The map $f$ thus induces an isomorphism 
$f_m:\linsys{\LB^m}\rightarrow \linsys{\LB_2^m}$, for every integer $m$. 
Hence, $\kappa(X,\LB)=\kappa(Y,\LB_2)=n$.
Any non-zero isotropic divisor class $D$ on a
$2n$ dimensional irreducible holomorphic symplectic manifold
satisfies $\nu(X,D)=n$, by a result of Verbitsky \cite{verbitsky-cohomology}. Hence, $\nu(X,\LB)=n$. 
The line bundle $\LB$ is assumed to be nef. Hence, $\LB$ is semi-ample, by 
Theorem \ref{thm-kawamata-semi-ampleness-criterion}. 
%Summarizing, $\alpha_X$ is a nef class on $X$ and $f^*(\alpha_X)$ is semi-ample over $Y$.
%Hence, the divisor class $\alpha_X$  is semi-ample as well, by 
%\cite[Prop. 2.4]{lai}.

\underline{Step 3:}
We return to the general case, where $X$ may or may not be projective. 
In both cases we have seen that there exists a positive integer $m$, such that the linear 
system $\linsys{\LB^m}$ is base point free and $\Phi_{\LB^m}$ is a regular morphism. 
Furthermore, the bimeromorphic map $f:X\dashedrightarrow Y$
is regular in co-dimension one and thus induces an isomorphism 
$f_k:\linsys{\LB^k}\rightarrow \linsys{\LB_2^k}$, for every positive integer $k$. Denote by 
$f_k^*:\linsys{\LB_2^k}^*\rightarrow \linsys{\LB^k}^*$ the transpose of $f_k$. 
We get the equality
$
\Phi_{\LB^k}=f_k^*\circ \Phi_{\LB_2^k}\circ f,
$
for all $k$. Let $V_m:\linsys{\LB_2}^*\rightarrow\linsys{\LB_2^m}^*$ be the Veronese embedding.
We get the equalities
\begin{equation}
\label{eq-closed-immersions}
V_m\circ (f_1^*)^{-1}\circ \Phi_\LB=V_m\circ\Phi_{\LB_2}\circ f=
\Phi_{\LB_2^m}\circ f=(f_m^*)^{-1}\circ \Phi_{\LB^m}.
\end{equation}
%where all the maps above, other than $f$ and $\Phi_\LB$, are regular. 
Now, $V_m\circ (f_1^*)^{-1}:\linsys{\LB}^*\rightarrow \linsys{\LB_2^m}^*$ is a closed immersion
and the morphism on the right hand side of (\ref{eq-closed-immersions}) is regular.
Hence, the rational map $\Phi_\LB$ is a regular morphism. 
The base locus of the linear system $\linsys{\LB}$ is thus either empty, or a divisor.
The latter is impossible, since $f$ is regular in co-dimension one and $\linsys{\LB_2}$ is base point free.
Hence, $\linsys{\LB}$ is base point free.
\end{proof}

Let $X$ and $\LB$ be as in Theorem \ref{thm-main}, except that we drop the assumption that $\LB$ 
is nef and assume only that $c_1(\LB)$ belongs to the boundary of the 
positive cone. Assume that $X$ is projective. 

\begin{thm}
\label{thm-case-X-projective-L-isotropic-but-not-nef}
There exists an element $w\in W(X)$, a projective 
irreducible holomorphic symplectic manifold $Y$, a birational map $f:X\dashedrightarrow Y$,
and a Lagrangian fibration $\pi:Y\rightarrow \PP^n$, such that $w(\LB)=f^*\pi^*\StructureSheaf{\PP^n}(1)$.
\end{thm}

\begin{proof}
Let $(Y,\psi)$ be the marked pair constructed in Step 1 of the proof of Theorem \ref{thm-main}.
Then $Y$ admits a Lagrangian fibration $\pi:Y\rightarrow \PP^n$ and the class
$\pi^*c_1(\StructureSheaf{\PP^n}(1))$ was denoted $\alpha_Y$ in that proof.
In step 2.2 of that proof we showed the existence of  a birational map
$f:X\dashedrightarrow Y$ and an element $w\in W(X)$, such that $w(c_1(\LB))=f^*(\alpha_Y)$
(see Equality (\ref{eq-alpha-X-is-pullback-of-semiample-class})).
\end{proof}

%\begin{proof} (of Theorem \ref{cor--effective-semi-group}). 
%Set $N^1(X):=H^{1,1}(X,\Integers)\otimes_\Integers\RealNumbers$. 
%The intersection $\overline{\FE}_X\cap N^1(X)$ is equal to the closure of the movable cone of $X$
%\cite[Lemma 6.22]{markman-torelli}.
%If $\alpha$ is a primitive isotropic class on the boundary of the movable cone 
%of $X$, then $\alpha$ is the class of an effective divisor, by 
%Theorem \ref{thm-main} and the Semi-Continuity Theorem. 
%The conclusion follows by the argument in the 
%paragraph following \cite[Question 10.11]{markman-torelli}. (??? elaborate ???)
%\end{proof}

%************
% Hide
%************
\hide{

In fact, the above conclusion follows from 
an earlier result of Matsushita.
Given a Lagrangian fibration
$\pi:X\rightarrow \linsys{L}^*$, 
one has a isomorphism $R^i\pi_*\StructureSheaf{X}\cong \Omega^i_{\linsys{L}^*}$,
by (\cite{matsushita-higher-direct-images}, Theorem 1.3). Consequently,
$h^i(X,L)=h^i(\linsys{L}^*,\StructureSheaf{\linsys{L}^*}(1))=0$, for all $i>0$,
by Kodaira's vanishing theorem, 
and $h^0(L)=n+1$. 
The semi-continuity theorem  then implies that the same holds away from a closed
complex analytic subset of lower dimension in the connected component of 
$\FM^0_{\alpha^\perp}$.

%***************************************************************************
%
%***************************************************************************
\section{Proof of the main theorem}
Let $n$ be an integer $\geq 2$, 
$d$ a positive integer such that $d^2$ divides $n-1$, and $b$ an integer such that 
$\gcd(d,b)=1$. We construct an example of a pair $(Y,\alpha)$, such that 
$Y$ is of $K3^{[n]}$-type and $\alpha$ is a primitive isotropic class with
monodromy invariant $\left(L_{n,d},(d,b)\right)$ 
(see Example \ref{example-alpha-sum-of-stably-prime-exceptional}).

\begin{question}
\label{question-density}
\begin{enumerate}
\item
\label{question-item-sum-of-two-spe-classes}
Given a $Mon^2(X)$-orbit of primitive isotropic classes $\alpha$, 
describe the subset of $(Y,\eta)\in \FM^0_{\alpha^\perp}$, such that 
$\eta^{-1}(\alpha)$ is the sum of two stably prime-exceptional classes in $H^{1,1}(Y,\Integers)$.
Could this set be dense in $\FM^0_{\alpha^\perp}$? 
\item
Same question, except that 
$\eta^{-1}(\alpha)$ is the sum of a stably prime-exceptional class and a class
of positive degree.
\end{enumerate}
\end{question}

Note that density in Question \ref{question-density} part
(\ref{question-item-sum-of-two-spe-classes}) 
implies that the class $\eta^{-1}(\alpha)$ is effective,
for all $(Y,\eta)$ in $\FM^0_{\alpha^\perp}$, since stably prime-exceptional classes are effective,
and the locus where $\eta^{-1}(\alpha)$ is effective is closed, by the semi-continuity theorem.
It may, however, be the case that density never holds.

\begin{lem}
\label{lemma-sum-of-two-minus-2-classes-is-effective}
Let $X$ be of $K3^{[n]}$-type and assume that $\Pic(X)\cong A$,
where $A$ is the rank $2$ lattice with basis  $\{e_1,e_2\}$ and Gram matrix
$\left(\begin{array}{cc}
-2 & 2
\\
2 & -2
\end{array}\right)$.
Set $\alpha:=e_1+e_2$.
Then  either $\alpha$ is effective or $-\alpha$ is effective.
\end{lem}

\begin{proof}
One of $\alpha$ or $-\alpha$ belongs to the closure of the positive cone  $\C_X$.
We may assume that $\alpha$ does, possibly after 
replacing the basis $\{e_1,e_2\}$ by $\{-e_1,-e_2\}.$
Let $W_{Exc}\subset O(\Lambda)$ be the reflection subgroup 
generated by the two reflections $R_{e_1}$ and $R_{e_2}$, where
\[
R_{e_i}(\lambda) \  \ = \ \ \lambda +(\lambda,e_i)e_i.
\]
Then $W_{Exc}$ is the affine Weyl group of type $\widehat{A}_1$.
The $W_{Exc}$-orbit of $e_1$ is
\[
W_{Exc}\cdot e_1 \ \ \ = \ \ \ \pm e_1 + k\alpha, \ \ \ k\in\Integers.
\]
Choose a K\"{a}hler class $\kappa\in \K_X$, such that 
$\kappa^\perp\cap H^{1,1}(X,\Integers)=(0)$. 
Then $(\kappa,\alpha)>0$.
So $(\kappa,e_1)+(\kappa,e_2)>0$.
We may assume that $(\kappa,e_1)>0$, possibly after interchanging 
the roles of $e_1$ and $e_2$. 
Set $n_0:=\lceil\frac{(\kappa,e_1)}{(\kappa,\alpha)}\rceil$.
Then 
$1+\frac{(\kappa,e_1)}{(\kappa,\alpha)}>n_0>\frac{(\kappa,e_1)}{(\kappa,\alpha)}>0$, 
by our choice of $\alpha$.
Set $\delta_2:=-e_1+n_0\alpha$. If $(\kappa,e_2)>0$, then $n_0=1$ and $\delta=e_2$.
If $(\kappa,e_2)<0$, then
\begin{eqnarray*}
(\kappa,\delta_2)&=&n_0(\kappa,\alpha)-(\kappa,e_1)>0,
\\
(\kappa,\delta_2)-(\kappa,\alpha)&=&(n_0-1)(\kappa,\alpha)-(\kappa,e_1)<0.
\end{eqnarray*}
Set $\delta_1:=\alpha-\delta_2$. Then $(\delta_1,\delta_1)=-2$, 
$\alpha=\delta_1+\delta_2$, and 
$(\kappa,\delta_i)>0$, for $i=1,2$.

The classes $\delta_i$, $i=1,2$,  
are both stably prime-exceptional, by (\cite{markman-prime-exceptional}, Theorem 1.12).
Hence, both $\delta_i$ are effective, and so is $\alpha$.
\end{proof}

\begin{example}
\label{example-alpha-sum-of-stably-prime-exceptional}
Consider the lattice $\widetilde{L}_{n,d}$ with Gram matrix 
$\left(\begin{array}{ccc}
2n-2 & \frac{-b(2n-2)}{d} & 0
\\
\frac{-b(2n-2)}{d} & \frac{b^2(2n-2)}{d^2} & 0
\\
0 & 0 & -2
\end{array}\right)$
in the basis $\{f_1,f_2,f_3\}$. 
The class $bf_1+df_2$ spans the null space of $\widetilde{L}_{n,d}$.
Denote by $Q_{n,d}$ the quotient of $\widetilde{L}_{n,d}$ by its null space. 
Note that $f:=c_1f_1+c_2f_2$ satisfies $(f,f)=\frac{2n-2}{d^2}(c_1d-c_2b)^2$. 
Let $c_1$ and $c_2$ be integers, such that $c_1d-c_2b=1$. 
Let $\bar{f}_i$ be the image of $f_i$ in $Q_{n,d}$ and 
set $q_1:=c_1\bar{f}_1+c_2\bar{f}_2$ and $q_2:=\bar{f}_3$.
Then $Q_{n,d}$ has Gram matrix
$\left(\begin{array}{cc}
\frac{2n-2}{d^2} & 0
\\
0 & -2
\end{array}\right)$
in the basis $\{q_1,q_2\}$.

Let $\Lambda_{K3}$ be the orthogonal direct sum 
$U\oplus U\oplus U\oplus E_8(-1)\oplus E_8(-1)$. 
$O(\widetilde{\Lambda})$-orbits of 
primitive isometric embeddings of $\widetilde{L}_{n,d}$ in the Mukai lattice 
$\widetilde{\Lambda}$ are in one-to-one correspondence with 
$O\left(\Lambda_{K3}\right)$-orbits of primitive isometric embeddings of 
$Q_{n,d}$ into $\Lambda_{K3}$, by (\cite{nikulin}, Proposition 1.17.1). 
There exists a unique
$O(\widetilde{\Lambda})$-orbit of primitive isometric embeddings
$\phi:\widetilde{L}_{n,d}\hookrightarrow \widetilde{\Lambda}$, by 
(\cite{nikulin}, Proposition 1.14.4). Choose one such embedding
$\phi$. Let $\bar{\phi}$ be the corresponding embedding
of $Q_{n,d}$ in $\Lambda_{K3}$.

Set $\tilde{f}_i:=\phi(f_i)$, $v:=\tilde{f}_1$,
$\gamma:=\tilde{f}_2$, $e_2:=\tilde{f}_3$,
$\alpha:=d\gamma+bv=b\tilde{f}_1+d\tilde{f}_2$, 
$e_1:=\alpha-e_2=b\tilde{f}_1+d\tilde{f}_2-\tilde{f}_3$. Then 
$\alpha$ is a primitive isotropic class and $\alpha$ spans the null space of
the image of $\phi$. Hence, $(e_1,e_1)=-2$. Furthermore, $(e_1,e_2)=2$.
Recall that $\Lambda$ is $\phi(v)^\perp_{\widetilde{\Lambda}}$.
The sub-lattice $A:=\Lambda\cap \phi\left(\widetilde{L}_{n,d}\right)$ of $\Lambda$
is the rank $2$ lattice with basis  $\{e_1,e_2\}$ and Gram matrix
$\left(\begin{array}{cc}
-2 & 2
\\
2 & -2
\end{array}\right)$.

Given a period $\ell\in \Omega_\Lambda$, orthogonal to $A$, 
we get a period $\underline{\ell}\in \Omega_{Q_\alpha}\subset \Omega_{K3}$ orthogonal to
$\bar{\phi}\left(Q_{n,d}\right)\subset\Lambda_{K3}$. 
If $\Lambda^{1,1}(\ell,\Integers)=A$, then $\ell$ is the period of
some marked pair $(Y,\eta)$ of $K3^{[n]}$-type, such that one of $\eta^{-1}(\alpha)$
or $-\eta^{-1}(\alpha)$ is effective, by Lemma 
\ref{lemma-sum-of-two-minus-2-classes-is-effective}.
Clearly, $\alpha-bv=d\tilde{f}_2$ is divisible by $d$. Hence, the monodromy invariant of 
$(Y,\eta^{-1}(\alpha))$ is $\left(L_{n,d},(d,b)\right)$.

If $\Lambda^{1,1}(\ell,\Integers)=A$, 
then $\Lambda_{K3}^{1,1}(\underline{\ell},\Integers)=\bar{\phi}\left(Q_{n,d}\right)$ and 
$Q_\alpha^{1,1}(\underline{\ell},\Integers)=\Integers e$, with $(e,e)=-2$.
In other words, $\underline{\ell}$ is the period of a $K3$-surface
with a semi-ample line-bundle of degree $\frac{2n-2}{d^2}$ containing
a $-2$ curve. Equivalently, $\underline{\ell}$ is the period of  a polarized singular 
$K3$ surface of degree $\frac{2n-2}{d^2}$ with a single ordinary double point. 
The set of such periods is {\em not} dense in $\Omega_{Q_\alpha}$.
\end{example}
%************
% End Hide
%************
}

%**********************************************************************************************
%
%**********************************************************************************************
\section{Tate-Shafarevich lines and twists}
\label{sec-tate-shafarevich}

%**********************************************************************************************
%
%**********************************************************************************************
\subsection{The geometry of the universal curve}
Let $S$ be a projective $K3$ surface, $d$ a positive integer, and
$\LB$ a nef line bundle on $S$ of positive degree,
such that the class $c_1(\LB)$ is indivisible. Set $n:=1+\frac{d^2\deg(\LB)}{2}$.
Let $\C\subset S\times \linsys{\LB^d}$ be the universal curve,
$\pi_i$ the projection from $S\times \linsys{\LB^d}$ to the $i$-th factor, $i=1,2$, and
$p_i$ the restriction of $\pi_i$ to $\C$. We assume in this section the following assumptions
about the line bundle $\LB$. 
\begin{assumption}
\label{assumption}
\begin{enumerate}
\item
\label{assumption-item-bpf}
The linear system $\linsys{\LB^d}$ is base point free.
\item
\label{assumption-item-locus}
The locus in $\linsys{\LB^d}$, consisting of divisors which are non-reduced, or reducible having a singularity which is not an ordinary double point, has co-dimension at least $2$.
%One of the following holds.
%\begin{enumerate}
%\label{assumption-codimension}
%\item
%\label{assumption-item-locus}
%The locus in $\linsys{\LB^d}$, consisting of divisors which are reducible or non-reduced,
%has co-dimension at least $2$.
%\item
%\label{assumption-Pic-S-is-cyclic}
%$\Pic(S)$ is cyclic generated by $\LB$. 
%\end{enumerate}
\end{enumerate}
\end{assumption}
 
 \begin{rem}
 \label{remark-on-assumption}
 Assumption \ref{assumption} holds whenever $\Pic(S)$ is cyclic 
 generated by $\LB$. 
The base point freeness Assumption \ref{assumption} (\ref{assumption-item-bpf})
follows from \cite[Prop. 1]{mayer}.
Assumption \ref{assumption} (\ref{assumption-item-locus}) 
is verified as follows.
 If $a+b=d$, $a\geq 1$, $b\geq 1$, then the image of
 $\linsys{\LB^a}\times \linsys{\LB^b}$ in $\linsys{\LB^d}$ has co-dimension
 $2ab\left(\frac{n-1}{d^2}\right)-1$. The co-dimension is at least two, except in 
 the case $(n,d)=(5,2)$. In the latter case $\linsys{\LB}\cong\PP^2$,  
 $\linsys{\LB^2}\cong\PP^5$  and  
 the generic curve in the image of
 $\linsys{\LB}\times \linsys{\LB}$ in $\linsys{\LB^2}$ is the union of two smooth curves
 of genus $2$ meeting transversely at two points. Hence, 
 Assumption \ref{assumption} (\ref{assumption-item-locus}) holds in this case as well.
 \end{rem}
 
The morphism $p_1:\C\rightarrow S$ is a projective
hyperplane sub-bundle of the trivial bundle over $S$ with fiber $\linsys{\LB^d}$,
by the base point freeness Assumption \ref{assumption} (\ref{assumption-item-bpf}). 
Assumption \ref{assumption} (\ref{assumption-item-locus}) will be used in the proof of 
Lemma \ref{lemma-vertical-tangent-sheaf-is-pullback-of-cotangent-bundle}.
Consider the exponential short exact sequence over $\C$
\[
0\rightarrow \Integers\rightarrow \StructureSheaf{\C}\rightarrow \StructureSheaf{\C}^*\rightarrow 0.
\]
We get the exact sequence of sheaves of abelian groups over $\linsys{\LB^d}$
\begin{equation}
\label{eq-exact-sequence-of-abelian-groups}
0\longrightarrow R^1p_{2_*}\Integers \longrightarrow 
R^1p_{2_*}\StructureSheaf{\C}\longrightarrow 
R^1p_{2_*}\StructureSheaf{\C}^*\LongRightArrowOf{\deg} 
R^2p_{2_*}\Integers \longrightarrow 0,
\end{equation}
where we work in the complex analytic category. 
Note that $\deg$ above is surjective, since $R^2p_{2_*}\StructureSheaf{\C}$ vanishes.
Set $\Sha:=H^1_{}(\linsys{\LB^d},R^1p_{2_*}\StructureSheaf{\C}^*)$ and 
%be the image of $H^1_{}(\linsys{\LB^d},R^1p_{2_*}\StructureSheaf{\C})\rightarrow
%H^1_{}(\linsys{\LB^d},R^1p_{2_*}\StructureSheaf{\C}^*)$. 
%Note that $\Sha$ is contained in the identity component of the
%analytic Tate-Shafarevich group
%$H^1_{}(\linsys{\LB^d},R^1p_{2_*}\StructureSheaf{\C}^*)$. 
$\widetilde{\Sha}:=H^1_{}(\linsys{\LB^d},R^1p_{2_*}\StructureSheaf{\C})$. 
Set $Br'(S):=H^2_{}(S,\StructureSheaf{S}^*)$ and  $Br'(\C):=H^2_{}(\C,\StructureSheaf{\C}^*)$.

\begin{lem}
\label{lemma-Sha-is-one-dimensional}
\begin{enumerate}
\item
\label{lemma-item-R-1-p-2-of-O-C-is-the-cotangent-bundle}
There is a natural isomorphism 
$R^1p_{2_*}\StructureSheaf{\C} \cong T^*\linsys{\LB^d}\otimes_{\ComplexNumbers}H^{2,0}(S)^*$.
\item
\label{lemma-item-Sha-is-one-dimensional}
$\widetilde{\Sha}$ is naturally isomorphic to 
%both $H^{2,0}(S)^*\otimes H^{1,1}(\linsys{\LB^d})$ and 
$H^{0,2}(\C)$. Consequently, $\widetilde{\Sha}$ is one dimensional.
\item
\label{lemma-item-second-cohomology-of-K3-is-a-direct-summand}
$H^2(\C,\Integers)$ decomposes as a direct sum 
$H^2(\C,\Integers)=p_1^*H^2(S,\Integers)\oplus p_2^*H^2(\linsys{\LB^d},\Integers)$. 
The groups $H^i(\C,\Integers)$ vanish for odd $i$. The Dolbeault cohomologies $H^{p,q}(\C)$ vanish, 
if $\Abs{p-q}>2$.
\item
\label{lemma-item-Tate-Shafarevich-groups-of-S-and-C-are-isomorphic}
The pullback homomorphism $p_1^*:H^2_{}(S,\StructureSheaf{S}^*)\rightarrow 
H^2_{}(\C,\StructureSheaf{\C}^*)$ is an isomorphism.
The Leray spectral sequence yields an isomorphism
\[
b:H^2_{}(\C,\StructureSheaf{\C}^*)\rightarrow H^1(\linsys{\LB^d},R^1p_{2_*}\StructureSheaf{\C}^*).
\]
%whose image is $\Sha$ and whose co-kernel is either trivial or infinite cyclic. 
Consequently, we have the isomorphisms
\[
\xymatrix{
Br'(S) \ar[r]^{p_1^*}_-\cong &
Br'(\C) \ar[r]^-b_-\cong &
\Sha.
}
\]
%\[
%\xymatrix{
%H^2_{}(S,\StructureSheaf{S}^*)\ar[r]_{\cong}^{p_1^*} &
%H^2_{}(\C,\StructureSheaf{\C}^*)\ar[r]_-{\cong} &
%H^1(\linsys{\LB^d},R^1p_{2_*}\StructureSheaf{\C}^*).
%}
%\]
%Furthermore, the exponential homomorphism 
%$H^1(\linsys{\LB^d},R^1p_{2_*}\StructureSheaf{\C})\rightarrow 
%H^1(\linsys{\LB^d},R^1p_{2_*}\StructureSheaf{\C}^*)$ is surjective.
\end{enumerate}
\end{lem}

Let $\F$ be a sheaf of abelian groups over $\C$.
Let $F^pH^k(\C,\F)$ be the Leray filtration associated to the morphism $p_2:\C\rightarrow \linsys{\LB^d}$
and $E^{p,q}_\infty:=F^pH^{p+q}(\C,\F)/F^{p+1}H^{p+q}(\C,\F)$ its graded pieces. 
Recall that the $E^{p,q}_2$ terms are
$E^{p.q}_2:=H^p(\linsys{\LB^d},R^qp_{2_*}\F)$
and the differential at this step is $d_2:E^{p,q}_2\rightarrow E^{p+2,q-1}_2$.

\begin{proof}
%It suffices to prove that $H^1(\linsys{\LB^d},R^1p_{2_*}\StructureSheaf{\C})$ is one-dimensional.
%The statement would thus follow, once we prove that $R^1p_{2_*}\StructureSheaf{\C}$
%is isomorphic to $T^*\linsys{\LB^d}$.
\ref{lemma-item-R-1-p-2-of-O-C-is-the-cotangent-bundle})
We have the isomorphism
$\StructureSheaf{S\times \linsys{\LB^d}}(\C)\cong\pi_1^*\LB^d\otimes\pi_2^*\StructureSheaf{\linsys{\LB^d}}(1)$.
Apply the functor $R\pi_{2_*}$ to the short exact sequence
$
0\rightarrow 
\StructureSheaf{S\times \linsys{\LB^d}}\rightarrow
\StructureSheaf{S\times \linsys{\LB^d}}(\C)\rightarrow 
\StructureSheaf{\C}(\C)\rightarrow 0
$
to obtain the Euler sequence of the tangent bundle.
\[
0\rightarrow \StructureSheaf{\linsys{\LB^d}}
\rightarrow H^0(S,\LB^d)\otimes_\ComplexNumbers\StructureSheaf{\linsys{\LB^d}}(1)
\rightarrow T\linsys{\LB^d}\rightarrow 0.
\]
Now $\StructureSheaf{\C}(\C)\otimes_\ComplexNumbers H^{2,0}(S)$ is isomorphic to the relative dualizing sheaf 
$\omega_{p_2}$. 
We get the isomorphisms
\[
R^1p_{2_*}\StructureSheaf{\C}\cong 
[R^0p_{2_*}\StructureSheaf{\C}(\C)\otimes_\ComplexNumbers H^{2,0}(S)]^*\cong
[R^0p_{2_*}\StructureSheaf{\C}(\C)]^*\otimes_\ComplexNumbers H^{2,0}(S)^*\cong 
T^*\linsys{\LB^d}\otimes_\ComplexNumbers H^{2,0}(S)^*.
\]

\ref{lemma-item-Sha-is-one-dimensional})
%We construct next a natural isomorphism $H^{0,2}(\C)\cong \widetilde{\Sha}$. 
$R^2p_{2_*}\StructureSheaf{\C}$ vanishes, since $p_2$ has one-dimensional fibers.
$H^2(\linsys{\LB^d},p_{2_*}\StructureSheaf{\C})$ vanishes, since  
$p_{2_*}\StructureSheaf{\C}\cong \StructureSheaf{\linsys{\LB^d}}$. The latter isomorphism follow from the fact that
$p_2$ has connected fibers. We conclude that $H^2(\C,\StructureSheaf{\C})$ is isomorphic to the $E^{1,1}_\infty$ graded
summand of its Leray filtration.
%by 
%considering the long exact sequence of higher direct images $R^i\pi_{2_*}$ associated
%to the short exact sequence
%\[
%0\rightarrow \StructureSheaf{S\times \linsys{\LB^d}}(-\C)\rightarrow
%\StructureSheaf{S\times \linsys{\LB^d}}\rightarrow \StructureSheaf{\C}\rightarrow 0.
%\]
The differential $d_2:H^1(\linsys{\LB^d},R^1p_{2_*}\StructureSheaf{\C})\rightarrow 
H^3(\linsys{\LB^d},p_{2_*}\StructureSheaf{\C})$ vanishes, since $H^{0,3}(\linsys{\LB^d})$ vanishes.
Hence, the $E^{1,1}_2$ term $\widetilde{\Sha}:=H^1(\linsys{\LB^d},R^1p_{2_*}\StructureSheaf{\C})$ 
is isomorphic to $H^2(\C,\StructureSheaf{\C})$.
%which in turn is isomorphic to $H^{0,2}(S)$. 

\ref{lemma-item-second-cohomology-of-K3-is-a-direct-summand}) 
The statement is topological and so it suffices to prove it in the case where $\Pic(S)$
is cyclic generated by $\LB$. In this case $\LB$ is ample, and so 
the line bundle $\pi_1^*\LB^d\otimes\pi_2^*\StructureSheaf{\linsys{\LB^d}}(1)$ is ample. 
The Lefschetz Theorem on Hyperplane Sections 
implies that the restriction homomorphism  $H^2(S\times\linsys{\LB^d},\Integers)\rightarrow H^2(\C,\Integers)$
is an isomorphism. 

$\C$ is the projectivization of a rank $n$ vector bundle $F$ over $S$. 
Hence, $H^*(\C,\Integers)$ is the quotient of $H^*(S,\Integers)[x]$, with $x$ of degree $2$, by the 
ideal generated by $\sum_{i=0}^{n+1}c_i(F)x^i$. The image of $x$ in $H^*(\C,\Integers)$
corresponds to the class $\bar{x}:=c_1(\StructureSheaf{\C}(1))$ of Hodge type $(1,1)$.
In particular, $H^*(\C,\Integers)$ is a free $H^*(S,\Integers)$-module
of rank $n$ generated by $1$, $\bar{x}$, \dots, $\bar{x}^{n-1}$. 

\ref{lemma-item-Tate-Shafarevich-groups-of-S-and-C-are-isomorphic})
The vanishing of $H^3(S,\Integers)$ and $H^3(\C,\Integers)$ yields the commutative diagram 
with exact rows:
\[
\xymatrix{
0\ar[r] &
% NS(S) \ar[r] \ar[d]_{p_1^*} & 
H^2(S,\Integers)/NS(S) \ar[r] \ar[d]_{p_1^*} & 
H^2(S,\StructureSheaf{S}) \ar[r] \ar[d]_{p_1^*}^\cong & 
H^2(S,\StructureSheaf{S}^*) \ar[r] \ar[d]_{p_1^*} & 0
\\
0\ar[r] &
% NS(\C) \ar[r] & 
H^2(\C,\Integers)/NS(\C) \ar[r] & H^2(\C,\StructureSheaf{\C}) \ar[r] & 
H^2(\C,\StructureSheaf{\C}^*) \ar[r] & 0
}
\]
Part \ref{lemma-item-second-cohomology-of-K3-is-a-direct-summand} of the Lemma implies that 
the left and middle vertical homomorphism are isomorphisms. It follows that the right vertical homomorphism is an isomorphism as well.

The sheaf $R^2p_{2_*}\StructureSheaf{\C}^*$ vanishes, by the exactness 
of $R^2p_{2_*}\StructureSheaf{\C}\rightarrow R^2p_{2_*}\StructureSheaf{\C}^*\rightarrow R^3p_{2_*}\Integers$ and the 
vanishing of the left and right sheaves due to the fact that 
$p_2$ has one-dimensional fibers.
The sheaf $p_{2_*}\StructureSheaf{\C}^*$ is isomorphic to $\StructureSheaf{\linsys{\LB^d}}^*$,
since $p_2$ has connected complete fibers. Thus, $H^2(\C,\StructureSheaf{\C}^*)$
is isomorphic to the kernel of the differential
\begin{equation}
\label{eq-d-2-1-1}
d_2 : E_2^{1,1}:=H^1(\linsys{\LB^d},R^1p_{2_*}\StructureSheaf{\C}^*)\rightarrow 
E_2^{3,0}:=H^3(\linsys{\LB^d},\StructureSheaf{\linsys{\LB^d}}^*).
\end{equation}
We prove next that $d_2$ vanishes. 
The co-kernel of $d_2$ is equal to $F^3H^3(\C,\StructureSheaf{\C}^*)$.
Now $F^3H^3(\C,\StructureSheaf{\C}^*)$ is equal to the image of 
$p_2^*:H^3(\linsys{\LB^d},\StructureSheaf{\linsys{\LB^d}}^*)\rightarrow H^3(\C,\StructureSheaf{\C}^*)$.
We have a commutative diagram
\[
\xymatrix{
H^3(\C,\StructureSheaf{\C}^*) \ar[r]^-\cong &
H^4(\C,\Integers) 
\\
H^3(\linsys{\LB^d},\StructureSheaf{\linsys{\LB^d}}^*) \ar[r]^-\cong \ar[u]^-{p_2^*} &
H^4(\linsys{\LB^d},\Integers). \ar[u]^-{p_2^*}
}
\]
The horizontal homomorphisms, induced by 
the connecting homomorphism of the exponential sequence,
are isomorphisms, since $h^{0,3}(\C)=h^{0,3}(\linsys{\LB^d})=0$
and $h^{0,4}(\C)=h^{0,4}(\linsys{\LB^d})=0$. 
The right vertical homomorphism is injective. We conclude that the left
vertical homomorphism is injective. Hence the differential $d_2$ in
(\ref{eq-d-2-1-1})  vanishes and 
$H^2(\C,\StructureSheaf{\C}^*)$ is isomorphism to $H^1(\linsys{\LB^d},R^1p_{2_*}\StructureSheaf{\C}^*)$, 
yielding the isomorphism $b$.
\end{proof}

Let $\Sigma\subset H^2(S,\Integers)$ be the sub-lattice generated by classes of irreducible components
of divisors in the linear system $\linsys{\LB^d}$. Denote by $\Sigma^\perp$ the sub-lattice of 
$H^2(S,\Integers)$ orthogonal to $\Sigma$.

\begin{lem}
\label{lemma-Leray-filtration}
\begin{enumerate}
\item
\label{lemma-item-filtration-of-H-2}
The Leray filtration of $H^2(\C,\Integers)$ associated to $p_2$ is
identified as follows:
\begin{eqnarray*}
F^2H^2(\C,\Integers) & = & p_2^*H^2(\linsys{\LB^d},\Integers),
\\
F^1H^2(\C,\Integers) & = & p_2^*H^2(\linsys{\LB^d},\Integers) \ \oplus \ p_1^*\Sigma^\perp.
\end{eqnarray*}
\item
\label{lemma-item-E-p-q}
$E_2^{p,q}=E_\infty^{p,q}$, if $(p,q)=(2,0)$, or $(1,1)$. Consequently, we get the following isomorphisms.
\begin{eqnarray*}
E_2^{2,0}:=H^2(\linsys{\LB^d},p_{2_*}\Integers) & \cong & p_2^*H^2(\linsys{\LB^d},\Integers),
\\
E_2^{1,1}:=H^1(\linsys{\LB^d},R^1p_{2_*}\Integers) & \cong & p_1^*\Sigma^\perp,
%\\
%E_2^{0,2}:=H^0(\linsys{\LB^d},R^2p_{2_*}\Integers) & \cong & H^2(S,\Integers)/ p_1^*\Sigma^\perp\cong \Sigma^*.
\end{eqnarray*}
\item
\label{lemma-item-E-2-1}
If the sub-lattice $\Sigma$ is saturated in $H^2(S,\Integers)$, then $H^2(\linsys{\LB^d},R^1p_{2_*}\Integers)$
vanishes.
\end{enumerate}
\end{lem}

\begin{proof} (\ref{lemma-item-filtration-of-H-2}), (\ref{lemma-item-E-p-q})
The sheaf $p_{2_*}\Integers$ is the constant sheaf $\Integers$, since $p_2$
has connected fibers. 
Then $E^{3,0}_2=H^3(\linsys{\LB^d},\Integers)=0$, and so 
$E^{1,1}_\infty=E^{1,1}_2=H^1(\linsys{\LB^d},R^1p_{2_*}\Integers)$. 
$E^{2,0}_2:=H^2(\linsys{\LB^d},p_{2_*}\Integers)$ has rank $1$ and
it maps injectively into 
$H^2(\C,\Integers)$, with image  equal to $p_2^*H^2(\linsys{\LB^d},\Integers)$. 
Thus, $E^{2,0}_2=E^{2,0}_\infty$ and $E^{1,1}_\infty:=F^1H^2(\C,\Integers)/E^{2,0}_\infty$ is
isomorphic to $F^1H^2(\C,\Integers)/p_2^*H^2(\linsys{\LB^d},\Integers)$. 
Finally, $E^{0,2}_2$ is the kernel of
\[
d_2:H^0(\linsys{\LB^d},R^2p_{2_*}\Integers)\rightarrow H^2(\linsys{\LB^d},R^1p_{2_*}\Integers).
\]
%The equality $E^{0,2}_2=E^{0,2}_\infty$ holds, since 
%$E^{0,1}_2=H^1(\linsys{\LB^d},p_{2_*}\Integers)=H^1(\linsys{\LB^d},\Integers)$ vanishes.
Thus, $F^1H^2(\C,\Integers)$ is the kernel of the homomorphism
$H^2(\C,\Integers)\rightarrow H^0(\linsys{\LB^d},R^2p_{2_*}\Integers)$.
The latter kernel is equal to $p_1^*\Sigma^\perp\oplus p_2^*H^2(\linsys{\LB^d},\Integers)$, by 
Lemma \ref{lemma-Sha-is-one-dimensional} 
(\ref{lemma-item-second-cohomology-of-K3-is-a-direct-summand}).
We conclude that
$F^1H^2(\C,\Integers)/p_2^*H^2(\linsys{\LB^d},\Integers)$ is isomorphic to both
$H^1(\linsys{\LB^d},R^1p_{2_*}\Integers)$ and $p_1^*\Sigma^\perp$.

(\ref{lemma-item-E-2-1})
The composition $H^2(\C,\Integers)\rightarrow H^0(R^2p_{2_*}\Integers)\hookrightarrow \Sigma^*$
factors through $H^2(S,\Integers)$. 
If $\Sigma$ is saturated, then the composition is surjective, since $H^2(S,\Integers)$ is unimodular. 
Thus, $d_2^{0,2}:H^0(R^2p_{2_*}\Integers)\rightarrow H^2(R^1p_{2_*}\Integers)$ vanishes.
The sheaf $p_{2_*}\Integers$ is the trivial local system and the homomorphism 
$H^4(\linsys{\LB^d},p_{2_*}\Integers)\cong 
H^4(\linsys{\LB^d},\Integers)\rightarrow H^4(\C,\Integers)$ is the injective pull-back homomorphism $p_2^*$.
Thus the differential
$d_2^{2,1}:H^2(R^1p_{2_*}\Integers)\rightarrow \nolinebreak H^4(p_{2_*}\Integers)$ vanishes.
We conclude that $E_2^{2,1}:=H^2(R^1p_{2_*}\Integers)$ is isomorphic to $E_\infty^{2,1}$.
Now $E_\infty^{2,1}$ vanishes, since $H^3(\C,\Integers)$ vanishes.
\end{proof}

Let $\A^0$ be the kernel of the homomorphism $\deg$, given in (\ref{eq-exact-sequence-of-abelian-groups}). 
Then $\A^0$ is a subsheaf of $R^1p_{2_*}\StructureSheaf{\C}^*$
and we get the short exact sequences
\begin{eqnarray}
\label{eq-short-exact-with-kernel-A-0}
&0\longrightarrow \A^0\longrightarrow R^1p_{2_*}\StructureSheaf{\C}^*
\LongRightArrowOf{\deg} R^2p_{2_*}\Integers\longrightarrow 0,
\\
\label{eq-second-short-exact-sequence-involving-A-0}
&0\longrightarrow R^1p_{2_*}\Integers \longrightarrow 
R^1p_{2_*}\StructureSheaf{\C}\longrightarrow \A^0\longrightarrow 0,
\end{eqnarray}
and the long exact
\[
\cdots \rightarrow
H^1_{}(\linsys{\LB^d},R^1p_{2_*}\Integers)\rightarrow 
H^1_{}(\linsys{\LB^d},R^1p_{2_*}\StructureSheaf{\C})\RightArrowOf{} 
H^1_{}(\linsys{\LB^d},\A^0) \rightarrow \cdots
\]

\begin{lem}
\label{lemma-cohomology-of-A-0}
The group 
$H^0(\linsys{\LB^d},\A^0)$ is isomorphic to $NS(S)\cap \Sigma^\perp$. 
The composite homomorphism 
\[
H^2(S,\Integers)\rightarrow H^{0,2}(S)\RightArrowOf{p_1^*} H^{0,2}(\C) \cong \widetilde{\Sha}
\RightArrowOf{} H^1(\linsys{\LB^d},\A^0)
\]
factors through an injective homomorphism from $H^2(S,\Integers)/[\Sigma^\perp+NS(S)]$
into the kernel of the homomorphism
$H^1(\linsys{\LB^d},\A^0)\rightarrow \Sha$.
\end{lem}

\begin{proof}
The space $H^0(\linsys{\LB^d},R^1p_{2_*}\StructureSheaf{\C})$ vanishes, by Lemma
\ref{lemma-Sha-is-one-dimensional} (\ref{lemma-item-R-1-p-2-of-O-C-is-the-cotangent-bundle}). Hence,
$H^0(\linsys{\LB^d},\A^0)$ is the kernel of the homomorphism 
$H^1(\linsys{\LB^d},R^1p_{2_*}\Integers)\rightarrow \widetilde{\Sha}\cong H^{0,2}(S)$.
Compose the above homomorphism with the isomorphism $\Sigma^\perp \cong H^1(\linsys{\LB^d},R^1p_{2_*}\Integers)$
of Lemma \ref{lemma-Leray-filtration} in order to get the isomorphism
$H^0(\linsys{\LB^d},\A^0)\cong NS(S)\cap \Sigma^\perp$.

We have a commutative diagram with short exact rows
\begin{equation}
\label{eq-j}
\xymatrix{
0 \ar[r] & \frac{\Sigma^\perp}{NS(S)\cap \Sigma^\perp} \ar[r] \ar[d] &
\widetilde{\Sha} \ar[r]^-j \ar[d]_-\cong & 
\ker[H^1(\A^0)\rightarrow H^2(R^1p_{2_*}\Integers)] \ar[r] \ar[d] & 0
\\
0 \ar[r] &
\frac{H^2(S,\Integers)}{NS(S)} \ar[r] &
H^2(S,\StructureSheaf{S}) \ar[r] & 
H^2(S,\StructureSheaf{S}^*) \ar[r] & 0.
}
\end{equation}
The top row is obtained from the long exact sequence of sheaf cohomologies associated to the short exact sequence (\ref{eq-second-short-exact-sequence-involving-A-0}).
The left vertical homomorphism is injective and the right vertical homomorphism 
is surjective. The co-kernel of the former is isomorphic to the kernel of the latter and 
both are isomorphic to $H^2(S,\Integers)/[\Sigma^\perp+NS(S)]$. Setting 
\begin{equation}
\label{eq-Sha-0}
\Sha^0 \ \ := \ \ \ker[H^1(\A^0)\rightarrow H^2(R^1p_{2_*}\Integers)],
\end{equation}
we see that the right vertical homomorphism fits in the short exact sequence
\begin{equation}
\label{eq-kernel-of--homomorphism-from-Sha-0-to-Sha}
0\longrightarrow 
\frac{H^2(S,\Integers)}{\Sigma^\perp+NS(S)} \longrightarrow
\Sha^0\longrightarrow \Sha \longrightarrow 0.
\end{equation}
The statement of the Lemma follows.
\end{proof}

Let $\Sha^0$ be the group given in Equation (\ref{eq-Sha-0}).
%Set $\Sha^0:=\ker[H^1(\A^0)\rightarrow H^2(R^1p_{2_*}\Integers)].$
%Lemma \ref{lemma-cohomology-of-A-0} yields the short exact sequence
%where the map $j$ is induced by the exponential homomorphism.
Classes of $\Sha$ represent torsors for the relative Picard group scheme, while classes
of $\Sha^0$ represent torsors for the relative $\Pic^0$ group scheme. This comment will be illustrated in 
Example \ref{example-kernel-of-Sha-0-to-Sha} below.

%**********************************************************************************************
%
%**********************************************************************************************
\subsection{A universal family of Tate-Shafarevich twists}

Let $S$ be the marked $K3$ surface in Diagram (\ref{diagram-marked-M-H-u})
and $M_H(u)$ the moduli space of 
$H$-stable sheaves of pure one-dimensional support on $S$ in that Diagram.
Recall that $c_1(u)$ is the first Chern class of $\LB^d$, for a nef line-bundle $\LB$
on $S$, and 
%$\LB$ is the line bundle over $S$, such that the marking 
%$\eta:H^2(S,\Integers)\rightarrow \Lambda_{k3}$ maps
%$c_1(\LB^d)$ to the class $\bar{v}$ of $\Lambda_{k3}$.
the support map $\pi:M_H(u)\rightarrow \linsys{\LB^d}$ is a Lagrangian fibration.

Let $\sigma$ be a section of $R^1p_{2_*}(\StructureSheaf{\C}^*)$ 
over an open subset $U$ of $\linsys{\LB^d}$. Assume that $\sigma$ is the image of a section
$\tilde{\sigma}$ of $R^1p_{2_*}(\StructureSheaf{\C})$ over $U$.
Then $\sigma$ lifts to an automorphism of the open subset $\pi^{-1}(U)$ of $M_H(u)$.
This is seen as follows. Fix a point $t\in \linsys{\LB^d}$ and denote by $C_t$ the corresponding divisor in $S$.
Denote by $\sigma(t)$ the image of $\sigma$ in $H^1(C_t,\StructureSheaf{C_t}^*)$ 
and by $L_{\sigma(t)}$ the line-bundle over $C_t$ with class $\sigma(t)$. 
A sheaf $F$ over $C_t$ is $H$-stable, if and only if $F\otimes L_{\sigma(t)}$ is $H$-stable, 
since tensorization by  $L_{\sigma(t)}$ induces a one-to-one correspondence
between the set of subsheaves, which is slope-preserving, since $L_{\sigma(t)}$ belongs to 
the identity component of the Picard group of $C_t$. 

Let $s$ be an element of $\Sha^0$. We can choose 
a \v{C}ech $1$-co-cycle $\sigma:=\{\sigma_{ij}\}$ for the sheaf $\A^0$
%$R^1p_{2_*}\StructureSheaf{\C}^*$
representing $s$  in $\Sha^0$, with respect to an open covering
$\{U_i\}$ of $\linsys{\LB^d}$, such that each $\sigma_{ij}$ is the image of a section
$\tilde{\sigma}_{ij}$ of $R^1p_{2_*}(\StructureSheaf{\C})$, since the homomorphism 
$R^1p_{2_*}(\StructureSheaf{\C})\rightarrow \A^0$
is surjective. 
%by
%Lemma \ref{lemma-Sha-is-one-dimensional} (\ref{lemma-item-Tate-Shafarevich-groups-of-S-and-C-are-isomorphic}).
The co-cycle $\{\sigma_{ij}\}$
may be used to re-glue the open covering
$\pi^{-1}(U_i)$ of $M_H(u)$ to obtain a separated complex manifold $M_\sigma$ together with
a proper map
$
\pi_\sigma:M_\sigma\rightarrow \linsys{\LB^d}.
$
The latter is independent of the choice of the co-cycle, by the following Lemma,
so we denote it by 
\begin{equation}
\label{eq-pi-s}
\pi_s:M_s\rightarrow \linsys{\LB^d}.
\end{equation}
%The proof of the Lemma explains why we work with $\Sha^0$ instead of $\Sha$.

\begin{lem}
\label{lemma-M-s-is-independent-of-choice-of-co-cycle}
Let $\sigma:=\{\sigma_{ij}\}$ and $\sigma':=\{\sigma'_{ij}\}$ be two co-cycles representing the same class 
in $\Sha^0$. Then there exists an isomorphism $h:M_\sigma\rightarrow M_{\sigma'}$
satisfying $\pi_{\sigma'}\circ h=\pi_\sigma$. If the lattice $\Sigma$ of Lemma \ref{lemma-Leray-filtration}
has finite index in $NS(S)$, then 
$h$ depends canonically on $\sigma$ and $\sigma'$. 
\end{lem}

\begin{proof}
There exists a co-chain $h:=\{h_i\}$ in $C^0(\{U_i\},\A^0)$, such that $h_i\sigma_{ij}=\sigma'_{ij}h_j$,
possibly after refining the covering and restricting the co-cycles $\sigma$ and $\sigma'$ to the refinement.
Each $h_i$ is the image of a section $\tilde{h}_i$ of $R^1{p_{2_*}}\StructureSheaf{\C}$, possibly 
after further refinement of the covering, since the sheaf homomorphism 
$R^1{p_{2_*}}\StructureSheaf{\C}\rightarrow \A^0$ is surjective. 
Hence, $h_i$ lifts canonically to an automorphism of $\pi^{-1}(U_i)$. The 
co-chain $\{h_i\}$  of automorphisms glues to a global isomorphism 
from $M_{\sigma'}$ to $M_\sigma$, by the equality $h_i\sigma_{ij}=\sigma'_{ij}h_j$. 

If $h':=\{h'_i\}$ is another co-chain satisfying the equality $\delta(h)=\sigma(\sigma')^{-1}$, then 
$h^{-1}h'$ is a global section of $\A^0$. The assumption that $\Sigma$ has finite index in $NS(S)$
implies that $H^0(\A^0)$ vanishes, by Lemma \ref{lemma-cohomology-of-A-0}. Hence $h=h'$ and the
above refinements are not needed.
\end{proof}

In the relative setting 
the above construction gives rise to 
a natural proper family
\[
\tilde{\pi}:\M\rightarrow \widetilde{\Sha}\times \linsys{\LB^d},
\]
which restricts over $\{0\}\times \linsys{\LB^d}$
to $\pi:M_H(u)\rightarrow \linsys{\LB^d}$, and over
$\tilde{s}\in\widetilde{\Sha}$ to $\pi_{j(\tilde{s})}:M_{j(\tilde{s})}\rightarrow \linsys{\LB^d}.$ 
Indeed, 
let 
$(\{U_i\},\tilde{\sigma}_{ij})$ be a \v{C}ech  co-cycle representing a non-zero class
$\tilde{\sigma}$ in $H^1(\linsys{\LB^d},R^1p_*\StructureSheaf{\C})$.
Let 
\begin{equation}
\label{eq-function-rescating-sigma-to-tautological-class}
\tau:\widetilde{\Sha}\rightarrow \ComplexNumbers
\end{equation} 
be the 
function satisfying $\tau(x)\tilde{\sigma}=x$.
Then 
%$(\{U_i\},t\tilde{\sigma}_{ij})$ represents $t\sigma$, and 
$(\{\widetilde{\Sha}\times U_i\},\exp(\tau\tilde{\sigma}_{ij}))$
is a global co-cycle representing the desired family.
Let 
\[
f:\M\rightarrow \widetilde{\Sha} 
\]
be the composition of $\tilde{\pi}$ with the projection to 
$\widetilde{\Sha}$. 
%The marking $\eta_1$ in Diagram (\ref{diagram-marked-M-H-u}) 
%extends to a trivialization of the local system 
%$R^2f_*\Integers$ over the simply connected base $\widetilde{\Sha}$.

\begin{prop}
\label{prop-family-over-Sha-is-Kahler}
If the weight $2$ Hodge structure of $S$ is non-special, then $M_s$ is K\"{a}hler,
for all $s\in \Sha^0$.
\end{prop}

\begin{proof}
There is an open neighborhood of the origin in $\widetilde{\Sha}$,
over which the fibers of $f$ are K\"{a}hler, by the stability of K\"{a}hler manifolds
\cite[Theorem 9.3.3]{voisin-book-vol1}.
Let $j:\widetilde{\Sha}\rightarrow \Sha^0$ be the homomorphism given in Equation (\ref{eq-j}).
The kernel $\ker(j)$ is isomorphic to the group $[\Sigma^\perp+NS(S)]/NS(S)$, by Lemma 
\ref{lemma-cohomology-of-A-0}.
%\ref{lemma-Sha-is-one-dimensional} (\ref{lemma-item-Tate-Shafarevich-groups-of-S-and-C-are-isomorphic}).
As a subgroup of the base
$\widetilde{\Sha}$ of the family $f$, the kernel $\ker(j)$ acts on the base. Let $z$ be an element of $\ker(j)$ and 
$\tilde{s}$ an element of $\widetilde{\Sha}$. The fibers $M_{\tilde{s}}$ and $M_{\tilde{s}+z}$ of $f$
are both isomorphic to $M_{j(\tilde{s})}$. Let $V\subset \widetilde{\Sha}$ be the subset consisting 
of points over which the fiber of $f$ is K\"{a}hler. Then $V$ is an open and $\ker(j)$-invariant subset of 
$\widetilde{\Sha}$. 
%It thus suffices to prove that $\ker(j)$ is a dense subgroup of $\widetilde{\Sha}$.
Note that $\ker(j)$ is a finite index subgroup of $H^2(S,\Integers)/NS(S)$.
The kernel $\ker(j)$ is a dense subgroup of $\widetilde{\Sha}$, 
if and only if the image of $H^2(S,\Integers)/NS(S)$ is dense in
$H^{0,2}(S)$, by Lemma \ref{lemma-Sha-is-one-dimensional} 
(\ref{lemma-item-Tate-Shafarevich-groups-of-S-and-C-are-isomorphic}).
This is indeed the case, by the assumption that the
weight $2$ Hodge structure of $S$ is non-special, and
Lemmas \ref{lemma-density-in-a-generic-fiber} and  
\ref{lemma-density-of-Z-in-R-2}. The complement $V^c$ of $V$ in $\widetilde{\Sha}$ is $\ker(j)$ invariant.
If non-empty, then $V^c$ is dense and closed and so equal to $\widetilde{\Sha}$. But we know
that $V$ is non-empty. Hence, $V=\widetilde{\Sha}$.
%*********
% Hide
%*********
\hide{
%Denote by $p_2^*:H^{0,2}(S)\rightarrow \widetilde{\Sha}$ the composition
%\[
%H^{0,2}(S)\RightArrowOf{p_2^*}H^{0,2}(\C)\cong 
%H^1(\linsys{\LB^d},R^1p_{2_*}\StructureSheaf{\C}) = \widetilde{\Sha},
%\]
%which was shown to be an isomorphism in Lemma \ref{lemma-Sha-is-one-dimensional}.
Let $\Theta(S)\subset H^2(S,\Integers)$ be the transcendental sub-lattice.
Let $\phi:H^1(\linsys{\LB^d},R^1p_{2_*}\Integers)\rightarrow \widetilde{\Sha}$
be the homomorphism given in Equation (\ref{eq-j}), whose image is contained in $\ker(j)$. 
The image of $\Theta(S)$ in $H^{0,2}(S)$ is dense, by the assumption that the
weight $2$ Hodge structure of $S$ is non-special, and
Lemmas \ref{lemma-density-in-a-generic-fiber} and  
\ref{lemma-density-of-Z-in-R-2}.
The density of $\ker(j)$ would follow, once we construct a commutative diagram
\begin{equation}
\label{diagram-of-two-filtrations}
\xymatrix{
\Theta(S) \ar[rr] \ar[d]& & H^{0,2}(S) \ar[d]^{p_2^*}_{\cong}
\\
H^1(\linsys{\LB^d},R^1p_{2_*}\Integers) \ar[r]^{\hspace{3ex}\phi} & \widetilde{\Sha} \ar[r]^{\cong} & H^{0,2}(\C).
}
\end{equation}
Above we used the isomorphism 
$\widetilde{\Sha}=H^1(\linsys{\LB^d},R^1p_{2_*}\StructureSheaf{\C})\cong H^{0,2}(\C)$ 
constructed  in Lemma \ref{lemma-Sha-is-one-dimensional}.

A class in $\Theta(S)$ belongs to the kernel of the restriction from $H^2(S,\Integers)$ to the
second cohomology of an algebraic curve, and so the subspace $p_1^*\Theta(S)$ is contained in 
$F^1H^2(\C,\Integers)$. The left vertical homomorphism in Diagram (\ref{diagram-of-two-filtrations})
is defined to be the composition
of $p_1^*:\Theta(S)\rightarrow F^1H^2(\C,\Integers)$ with the quotient map
$F^1H^2(\C,\Integers)\rightarrow E^{1,1}_\infty\cong H^1(\linsys{\LB^d},R^1p_{2_*}\Integers)$.
The quotient homomorphism $F^1H^2(\C,\Integers)\rightarrow H^{0,2}(\C,\Integers)$
factors through the bottom horizontal homomorphism $\phi$, since $p_2^*H^2(\linsys{\LB^d},\Integers)$ 
restricts trivially to each fiber of $p_2$.
Diagram (\ref{diagram-of-two-filtrations}) is commutative, since 
$p_2^*:H^2(S,\Integers)\rightarrow H^2(\C,\Integers)$ is a morphism of Hodge structures.
%*********
% End Hide
%*********
}
\end{proof}

\begin{example}
\label{example-kernel-of-Sha-0-to-Sha}
Consider the case where  $d=1$ and 
$\Pic(S)$ is cyclic generated by the line bundle $\LB$ of degree $2n-2$, $n\geq 2$.
Then $H^2(\linsys{\LB^d},R^1p_{2_*}\Integers)$ vanishes, by
Lemma \ref{lemma-Leray-filtration} (\ref{lemma-item-E-2-1}), and 
$\Sha^0=H^1(\A^0)$.
The linear system $\linsys{\LB}$ consists of integral curves, and so
we can find an open covering $\{U_i\}$ of $\linsys{\LB}$, and sections $\zeta_i:U_i\rightarrow \C$,
such that $p_2\circ \zeta_i$ is the identity. Set $D_i:=\zeta_i(U_i)$. 
We get the line bundle
$\StructureSheaf{p_2^{-1}(U_i)}(D_i)$, which restricts to a line bundle 
of degree $1$ on fibers of $p_2$ over points of $U_i$. Let $h_i$ be the section 
of $R^1p_{2_*}\StructureSheaf{\C}^*$ over $U_i$ corresponding to $\StructureSheaf{p_2^{-1}(U_i)}(D_i)$
and denote by $h:=\{h_i\}$
the corresponding co-chain in $C^0(\{U_i\},R^1p_{2_*}\StructureSheaf{\C}^*)$.

Consider the Lagrangian fibrations 
$\pi_0:M_{\LB}(0,\LB,\chi)\rightarrow \linsys{\LB}$ and
$\pi_1:  M_{\LB}(0,\LB,\chi+1)\rightarrow \linsys{\LB}$, for some integer $\chi$.
The push-forward of
every rank $1$ torsion free sheaf on a curve in the linear system $\linsys{\LB}$ is an $\LB$-stable
sheaf on $S$, since the curve is integral. Hence,
the section $h_i$ induces an isomorphism
$h_i:\pi_0^{-1}(U_i)\rightarrow \pi_1^{-1}(U_i)$. 
The co-boundary $(\delta h)_{ij}:=h_jh_i^{-1}$ is a co-cycle in
$Z^1(\{U_i\},\A^0)$ representing a class $s\in \Sha^0$ mapping to the identity in $\Sha$.
The Lagrangian fibration $\pi_s:M_s\rightarrow \linsys{\LB}$, associated to the class $s$ in Equation (\ref{eq-pi-s})
with $u=(0,\LB,\chi)$, 
coincides with $\pi_1:M_{\LB}(0,\LB,\chi+1)\rightarrow \linsys{\LB},$ by the 
commutativity of the following diagram.
\[
\xymatrix{
\pi_0^{-1}(U_j) \ar[d]_{h_j} &
\pi_0^{-1}(U_{ij}) \ar[r]^-{h_i^{-1}h_j} \ar[l]_-{\supset} &
\pi_0^{-1}(U_{ij}) \ar[r]^-{\subset} &
\pi_0^{-1}(U_i) \ar[d]_{h_i}
\\
\pi_1^{-1}(U_j) &
\pi_1^{-1}(U_{ij}) \ar[r]^-{id} \ar[l]_-{\supset} &
\pi_1^{-1}(U_{ij}) \ar[r]^-{\subset} &
\pi_1^{-1}(U_i).
}
\]
The moduli spaces $M_{\LB}(0,\LB,\chi)$ and $M_{\LB}(0,\LB,\chi+1)$ are not isomorphic 
for generic $(S,\LB)$, since their weight $2$ Hodge structures are not Hodge isometric.

The kernel of $\Sha^0\rightarrow \Sha$ is cyclic of order $2n-2$, by the exactness of the sequence
(\ref{eq-kernel-of--homomorphism-from-Sha-0-to-Sha}).
The class $s$ constructed above generates the kernel. This is seen
as follows. The sheaf $R^2p_{2_*}\Integers$ is trivial, in our case, and the
homomorphism $\deg$, given in (\ref{eq-short-exact-with-kernel-A-0}), 
maps the $0$-co-chain $h$ to a global section of $R^2p_{2_*}\Integers$,
which generates $H^0(R^2p_{2_*}\Integers).$ Hence, $\delta h$ generates the
image of the connecting homomorphism 
$H^0(R^2p_{2_*}\Integers)\rightarrow H^1(\A^0)$ associated to the short 
exact sequence (\ref{eq-short-exact-with-kernel-A-0}). The latter image is precisely the kernel 
of $\Sha^0\rightarrow \Sha$.
\end{example}

%*********************************************************************************
% 
%*********************************************************************************
\subsection{The period map of the universal family is \'{e}tale} \hspace{1ex}\\
Denote by
$
T_{\pi_s} := \ker\left[
d\pi_s:TM_s\rightarrow \pi_s^*T\linsys{\LB^d}
\right]
$
the relative tangent sheaf of $\pi_s:M_s\rightarrow \nolinebreak\linsys{\LB^d}$.

\begin{lem}
\label{lemma-vertical-tangent-sheaf-is-pullback-of-cotangent-bundle}
The vertical tangent sheaf $T_{\pi_s}$ 
is isomorphic to $\pi_s^*T^*\linsys{\LB^d}$.
\end{lem}

\begin{proof}
Let $\mbox{sing}(\pi_s)$ be the support of the co-kernel of the differential
$d\pi_s:TM_s\rightarrow \pi_s^*T\linsys{\LB^d}$. We use 
Assumption \ref{assumption}  to prove that the
co-dimension of $\mbox{sing}(\pi_s)$ in $M_s$ is $\geq 2$. 
The generic fiber of $\pi_s$ is smooth, since $M_s$ is smooth. 
All fibers of $\pi_s$ have pure dimension $n$ 
\cite{matsushita-equidimensionality}.
Hence, the only way $\mbox{sing}(\pi_s)$ could contain a divisor is if
$\pi_s$ has fibers with a non-reduced irreducible component over some 
divisor in $\linsys{\LB^d}$.
The generic divisor in the linear system $\linsys{\LB^d}$
is a smooth curve, by Assumption \ref{assumption} (\ref{assumption-item-bpf}) and 
\cite[Prop. 1]{mayer}.
The fiber of $\pi_s$, over a reduced divisor $C\in \linsys{\LB^d}$, is isomorphic
to the compactified Picard of $C$, consisting of $\LB$-stable sheaves 
of Euler characteristic $\chi$ with pure one-dimensional
support $C$, which are the push forward of  rank $1$ torsion free sheaves over $C$. 
If $C$ is an integral curve, then the moduli space of rank $1$ torsion free sheaves over $C$
with a fixed Euler characteristic is irreducible and reduced \cite{aki}. 
If $C$ is reduced (possibly reducible) with at worst  ordinary double point singularities,
then the compactified Picard is reduced, by a result of Oda and Seshadri
\cite{oda-seshadri}. 
Assumption \ref{assumption} (\ref{assumption-item-locus}) thus implies that 
$\mbox{sing}(\pi_s)$ has co-dimension $\geq 2$ in $M_s$. 

%**************
% Hide
%**************
\hide{
Assume next that Assumption \ref{assumption} (\ref{assumption-Pic-S-is-cyclic}) holds. 
If $(n,d)\ne(5,2)$, then Assumption \ref{assumption} (\ref{assumption-item-locus}) follows as well,
by Remark \ref{remark-on-assumption}, and we are done. 
Assume that $(n,d)=(5,2)$ and $\Pic(S)$ is cyclic. The Euler characteristic $\chi$ is relatively prime to $d$. 
$M_\LB(0,\LB^2,\chi)$ and $M_\LB(0,\LB^2,-\chi)$ are naturally isomorphic, 
by \cite[Theorem 5.7]{le-potier}. Tensoring with $\LB$ we get an isomorphism between
$M_\LB(0,\LB^2,\chi)$ and $M_\LB(0,\LB^2,\chi+4)$. Hence, 
we may assume that $\chi=1$ and
$\pi_s:M_s\rightarrow \linsys{\LB^2}$ is a Tate-Shafarevich twist of
$\pi:M_\LB(0,\LB^2,1)\rightarrow \linsys{\LB^2}$.
In our case 
$\linsys{\LB^2}\cong\PP^5$, $\linsys{\LB}\cong\PP^2$, and the general 
point of the image of the Segre embedding $\linsys{\LB}\times\linsys{\LB}\hookrightarrow \linsys{\LB^2}$
corresponds to the union $C$ of two smooth genus $2$ curves $C_1$ and $C_2$ 
in the linear system $\linsys{\LB}$ meeting transversely at two points. 
The fiber of $\pi_s$ over such a point is the compactified Picard of $C$ of Euler 
characteristic $1$ with respect to the polarization $\LB$.  
According to Oda and Seshadri, 
the generic point of each component consists of a locally free sheaf over $C$, which restricts
to $C_1$ with degree $d_1$ and to $C_2$ with degree $d_2$,
such that $d_1+d_2=5$ and the stability condition $d_i-2<5/2$ holds for $i=1,2$ \cite{oda-seshadri}.
The compactified Picard thus has two connected components corresponding to degree vectors 
$(d_1,d_2)=(3,2)$ and $(2,3)$. The generic point of each component is smooth.
Hence, the differential $d\pi_s$ is submersive at the generic point of each of the two components.
Consequently,  $\mbox{sing}(\pi_s)$ has co-dimension $\geq 2$ in $M_s$.
%**************
% End Hide
%**************
}

Let $U$ be the complement of $\mbox{sing}(\pi_s)$ in $M_s$. 
The isomorphism $TM_s\rightarrow T^*M_s$, induced by a non-degenerate global holomorphic $2$-form,
maps the restriction of $T_{\pi_s}$ to $U$ isomorphically onto the restriction
of $\pi_s^*T^*\linsys{\LB^d}$. 
%Now $\det(T_{\pi_s})$ is isomorphic to $\pi_s^*\omega_{\linsys{\LB^d}}$. 
The isomorphism $TM_s\rightarrow T^*M_s$ must
map $T_{\pi_s}$ as a  subsheaf of the locally free $\pi_s^*T\linsys{\LB^d}$, by the fact that $\mbox{sing}(\pi_s)$
has codimension $\geq 2$.
But $T_{\pi_s}$ is a saturated subsheaf of $TM_s$. Hence, 
the image of $T_{\pi_s}$ is also saturated in $T^*M_s$, and is thus equal to $\pi_s^*T^*\linsys{\LB^d}$. 
\end{proof}

When the $K3$ surface $S$ is non-special, the fibers of the family $f$ are irreducible holomorphic symplectic
manifolds, by Proposition \ref{prop-family-over-Sha-is-Kahler} and the fact that K\"{a}hler deformations 
of an irreducible  holomorphic symplectic
manifold remain such \cite{beauville}.
Denote by 
\begin{equation}
\label{eq-trivialization-eta}
\eta:R^2f_*\Integers\rightarrow (\Lambda)_{\widetilde{\Sha}}
\end{equation}
 the trivialization,
which restricts to the the marking $\eta_1$ in Diagram
(\ref{diagram-marked-M-H-u}) over the point $0\in \widetilde{\Sha}$.
Let $P_f:\widetilde{\Sha}\rightarrow \Omega^+_{\alpha^\perp}$ be the period map of the family $f$ and the marking $\eta$.
Let $dP_f:T_{\tilde{s}}\widetilde{\Sha}\rightarrow H^{2,0}(M_s)^*\otimes H^{1,1}(M_s)$ be
the differential at $\tilde{s}$ of the period map. 

\begin{lem}
\label{lemma-the-differential-of-the-period-map-of-f}
The differential $dP_f$  is injective,  for all $\tilde{s}$ in $\widetilde{\Sha}$, and its image is equal to 
$H^{2,0}(M_s)^*\otimes\pi_s^*H^{1,1}(\linsys{\LB^d})$.
\end{lem}

\begin{proof}
Let $\psi:H^{2,0}(M_s)^*\otimes H^1(\linsys{\LB^d},T^*\linsys{\LB^d})\rightarrow H^1(M_s,T_{\pi_s})$
be the composition of 
\[
1\otimes\pi_s^*:H^{2,0}(M_s)^*\otimes H^1(\linsys{\LB^d},T^*\linsys{\LB^d})\rightarrow 
H^0(M_s,\Wedge{2}TM_s)\otimes H^1(M_s,\pi_s^*T^*\linsys{\LB^d})
\]
with the contraction homomorphism 
$H^0(M_s,\Wedge{2}TM_s)\otimes H^1(M_s,\pi_s^*T^*\linsys{\LB^d})\rightarrow H^1(M_s,T_{\pi_s})$.
Let $\kappa_{\tilde{s}}:T_{\tilde{s}}\widetilde{\Sha}\rightarrow H^1(M_s,TM_s)$ be the
Kodaira-Spencer map. We have the commutative diagram.
\[
\xymatrix{
%\widetilde{\Sha} \ar[r]^{\cong} \ar[d]_{\cong} & 
H^{2,0}(M_s)^*\otimes H^1(\linsys{\LB^d},T^*\linsys{\LB^d}) \ar[dr]^-{\psi} 
\ar@/^2pc/[rr]^{1\otimes \pi_s^*}& 
T_{\tilde{s}}\widetilde{\Sha} \ar[r]^-{dP_f} \ar[d]_-{\nu} \ar[dr]^{\kappa_{\tilde{s}}}&
H^{2,0}(M_s)^*\otimes H^{1,1}(M_s)
\\
&
H^1(M_s,T_{\pi_s}) \ar[r]^-{\gamma}&
H^1(M_s,TM_s). \ar[u]_{\cong}
}
\]
Above, 
%the left vertical isomorphism is the one constructed in Lemma \ref{lemma-Sha-is-one-dimensional}, 
the right vertical homomorphism is
induced by the sheaf homomorphism $TM_s\rightarrow T^*M_s$, associated to a holomorphic $2$-form, 
and $\gamma$ is induced by the inclusion of the relative
tangent sheaf $T_{\pi_s}$ as a subsheaf of $TM_s$.
The homomorphism $\nu$ is defined as follows.
A tangent vector $\xi$ at a class $\tilde{s}$ of 
$\widetilde{\Sha}$ is represented by a co-cycle of infinitesimal automorphisms - tangent vector fields -
which are vertical, being a limit of translations by local sections of the image of $R^1p_{2_*}\StructureSheaf{\C}$
in $R^1p_{2_*}\StructureSheaf{\C}^*$. So $\xi$ corresponds to 
an element $\nu(\xi)$ in $H^1(M_s,T_{\pi_s})$.

The top right triangle commutes, by Griffiths' identification of the differential of the period map \cite{cggh}.
The middle triangle commutes, by definition of the family $f$. The commutativity of the outer polygon is easily verified.
The top horizontal homomorphism $1\otimes \pi_s^*$ is injective, with image 
equal to the tangent line to the fiber of $q$. Hence, it suffices to prove that 
$\psi$ and $\nu$ have the same image in $H^1(M_s,T_{\pi_s})$. 
The latter statement would follow once we prove that $\nu$ is an isomorphism.

The homomorphism $\nu$ is induced by the pullback 
\[
\pi_s^*H^1(\linsys{\LB^d},R^1p_{2_*}\StructureSheaf{\C})\rightarrow H^1(M_s,\pi_s^*R^1p_{2_*}\StructureSheaf{\C}),
\]
followed by the homomorphism of sheaf cohomologies induced 
by an injective sheaf homomorphism
\[
\tilde{\nu} \ : \ \pi_s^*R^1p_{2_*}\StructureSheaf{\C}\rightarrow T_{\pi_s}.
\]
The domain of $\tilde{\nu}$ is isomorphic to $\pi_s^* T^*\linsys{\LB^d}$, by Lemma 
\ref{lemma-Sha-is-one-dimensional},
and its target is isomorphic to $\pi_s^* T^*\linsys{\LB^d}$, 
by Lemma \ref{lemma-vertical-tangent-sheaf-is-pullback-of-cotangent-bundle}.
Hence, $\tilde{\nu}$ is an isomorphism. It remains to prove that $H^1(M_s,\pi_s^*T^*\linsys{\LB^d})$
is one dimensional. We have the exact sequence
\[
0\rightarrow H^1(\linsys{\LB^d},\pi_{s_*}\pi_s^*T^*\linsys{\LB^d})\rightarrow
H^1(M_s,\pi_s^*T^*\linsys{\LB^d})\rightarrow 
H^0(\linsys{\LB^d},T^*\linsys{\LB^d}\otimes R^1\pi_{s_*}\StructureSheaf{M_s}).
\]
The left hand space is one-dimensional. It remains to prove that the right hand one vanishes.
It suffices to prove that $R^1\pi_{s_*}\StructureSheaf{M_s}$ is isomorphic to $T^*\linsys{\LB^d}$,
since $T^*\linsys{\LB^d}\otimes T^*\linsys{\LB^d}$ does not have any non-zero global sections. 

When $s=0$ and $M_0=M_H(u)$, then $M_0$ is projective and 
$R^1\pi_{0_*}\StructureSheaf{M_0}$ is isomorphic to $T^*\linsys{\LB^d}$,
by \cite[Theorem 1.3]{matsushita-higher-direct-images}.
Let us show that the sheaves $R^1\pi_{s_*}\StructureSheaf{M_s}$ are naturally 
isomorphic to $R^1\pi_{0_*}\StructureSheaf{M_0}$, 
for all $s$ in $\Sha$.
The fibrations $\pi_{s}$ agree, by definition,
over the open sets in a \v{C}ech covering of $\linsys{\LB^d}$, and the gluing transformations 
for the co-cycle representing the class $s$ do not change the induced
sheaf transition functions for the sheaves $R^1\pi_{s_*}\StructureSheaf{M_s}$, as we show next.
The gluing transformations glue locally free sheaves, so it suffices to prove that they agree with those of
$\pi_0$ over a dense open subset of $\linsys{\LB^d}$.
Indeed, if the fiber of $M_H(u)$ over $t\in \linsys{\LB^d}$ is a smooth and projective $\Pic^d(C_t)$, then 
an automorphism of an abelian variety $\Pic^d(C_t)$, acting by translation, acts trivially on the fiber
$H^1(\Pic^d(C_t),\StructureSheaf{\Pic^d(C_t)})$ of $R^1\pi_{*}\StructureSheaf{M_H(u)}$.
\end{proof}

%*******************
% Hide
%*******************
\hide{
\begin{lem}
\label{lemma-surjectivity}
Let $P:\ComplexNumbers\rightarrow \ComplexNumbers$ be a non-constant holomorphic map,
$G\subset (\ComplexNumbers,+)$ a dense subgroup in the analytic topology, and $\mu:G\rightarrow (\ComplexNumbers,+)$
a homomorphism, satisfying
\[
P(g+z)=\mu(g)+P(z),
\]
for all $z\in\ComplexNumbers$ and all $g$ in $G$. Then $P$ is surjective.
\end{lem}

\begin{proof}
Let $z_0\in\ComplexNumbers$ be a point where $dP_{z_0}$ does not vanish. 
Let $U\subset \ComplexNumbers$ be an open neighborhood of  $z_0$, such that $P$ maps $U$ bijectively 
onto $P(U)$. Let $\Sigma$ be a subset of $G$, such that the set
$\{z_0+g \ : \ g\in \Sigma\}$ is a dense subset of $U$. Then the subset 
$\{P(z_0)+\mu(g) \ : \ g\in \Sigma\}$ is dense in $P(U)$.
 We conclude that $\mu(G)$ is dense in  $\ComplexNumbers$.

Choose a rank $2$ subgroup $H$ of $G$, such that 
both $H$ and $\mu(H)$ span $\ComplexNumbers$ over $\RealNumbers$. 
 Then $P$ descends to a holomorphic  map 
 $\overline{P}:\ComplexNumbers/H\rightarrow \ComplexNumbers/\mu(H)$. 
 The domain of $\overline{P}$ is compact, its target is connected and Hausdorff, and $\overline{P}$
 is open.  Hence, $\overline{P}$ is surjective.
\end{proof}
%*******************
% End hide
%*******************
}

%**********************************************************************************
% 
%**********************************************************************************
\subsection{The Tate-Shafarevich line as the base of the universal family}
Let $q:\Omega_{\alpha^\perp}^+\rightarrow\Omega_{Q_\alpha}^+$ be
the morphism given in Equation (\ref{eq-q}).

\begin{thm}
\label{thm-Tate-Shafarevich-orbit}
Assume that the weight $2$ Hodge structure of $S$ is non-special and Assumption
\ref{assumption} holds. Then 
the period map $P_f$ of the family $f$ 
maps $\widetilde{\Sha}$ isomorphically 
onto the fiber of the morphism $q$  through the period of
$M_H(u)$.
\end{thm}

\begin{proof}
We already know that $P_f$ is non-constant, by Lemma \ref{lemma-the-differential-of-the-period-map-of-f}. 
The statement implies that $P_f$ is an affine linear isomorphism of one-dimensional complex
affine spaces. 
It suffices to prove the statement for a dense subset in moduli, since the condition of being affine linear 
is closed. We may thus assume that $\Pic(S)$ is cyclic generated by $\LB$. Then $H^0(\linsys{\LB^d},\A^0)$
is trivial, by Lemma \ref{lemma-cohomology-of-A-0}. 

Set $\Gamma:=c_1(\LB)^\perp$.
%$[c_1(\LB)^\perp+\Integers c_1(\LB)]/\Integers c_1(\LB)$. 
Note that $NS(S)=\Integers c_1(\LB)$ and $\Gamma$ has finite index in 
$H^2(S,\Integers)/NS(S)$. 
Let
\[
e:\Gamma\rightarrow \widetilde{\Sha}
\]
be the composition of the projection $\Gamma\rightarrow H^{0,2}(S)$ with the isomorphisms
$H^{0,2}(S)\cong H^{0,2}(\C)\cong \widetilde{\Sha}$ of Lemma 
\ref{lemma-Sha-is-one-dimensional}. Then $e$ is injective and its image is dense in $\widetilde{\Sha}$, by Lemma 
\ref{lemma-density-in-a-generic-fiber}. 

Given an element $x\in \widetilde{\Sha}$, we get a marked pair $(M_x,\eta_x)$, as above. 
$M_H(u)$ will be denoted  by $M_0$, it being the fiber of $f$ over the origin in $\widetilde{\Sha}$.
We associate next  to an element $\gamma\in \Gamma$ a canonical isomorphism 
\[
h_\gamma: M_0\rightarrow M_{e(\gamma)}.
\]
Let $\tau:\widetilde{\Sha}\rightarrow \ComplexNumbers$ be the function 
given in (\ref{eq-function-rescating-sigma-to-tautological-class}), which was  
used in the construction of the family $f$.
Let $\tilde{\sigma}:=\{\tilde{\sigma}_{ij}\}$ be the co-cycle used in that construction.
Let $a$ be the $1$-co-cycle given by 
$a_{ij}:=\exp(\tau(e(\gamma))\tilde{\sigma}_{ij}))$. Then $M_{e(\gamma)}$
is the Tate-Shafarevich twist of $M_0$ with respect to the co-cycle $a$.
The $1$-co-cycle $a$ is a co-boundary in $Z^1(\{U_i\},\A^0)$, 
%possibly after replacing the open covering by a refinement and restricting $a$ to the refinement,
by Lemma
\ref{lemma-cohomology-of-A-0} and the definition of $\Gamma$.
Thus, there exists a $0$-co-chain $h:=\{h_i\}$ in $C^0(\{U_i\},\A^0)$, satisfying 
$\delta h=a$. The co-chain $h$ is unique, since $H^0(\A^0)$ is trivial, by our assumption on $S$.
The co-chain $h$ determines the isomorphism $h_\gamma:M_0\rightarrow M_{e(\gamma)}$
(Lemma \ref{lemma-M-s-is-independent-of-choice-of-co-cycle}).

We define next a monodromy representation associated to  the family $f$.
Denote by $h_{\gamma_*}:H^2(M_0,\Integers)\rightarrow H^2(M_{e(\gamma)},\Integers)$
the isomorphism induced by $h_\gamma$.
Let
\[
\mu \ : \ \Gamma \rightarrow Mon^2(M_0)
\]
be given by the composition $\mu_\gamma:=\eta_0^{-1}\circ\eta_{e(\gamma)}\circ h_{\gamma_*}$
of the parallel-transport operator 
$\eta_0^{-1}\circ\eta_{e(\gamma)}$ and the isomorphism $h_{\gamma_*}$.

\begin{claim}
\label{claim-mu-is-a-group-homomorphism}
The map $\mu$ is a group homomorphism.
\end{claim}

\begin{proof}
Let $\gamma_1$, $\gamma_2$ be elements of $\Gamma$ and set $\gamma_3:=\gamma_1+\gamma_2$.
Let the topological space  $B$ be the quotient of $\widetilde{\Sha}$ obtained by identifying
the four points $0$, $e(\gamma_1)$, $e(\gamma_2)$, $e(\gamma_3)$. 
The family $f$ descends to a family $\bar{f}:\overline{\M}\rightarrow B$ by 
identifying the  fiber $M_{e(\gamma_i)}$ with $M_0$ via the isomorphisms
$h_{\gamma_i}$, $1\leq i \leq 3$. Then $\mu_{\gamma_i}$ is the monodromy operator corresponding to any
loop in $B$, which is the image of some continuous path from $0$ to $e(\gamma_i)$ in 
$\widetilde{\Sha}$. Let $\bar{0}\in B$ be the image of $0\in \widetilde{\Sha}$. 
The statement now follows from the fact that the monodromy representation of $\pi_1(B,\bar{0})$ 
in $H^2(M_0,\Integers)$ is a group homomorphism.
\end{proof}

The image of $\widetilde{\Sha}$ via
the period map is contained in the fiber of $q$, since the differential of the morphism
$q\circ P_f$ vanishes, by Lemma 
\ref{lemma-the-differential-of-the-period-map-of-f}.
It follows that the variation of Hodge structures of the local system  $R^2f_*\Integers$
over $\widetilde{\Sha}$ is the pullback of the one over the fiber of $q$ via the period map $P_f$. 
Let $\eta$ be the trivialization of $R^2f_*\Integers$ given in Equation (\ref{eq-trivialization-eta}).
Given a point $x\in\widetilde{\Sha}$, 
set $\alpha_{x}:=\eta_{x}^{-1}(\alpha)$. Then 
$\alpha_{x}=\pi_{x}^*(c_1(\StructureSheaf{\linsys{\LB^d}}(1)))$ and 
the sub-quotient variation of Hodge structures
$\alpha_{x}^\perp/\Integers\alpha_{x}$ is trivial. 

%The composite 
%homomorphism $H^{2,0}(M_x)\rightarrow (\alpha_x^\perp)_\ComplexNumbers\rightarrow 
%(\alpha_x^\perp/\Integers\alpha_x)_\ComplexNumbers = c_1(\LB)^\perp\subset H^2(S,\ComplexNumbers)$
%is injective and its image is $H^{2,0}(S)$, for every $x\in \widetilde{\Sha}$.
%Fix a non-zero class $w$ in $H^{2,0}(S)$. 
%Let $w_x$ be the class in $H^{2,0}(M_x)$ corresponding to $w$ via the above homomorphism.
%Note that $w_x$ is the $(2,0)$ part of the flat deformation of $w_0$ in the local system
%$R^2f_*\ComplexNumbers$.
%
%We will need the following more geometric description of $w_x$.
The vertical tangent sheaf $T_{\pi_x}$ is naturally isomorphic to $T_{\pi_0}$,
as we saw in the last paragraph of the proof of Lemma 
\ref{lemma-the-differential-of-the-period-map-of-f}. 
The $2$-form $w_x$ induces an isomorphism $\pi_{x_*}T_{\pi_x}\RightArrowOf{w_x} T^*\linsys{\LB^d}$,
by Lemma \ref{lemma-vertical-tangent-sheaf-is-pullback-of-cotangent-bundle}.
We get the composite isomorphism 
$\pi_{0_*}T_{\pi_0}\cong \pi_{x_*}T_{\pi_x}\RightArrowOf{w_x} T^*\linsys{\LB^d}.$
Let $w_x$ be the unique holomorphic $2$-form, for which the composite
isomorphism is equal to $\pi_{0_*}T_{\pi_0}\RightArrowOf{w_0} T^*\linsys{\LB^d}.$
Such a form $w_x$ exists, since the endomorphism algebra of $T^*\linsys{\LB^d}$
is one dimensional.

We show next 
that the class of $w_x$ is the $(2,0)$ part of the flat deformation of the class of $w_0$ in the local system
$R^2f_*\ComplexNumbers$. It suffices to prove the local version of that statement.
Let $x_0$ be a point of $\widetilde{\Sha}$. 
There is a differentiable trivialization of $f:\M\rightarrow \widetilde{\Sha}$, over an open analytic 
neighborhood $U$ of $x_0$, and a $C^\infty$ family of complex structures $J_x$, $x\in U$, 
such that $(M_{x_0},J_x)$ is biholomorphic to $M_x$. Furthermore, 
the complex structures $J_{x_0}$ and $J_x$ restrict to the same complex structure
on each fiber of $\pi_{x_0}$ and $\pi_{x_0}$ is holomorphic with respect to both. 
%Let $w_x$ be the $(2,0)$ part of the two form $w_{x_0}$
%of $H^0\left(M_{x_0},\left[
%\Wedge{2}T^*_{\RealNumbers}M_{x_0}\right]\otimes_{\RealNumbers}\ComplexNumbers\right)$
%with respect to the complex structure $J_x$.
Both complex structures induce the same complex structure on
$\Hom\left(T_{\pi_{x_0}},
\pi_{x_0}^*T^*_{\RealNumbers}\linsys{\LB^d}\right)$ and the two forms 
$w_{x_0}$ and $w_x$ induce the same section in the complexification of that bundle. 
Hence, the difference $w_{x_0}-w_x$ is a closed $2$-form in 
$\pi_{x_0}^*\Wedge{2}T^*_{\RealNumbers}\linsys{\LB^d}\otimes_\RealNumbers\ComplexNumbers$.
Being closed, the latter $2$-form must be the pull-back of a closed $2$-form $\theta$ 
on $\linsys{\LB^d}$,
since fibers of $\pi_{x_0}$ are connected. 
Now the cohomology class of $\pi_{x_0}^*\theta$ 
is of type $(1,1)$ with respect to all complex structures, since 
$H^{1,1}(\linsys{\LB^d})=H^2(\linsys{\LB^d},\ComplexNumbers)$. 
Hence, the class of $w_x$ is the $(2,0)$ part of 
the class of $w_{x_0}$ with respect to the complex structure $J_x$.
%But the latter $(2,0)$ summand is the same for all $J_x$, 
%$x\in U$, since the complex structures 
%on $T^*_{\pi_{x_0}}$ and on $\pi_{x_0}^*T^*\linsys{\LB^d}$ are constant.

There exists a constant $c_x\in \ComplexNumbers$, such that 
the equality
\[
\eta_x(w_x)=\eta_0(w_0)+c_x\alpha
\]
holds in $\Lambda_\ComplexNumbers$, by the characterization of $\omega_x$ in the above paragraph. 
The function $c:\widetilde{\Sha}\rightarrow\ComplexNumbers$ defined above is
equivalent to the period map $P_f$ and is thus holomorphic and its derivative is no-where vanishing,
by Lemma \ref{lemma-the-differential-of-the-period-map-of-f}.
If $x=e(\gamma)$, 
we get $\eta_0^{-1}\eta_{e(\gamma)}(w_{e(\gamma)})=w_0+c_{e(\gamma)}\alpha_0$,
Now 
$h_\gamma(w_0)=w_{e(\gamma)}$, by definition of $w_x$, $x\in \widetilde{\Sha}$,
and the construction of $h_\gamma$.
We get the equality
\begin{equation}
\label{eq-mu-gamma-is-a-translation}
\mu_\gamma(w_0)=w_0+c_{e(\gamma)}\alpha_0.
\end{equation}

The composition $c\circ e:\Gamma\rightarrow \ComplexNumbers$ is a 
group homomorphism,
\[
c(e(\gamma_1)+e(\gamma_2)) = c(e(\gamma_1))+c(e(\gamma_2)),
\]
by Equation (\ref{eq-mu-gamma-is-a-translation}) and 
Claim \ref{claim-mu-is-a-group-homomorphism}.
The image $e(\Gamma)$ is dense in $\widetilde{\Sha}$ and so 
$e(\Gamma)\times e(\Gamma)$ is dense in $\widetilde{\Sha}\times \widetilde{\Sha}$.
We conclude that $c$ is a group homomorphism, 
$c(x_1+x_2)=c(x_1)+c(x_2)$, for all $(x_1,x_2)\in\widetilde{\Sha}\times \widetilde{\Sha}$.
Continuity of $c$ implies that it is a linear transformation of real vector spaces. Indeed,
given $x_1$, $x_2$ in $\widetilde{\Sha}$, 
$c(ax_1+bx_2)=ac(x_1)+bc(x_2)$, for all $a, b\in \Integers$, hence also for all $a, b \in \RationalNumbers$,
and continuity implies that the equality holds also for all $a, b \in \RealNumbers$. The map $c$ is holomorphic, hence it is 
a linear transformation of one-dimensional complex vector spaces, which is an isomorphism, since $c$ is non-constant. 
This completes the proof of Theorem \ref{thm-Tate-Shafarevich-orbit}.
\end{proof}

Let $X$ be an irreducible holomorphic symplectic manifold of $K3^{[n]}$-type and $\pi:X\rightarrow \PP^n$
a Lagrangian fibration. Set $\alpha:=\pi^*c_1(\StructureSheaf{\PP^n}(1))$. 
Let $d$ be the divisibility of $(\alpha,\bullet)$. Let $(S,\LB)$ be the semi-polarized $K3$ surface associated to $(X,\alpha)$
in Diagram  (\ref{diagram-marked-M-H-u}) and $\chi$ the Euler characteristic of the Mukai vector $u$
in that diagram. Choose a $u$-generic polarization $H$ on $S$.

\begin{thm}
\label{thm-X-is-birational-to-a-Tate-Shafarevich-twist}
Assume that $X$ is non-special and $(S,\LB)$ satisfies Assumption \ref{assumption}. 
Then 
$X$ is bimeromorphic to a Tate-Shafarevich twist of the Lagrangian fibration
$M_H(0,\LB^d,\chi)\rightarrow \nolinebreak \linsys{\LB^d}$.
\end{thm}

\begin{proof}
Fix a marking $\eta:H^2(X,\Integers)\rightarrow \Lambda$. 
Starting with the period of $(X,\eta)$, 
Theorem \ref{thm-Tate-Shafarevich-orbit} exhibits a marked triple
$(X',\alpha',\eta')$, with  $\eta'(\alpha')=\eta(\alpha)$,
in the same connected component $\FM_{\eta(\alpha)^\perp}^+$ 
as the triple
$(X,\alpha,\eta)$, such that the class $\alpha'$ is semi-ample as well
%and .
and the periods $P(X,\eta)$ and $P(X',\eta')$ are equal. 
Furthermore, the Lagrangian fibration $\pi':X'\rightarrow \linsys{\LB^d}$
induced by $\alpha'$ is a Tate-Shafarevich twist of $\pi_0:M_H(0,\LB^d,\chi)\rightarrow \linsys{\LB^d}$.
Step 1 of the proof of Theorem
\ref{thm-main}
yields a bimeromorphic map $f:X\dashedrightarrow X'$,
which is shown in Step 2 of that proof to satisfy $f^*(\alpha')=\alpha$
(see Equation (\ref{eq-alpha-X-is-pullback-of-semiample-class})).
%relating the Lagrangian fibrations of $X$ and $X'$ induced by $\alpha$ and $\alpha'$.
\end{proof}

\begin{proof} (Of Theorem \ref{thm-Tate-Shafarevich-orbit-in-introduction})
The condition that $NS(X)\cap \alpha^\perp$ is cyclic generated by $\alpha$ 
implies that the semi-polarized $K3$ surface $(S,\LB)$, associated to $(X,\alpha)$, 
has a cyclic Picard group generated by $\LB$. Assumption
\ref{assumption} thus holds, by Remark \ref{remark-on-assumption}.
Theorem \ref{thm-Tate-Shafarevich-orbit-in-introduction} thus follows from
Theorem \ref{thm-X-is-birational-to-a-Tate-Shafarevich-twist}.
\end{proof}

%*******************
% Hide
%*******************
\hide{
Denote the transcendental lattice of $S$ by $\Theta(S)$.
Let $\rho:\Theta(S)\rightarrow \widetilde{\Sha}$ be the composition of the inclusion
of $\Theta(S)\subset H^2(S,\Integers)$, the projection $H^2(S,\Integers)\rightarrow H^{0,2}(S)$,
and the isomorphism $H^{0,2}(S)\cong \widetilde{\Sha}$ of 
Lemma \ref{lemma-Sha-is-one-dimensional}.
%Diagram (\ref{diagram-of-two-filtrations}). 
Let $\tilde{\mu}:\Theta(S)\rightarrow Mon^2(M_H(v),\Integers)$ be the monodromy representation of the family $f$
(??? define ???).
Then $\eta_1^{-1}(\alpha)$ is invariant under the image of $\tilde{\mu}$ and the induced action
on $\eta_1^{-1}(\alpha)^\perp/\Integers\eta_1^{-1}(\alpha)$ is trivial (??? why).
Hence, the image of $\tilde{\mu}$ belongs to the kernel of the
homomorphism $h$ given in (\ref{eq-h}). Thus, the action factors, via a homomorphism 
$\Theta(S)\rightarrow Q_\alpha:=\alpha^\perp/\Integers\alpha$,
through the homomorphism 
$g:Q_\alpha \rightarrow Mon^2(\Lambda,\iota)$, by
Lemma \ref{lemma-q-is-Q-alpha-invariant} part (\ref{lemma-item-ker-h-equals-im-g}).
Let $\mu:\Theta(S)\rightarrow \Aut(q^{-1}(\uell))$
be the homomorphism induced by $\tilde{\mu}$. 
The image of $\mu$ acts on $q^{-1}(\uell)$ by translations, as shown in Equation
(\ref{eq-homomorphism-from-Q-alpha-to-C}). The subgroup
$G:=\rho(\Theta(S))$ is dense in $\widetilde{\Sha}$, since $S$ is non-special. 
Hence, the period map $P_f:\widetilde{\Sha}\rightarrow q^{-1}(\uell)$ is surjective, by
Lemma \ref{lemma-surjectivity}. $P_f$ is \'{e}tale, by Lemma \ref{lemma-the-differential-of-the-period-map-of-f}.
We conclude that $P_f$ is an isomorphism, since $q^{-1}(\uell)$ is simply connected.
%*******************
% End hide
%*******************
}

%*******************
% Hide
%*******************
\hide{
Strategy:
Fix $\tilde{s}\in \widetilde{\Sha}$ and 
let $\alpha_{\tilde{s}}\in H^2(M_{\tilde{s}},\Integers)$ be $\pi_{\tilde{s}}^*c_1(\StructureSheaf{\linsys{\LB^d}}(1))$.
Show that the Hodge structure of $Q_{\tilde{s}}:=\alpha_{\tilde{s}}^\perp/\Integers\alpha_{\tilde{s}}$
is constant, hence Hodge isometric to that of the orthogonal complement $c_1(\LB)^\perp$
to $c_1(\LB)$  in $H^2(S,\Integers)$. 
%We have seen that $c_1(\LB)^\perp$ is isomorphic to $H^1(\linsys{\LB^d},R^1p_{2_*}\Integers)$ in
%the proof of Proposition \ref{prop-family-over-Sha-is-Kahler}.
%The idea is to show that one can recover the complex $\A^\bullet$ of sheaves over $\linsys{\LB^d}$
%\[
%R^1p_{2_*}\StructureSheaf{\C}\rightarrow \A^0
%\]
%from each $\pi_{s}:M_s\rightarrow \linsys{\LB^d}$. 
%Note that $R^1p_{2_*}\StructureSheaf{\C}$ was shown to be isomorphic to $T^*\linsys{\LB^d}$.
%The complex $\A^\bullet$ is quasi-isomorphic to $R^1p_{2_*}\Integers$ and the natural morphism 
%$R^1p_{2_*}\Integers\rightarrow R^1p_{2_*}\StructureSheaf{\C}$
%induces the homomorphism $Q_{\tilde{s}}\rightarrow Q_{\tilde{s}}^{0,2}$
%after taking the first sheaf cohomology, so it determines the period as the constant period of
%$c_1(\LB)^\perp$. 

Consider the short exact sequence 
$0\rightarrow \Integers\rightarrow \StructureSheaf{M_s}\rightarrow \StructureSheaf{M_s}^*\rightarrow 0$
and the associated long exact sequence of sheaves of abelian groups over $\linsys{\LB^d}$
\[
0\rightarrow R^1\pi_{s_*}\Integers\rightarrow 
R^1\pi_{s_*}\StructureSheaf{M_s}\rightarrow R^1\pi_{s_*}\StructureSheaf{M_s}^*
\RightArrowOf{\delta_s} R^2\pi_{s_*}\Integers \rightarrow \cdots
\]
Set $\B_s^0:=\ker(\delta_s)$. Again we get a short exact sequence
\[
0\ \rightarrow R^1\pi_{s_*}\Integers\rightarrow 
R^1\pi_{s_*}\StructureSheaf{M_s}\rightarrow \B^0_s\rightarrow 0.
\]
Denote by $\B_s^\bullet$ the complex $R^1\pi_{s_*}\StructureSheaf{M_s}\rightarrow \B^0_s$
with the morphism from the above short exact sequence. Then 
$\B_s^\bullet$ is quasi-isomorphic to $R^1\pi_{s_*}\Integers$. 
The complex $\B_s^\bullet$ is naturally isomorphic to the complex
$\B_0^\bullet$ associated to
$\pi:\M_H(u)\rightarrow \linsys{\LB^d}$, since the fibrations $\pi_{\tilde{s}}$ agree, by definition,
over the open sets in the \v{C}ech covering of $\linsys{\LB^d}$, and the gluing transformations 
for the co-cycle representing the class $j(\tilde{s})$ do not change the induced
sheaf transition functions for the sheaves $R^1\pi_{s_*}\StructureSheaf{M_s}$ and $R^1\pi_{s_*}\StructureSheaf{M_s}^*$.
For example, if the fiber of $M_H(u)$ over $t\in \linsys{\LB^d}$ is $\Pic^d(C_t)$, then 
an automorphism of an abelian variety $\Pic^d(C_t)$, acting by translation, acts trivially on the fibers
$H^1(\Pic^d(C_t),\StructureSheaf{\Pic^d(C_t)})$ and 
$\Pic[\Pic^d(C_t)]$ of $R^1\pi_{*}\StructureSheaf{M_H(u)}$ and $R^1\pi_{*}\StructureSheaf{M_H(u)}^*$.

Let $F^pH^2(M_s,\Integers)$ be the filtration induced by the Leray spectral sequence 
associated to the morphism $\pi_s:M_s\rightarrow \linsys{\LB^d}$.
Let $E_\infty^{p,q}$ be the graded summands, and set $E_2^{p,q}:=H^p(\linsys{\LB^d},R^q\pi_{s_*}\Integers)$.
We claim (??? prove it ???) that we have the isomorphisms
\[
E^{1,1}_\infty \cong E_2^{1,1} := H^1(\linsys{\LB^d},R^1\pi_{s_*}\Integers) \cong \alpha_s^\perp/\Integers\alpha_s.
\]
Set $Q_s:=\alpha_s^\perp/\Integers\alpha_s$ and endow it with the induced Hodge structure.
$R^1\pi_{s_*}\StructureSheaf{M_s}$ is isomorphic to 
$R^1\pi_*\StructureSheaf{M_H(u)}$, which is isomorphic to $T^*\linsys{\LB^d}$, and
the homomorphism 
\begin{equation}
\label{eq-period-of-Q-s}
H^1_{}(\linsys{\LB^d},\B^\bullet_s) \rightarrow H^1_{}(\linsys{\LB^d},R^1\pi_*\StructureSheaf{M_s})
\end{equation}
coincides with the homomorphism $Q_s\rightarrow Q_s^{0,2}$, which 
determines the period of $Q_s$. The homomorphism (\ref{eq-period-of-Q-s})
is determined by the complex $\B^\bullet_s$, which is constant with respect to $s$.
Hence, the Hodge structure of $Q_s$ is constant.

%If the sections of the co-cycle representing $j(\tilde{s})$ are not Lagrangian, 
%then translation by these sections does not preserve the holomorphic symplectic structure.
%The difference between the original and its translate is a pull-back of a $2$-form on the open subset
%on $\linsys{\LB^d}$.

Next show that the period map of $R^2f_*\Integers$ is non-constant and 
$\Theta(S)$-equivariant. Conclude that it maps $\widetilde{\Sha}$ onto the fiber of $q$.
The tangent bundle to the fiber of $q$ at a point $\ell$ over $\uell$ is 
$T_\ell\PP[\ell\oplus \ComplexNumbers\alpha]$, which is isomorphic to 
$\Hom(\ell,\ComplexNumbers\alpha)$.
Now the lines $\uell:=H^{2,0}(S)$ and $\ell$ are naturally isomorphic, as are
$\ComplexNumbers\alpha$ and $H^{1,1}(\linsys{\LB^d})$. Lemma
\ref{lemma-Sha-is-one-dimensional} part 
(\ref{lemma-item-Sha-is-one-dimensional}) identified 
$T_\ell({q^{-1}(\uell)})$ and $\widetilde{\Sha}=T_{0}\widetilde{\Sha}$.
We claim that this identifies the differential of the period map for the family $f$ when $s=0$ and $M_s=M_H(u)$. 

We work out the identification of the differential of the period map at a general point 
of $\widetilde{\Sha}$.
Set $s:=j(\tilde{s})$.
According to Griffiths' identification, the differential of the period map is the composition of
the Kodaira Spencer map $\kappa_{\tilde{s}}:T_{\tilde{s}}\widetilde{\Sha}\rightarrow H^1(M_s,TM_s)$ with
the natural map
\[
H^1(M_s,TM_s)\rightarrow \Hom(H^{2,0}(M_s),H^{1,1}(M_s)).
\]
We identify $\kappa_{\tilde{s}}$ first.
A tangent vector $\xi$ at a class $\tilde{s}$ of 
$\widetilde{\Sha}$ corresponds to a co-cycle of infinitesimal automorphisms - a tangent vector field -
which is vertical, being a limit of translations by local sections of $\A^0$. So $\xi$ corresponds to 
an element $\kappa_{\tilde{s}}(\xi)$ in the image of 
$H^1(M_s,T_{\pi_s})$ in $H^1(M_s,TM_s)$, where 
\[
T_{\pi_s} := \ker\left[
d\pi_s:TM_s\rightarrow \pi_s^*T\linsys{\LB^d}
\right].
\]
The contraction with holomorphic two-forms yields the sheaf homomorphism
\[
T_{\pi_s}\otimes_\ComplexNumbers H^{2,0}(M_s)\rightarrow \pi^*T^*\linsys{\LB^d}.
\]
Taking first \v{C}ech cohomology, we get
\[
H^1(M_s,T_{\pi_s})\otimes H^{2,0}(M_s)\rightarrow H^1(M_s,\pi_s^*T^*\linsys{\LB^d})
\cong H^1(\linsys{\LB^d},T^*\linsys{\LB^d}). 
\]
The last isomorphism follows from the vanishing of 
$H^0(\linsys{\LB^d},R^1\pi_{s_*}\pi_s^*T^*\linsys{\LB^d})$, since
$R^1\pi_{s_*}\pi_s^*T^*\linsys{\LB^d}$ is isomorphic to $T^*\linsys{\LB^d}\otimes T^*\linsys{\LB^d}$.
%*******************
% End hide
%*******************
}
%***************************************************************************
%
%***************************************************************************

\end{document}